\documentclass[11pt]{article}
\usepackage[toc,page]{appendix}

\usepackage{etoolbox}

\usepackage{enumitem}
\usepackage{amsfonts,amsmath, amssymb,latexsym}
\usepackage{mathtools}
\usepackage{multirow,comment}
\def\Z{{\mathbb Z}}

\def\cI{{\mathcal I}}

\def\SL{{\rm SL}}

\def\supp{{\rm supp}}
\def\mon{{\rm mon}}

\def\GL{{\rm GL}}

\def\SO{{\rm SO}}

\def\PGL{{\rm PGL}}

\def\sk{{\rm sk}}

\def\inv{{\rm inv}}

\def\diag{{\rm diag}}
\def\gen{{\rm gen}}

\def\rad{{\rm rad}}
\def\Stab{{\rm Stab}}

\def\Sel{{\rm Sel}}
\def\Inv{{\rm Inv}}
\def\Sym{{\rm Sym}}

\def\Cl{{\rm Cl}}

\def\O{{\mathcal O}}
\newcommand{\cO}{{\mathcal O}}

\def\P{{\mathbb P}}

\def\Disc{{\rm Disc}}

\def\ind{{\rm ind}}

\def\Aut{{\rm Aut}}

\def\Tr{{\rm Tr}}
\def\mg{{\rm mg}}

\def\red{{\rm red}}
\def\ord{{\rm ord}}

\def\tTr{{\rm Tr}}

\def\sd{{\rm sd}}

\def\r{{\rm r}}
\def\Res{{\rm Res}}

\def\Vol{{\rm Vol}}
\def\R{{\mathbb R}}
\def\F{{\mathbb F}}

\def\cB{{\mathcal B}}
\def\EE{{\mathcal E}}

\def\FF{{\mathcal F}}

\def\RR{{\mathcal R}}
\def\Q{{\mathbb Q}}

\def\cB{{\mathcal B}}
\def\H{{\mathcal H}}

\def\V{{\mathcal V}}
\def\Z{{\mathbb Z}}

\def\P{{\mathbb P}}
\def\F{{\mathbb F}}
\def\Q{{\mathbb Q}}
\def\C{{\mathbb C}}
\def\H{{\mathcal H}}
\def\W{{\mathcal W}}

\def\fz1{{F_{\Z,1}}}

\def\SO{{\rm SO}}

\def\max{{\rm max}}

\def\s{{\rm {(1)}}}

\newtheorem{theorem}{Theorem}[section]
\newtheorem{Theorem}{Theorem}
\newtheorem{corollary}[theorem]{Corollary}
\newtheorem{cor}[theorem]{Corollary}
\newtheorem{lemma}[theorem]{Lemma}

\newtheorem{remark}[theorem]{Remark}
\newtheorem{Remark}[Theorem]{Remark}
\newtheorem{proposition}[theorem]{Proposition}

\newenvironment{proof}{\noindent {\bf Proof:}}{$\Box$ \vspace{2 ex}}

\usepackage{xcolor,color,url,lmodern}

\title{Secondary terms in the first moment of $|\Sel_2(E)|$}
\author{Arul Shankar and Takashi Taniguchi}
\date{}

\begin{document}

\maketitle

\begin{abstract}
We prove the existence of secondary terms of order $X^{3/4}$, with power saving error terms, in the counting functions of $|\Sel_2(E)|$, the $2$-Selmer group of $E$, for elliptic curves $E$ having height bounded by $X$. This is the first improvement on the error term of $o(X^{5/6})$, proved in \cite{BS2Sel}, where the primary term of order $X^{5/6}$ for this counting function was obtained.
\end{abstract}

\vspace{-.2in}
\tableofcontents

\section{Introduction}

The Poonen--Rains heuristics \cite{MR2833483} predict the distribution of the $p$-Selmer groups of  elliptic curves over $\Q$, for all primes $p$. These heuristics are supported by works of Bhargava and the first named author \cite{BS2Sel,BS3Sel,BS5Sel}, where it is proven that when the set of all elliptic curves over $\Q$ is ordered by height, the average size of the $p$-Selmer groups is $p+1$ for $p=2$, $3$, and $5$. On the computational side, Balakrishnan--Ho--Kaplan--Spicer--Stein--Weigandt \cite{MR3540965} collect and analyze data on ranks, $2$-Selmer groups, and other arithmetic invariants of elliptic curves when they are ordered by height. They find a persistently smaller average size of the $2$-Selmer group in the data. This is despite the fact that the average rank appears bigger in the data than its predicted value of 0.5. This discrepancy between the Goldfeld \cite{MR0564926} and Katz--Sarnak \cite{MR1659828} prediction of the average rank and the data was observed, for families of quadratic twists of an elliptic curve, and the family of all elliptic curves ordered by conductor, in \cite{MR2291676}. Thus it is natural to ask whether there exists a secondary term in the counting function of $|\Sel_2(E)|$, which explains the discrepancy between the data and and the theory.

Our main result proves the existence of a secondary term in this counting function. More precisely, let $\EE$ denote the family of all elliptic curves over $\Q$. Every elliptic curve in $\EE$ can be uniquely represented in the form $E_{AB}:y^2=x^3+Ax+B$, where $A$ and $B$ are integers such that $p^4\mid A$ implies that $p^6\nmid B$ for all primes $p$. We define the height $H(E_{AB}):=\max\{4|A|^3,27B^2\}$, and for a real number $X$, define the set 
\begin{equation*}
\EE^\pm_X:=\bigl\{E\in\EE: H(E)<X,\,\pm\Delta(E)>0\bigr\},
\end{equation*}
where $\Delta(E)$ is the discriminant of $E$.
Then we prove the following result.
\begin{Theorem}\label{Theorem1}
With notation as above, we have
\begin{equation*}
\sum_{E\in\mathcal{E}_X^\pm}|\Sel_2(E)|=
3\cdot\sum_{E\in\mathcal{E}_X^\pm} 1+C(\EE)^\pm X^{3/4}+O_\epsilon(X^{3/4-\alpha+\epsilon}),
\end{equation*}
for constants $C(\EE)^\pm$ and some $\alpha>0$. $($More precisely, we show that we can take $\alpha$ to be $1/3804$.$)$
\end{Theorem}

\noindent Since the size of $\EE_X^\pm$ grows like $c^\pm X^{5/6}+O(X^{1/2})$ for positive constants $c^\pm$, the above result recovers a secondary term with a power saving error term for the sum of $|\Sel_2(E)|$ over elliptic curves~$E$ ordered by height. In fact, this result is the first instance of a power saving error term obtained for the counting function of the $2$-Selmer groups of elliptic curves.

We express the constants $C^\pm$ in Theorem \ref{Theorem1} as the limit of the values at $s=1/2$ of certain Dirichlet series' which converges absolutely only to the right of ${\rm Re}(s)=1$. We prove that these Dirichlet series' have a meromorphic continuation to the right of $\Re(s)=1/3$ (in particular, their values at $s=1/2$ are well defined!), and that the limit of these values at $s=1/2$ exists. However, we are not yet able to find closed form formulas for $C^\pm$, or  numerically evaluate them, leaving this investigation to future work.

There has been, in arithmetic statistics and analytic number theory, a long history of studying lower order terms \cite{MR0215797,MR289428,MR0384717,MR1836927,MR3090184,MR3127806,MR4618015,MR3912928,KC, BTT}. This is because, beyond their own inherent interest, proving the existence of secondary terms have a number of consequences. First, the implied improvement in the size of the error terms have many applications, for example to the study of associated families of $L$-functions (see for example \cite{MR3318244,MR3993807}). 
Next, understanding secondary terms is necessary for providing numerical evidence to support conjectures, since these secondary terms have a significant contribution in the height range in which we are able to perform computations. (For example, even in the height range $H(E)\sim 10^{12}$, the secondary term of $X^{3/4}$ is within a factor of 10 of the primary term of $X^{5/6}$.) Finally, secondary terms are of considerable theoretical interest. In the function field case for instance, primary terms are obtained via proving homological stability results. In \cite[Problem 5]{MR4245477}, Venkatesh poses the question of what the topological significance of secondary terms are, and it is speculated in \cite{BDPW} that secondary terms might be related to secondary homological stability in the sense of \cite{MR4028513}.

\medskip

As in \cite{BS2Sel}, Theorem \ref{Theorem1} is proven by exploiting the connection between the $2$-Selmer groups of elliptic curves, and $\GL_2(\Z)$-orbits on  $V(\Z)={\rm Sym^4}(\Z^2)$, the set of irreducible integral binary quartic forms. The action of $\GL_2(\Z)$ on the set of integral binary {\it cubic} forms has been extensively studied. The symmetric cubed representation of $\GL_2(\Q)$ is {\it prehomogeneous}, i.e., the action over $\overline{\Q}$ has a Zariski open orbit. As a consequence, the ring of polynomial invariants for the action of $\GL_2(\Z)$ on integral binary cubic forms is generated by a single element, namely the discriminant. A famous result of Davenport \cite{MR43822} develops and uses geometry-of-numbers tehniques to determine asymptotics (of size $\sim \zeta(2)X/3$) for the number of $\GL_2(\Z)$-orbits on the set of integral irreducible binary cubic forms having discriminant bounded by $X$. 

Landmark work of Shintani \cite{MR289428} recovered Davenport's result, and in addition proved the existence of a secondary term of size $cX^{5/6}$ for a negative constant $c$.
More precisely, combined with his later work \cite{MR0384717} on his double zeta function, Shintani
proved that the number of $\GL_2(\Z)$-orbits on integral irreducible binary cubic forms, having nonzero discriminant bounded by $X$, can be expressed as a sum of two main terms, growing like $X$ and $X^{5/6}$, along with an error term of size $O_\epsilon(X^{2/3
+\epsilon})$. This was accomplished by considering the (Shintani) zeta function constructed from the counting function of binary cubic forms, proving that this zeta function has a meromorphic continuation to $\C$, and analyzing the location and residues of the poles. In fact, general theory developed by Sato--Shintani \cite{MR0344230} considers any prehomogeneous vector space of a reductive group whose singular set is an irreducible hypersurface, and proves that the associated zeta function has a meromorphic continuation to $\C$ under a certain condition. Moreover, they prove that the set of poles is contained within the set of zeros of the Bernstien--Sato polynomial associated to the invariant polynomial of this prehomogeneous representation. Thus in those prehomogeneous cases, there are natural guesses for the possible exponents of the lower order terms in the counting functions.

In our case, however, the symmetric fourth power representation of $\GL_2(\Z)$ is not prehomogeneous. The ring of polynomial invariants for the action of $\GL_2(\Z)$ on $V(\Z)$ is generated by two elements, usually denoted $I$ and $J$. Explicitly, for $f(x,y)=ax^4+bx^3y+cx^2y^2+dxy^3+ey^4$, we have
\begin{equation*}
I(f)=12ae-3bd+c^2;
\quad
J(f)=72ace+9bcd-27ad^2-27eb^2-2c^3  .
\end{equation*}
The discriminant polynomial on binary quartic forms can be expressed in terms of $I$ and $J$: we have $\Delta(f):=\Delta(I(f),J(f))=(4I(f)^3-J(f)^2)/27$.
Throughout this paper, we order $\GL_2(\Z)$-orbits on $V(\Z)$ by {\it height} $H$, defined as
\begin{equation}
    H(f):=H(I(f),J(f)):=\max\{|I(f)|^3,J(f)^2/4\}.
\end{equation}
The primary term in the counting function of $\GL_2(\Z)$-orbits on irreducible binary quartic forms, ordered by height, was obtained in \cite[Theorem 1.6]{BS2Sel}. In this article, we prove the existence of a secondary term, generalizing Shintani's theorem to the setting of binary quartic forms. Before stating this result, we need some additional notation. For a pair $(I,J)\in\R^2$ with $\Delta(I,J):=(4I^3-J^2)/27\neq 0$, let $E^{I,J}$ be the elliptic curve over $\R$ given by the equation $y^2=x^3-(I/3)x-(J/27)$, define
\begin{equation}\label{eq:omegatildeded}
\Omega(E^{I,J}):=\int_{\substack{(x,y)\in E^{I,J}(\R)\\y> 0}}\frac{dx}{y};\quad\quad\widetilde{\Omega}(I,J):=\Omega(E^{I,J})+\Omega(E^{I,-J}).
\end{equation}
We then define the quantities $C^{\Delta>0}_*$ and $C^{\Delta<0}_*$, for $*\in\{5/6,3/4\}$ to be
\begin{equation}\label{eq:mainconstdeg}
\begin{array}{rclrcl}
C_{5/6}^{\Delta>0}&:=&\displaystyle\int_{\R^2_{H<1,\Delta>0}}dIdJ,
&\quad C_{5/6}^{\Delta<0}&:=&\displaystyle\int_{\R^2_{H<1,\Delta<0}}dIdJ,
\\[.15in]
C_{3/4}^{\Delta>0}&:=&\displaystyle\int_{\R^2_{H<1,\Delta>0}}\widetilde{\Omega}(I,J)dIdJ,
&\quad
C_{3/4}^{\Delta<0}&:=&\displaystyle\int_{\R^2_{H<1,\Delta<0}}\widetilde{\Omega}(I,J)dIdJ,
\end{array}
\end{equation}
where
\begin{equation}
\begin{array}{rcl}
\R^2_{H<1,\Delta>0}&:=&\{(I,J)\in \R^2\mid H(I,J)<1,\Delta(I,J)>0\},\\[.1in]
\R^2_{H<1,\Delta<0}&:=&\{(I,J)\in \R^2\mid H(I,J)<1,\Delta(I,J)<0\}.
\end{array}
\end{equation}
For $i\in\{0,1,2\}$, and any subset $S$ of $V(\R)$, let $S^{(i)}$ denote the set of elements in $S$ having $4-2i$ real roots and $i$ pairs of complex conjugate roots in $\P^1_\C$. We further partition $S^{(2)}$ into $S^{(2+)}\cup S^{(2-)}$, where $S^{(2+)}$ (resp.\ $S^{(2+)}$) consists of positive (resp.\ negative) definite forms. Finally, for $i\in\{0,1,2+,2-\}$, let $h^{(i)}(I,J)$ denote the number of $\GL_2(\Z)$-orbits on irreducible elements in $V(\Z)^{(i)}$ having invariants $I$ and $J$.
Then we have the following result.
\begin{Theorem}\label{Th:Shintani4}
For $i\in\{0,1,2+,2-\}$, we have
\begin{equation*}
\sum_{\substack{(I,J)\in\Z^2\\H(I,J)<X}}h^{(i)}(I,J)=
\frac{2\zeta(2)}{27\sigma_i}C_{5/6}^{\circ}\cdot X^{5/6}
+
\frac{\zeta(1/2)}{108\sigma_i}C_{3/4}^{\circ}\cdot X^{3/4}+O_\epsilon(X^{2/3+\epsilon}),
\end{equation*}
where $\sigma_0=\sigma_{2\pm}=4$, $\sigma_1=2$, and we take $\circ$ to be $\Delta>0$ when $i\in\{0,2\pm\}$ and $\Delta<0$ when $i=1$.
\end{Theorem}

\noindent We note that the values of $C_{5/6}^{\circ}$ are easily computed: from \cite[(22),(23)]{BS2Sel}, we see that $C_{5/6}^{\Delta>0}=8/5$ and $C_{5/6}^{\Delta<0}=32/5$. Thus the primary terms of Theorem \ref{Th:Shintani4} agree with \cite[Theorem 1.6]{BS2Sel}. We leave a numerical estimation of the secondary constant to future work, but note that since $\zeta(1/2)<0$, the secondary term must be negative.

We note that Yukie (in \cite{yukie}) introduced 
and analyzed a certain ``global zeta integral'' for the space of binary forms of general degree $d$. For $d=4$, he showed that the zeta integral converges absolutely for ${\Re}(s)>5/6$ and has a holomorphic continuation to the region ${\Re}(s)>2/3$ except for simple poles at $s=5/6$ and $s=3/4$. We expect that Yukie's result is related to a smoothed version of Theorem \ref{Th:Shintani4}.

\subsubsection*{Outline of the proofs} As described previously, asymptotics for the number of $\GL_2(\Z)$-orbits on integral irreducible binary cubic forms with bounded discriminant were first obtained by Davenport \cite{MR43822} using geometry-of-numbers techniques. Using zeta function methods, Shintani \cite{MR289428, MR0384717} recovered these asymptotics, and also proved the existence of secondary main terms in these counting functions. Using a ``slicing technique'', Bhargava--Shankar--Tsimerman \cite{MR3090184}, obtained secondary terms for the number of $\GL_2(\Z)$-orbits on integral binary cubic forms, reproving Shintani's result. Combining sieving techniques developed by Belabas--Bhargava--Pomerance \cite{BBPEE} with these two counting methods, Taniguchi--Thorne \cite{MR3127806} and Bhargava--Shankar--Tsimerman \cite{MR3090184} independently and simultaneously obtained secondary terms for the counting functions of cubic fields.

In our case, since the representation $V(\Q)=\Sym^4(\Q^2)$ of $\GL_2(\Q)$ is not prehomogeneous, Shintani's methods are not available to us. However, the slicing technique generalizes in a straightforward manner and allows us to prove Theorem \ref{Th:Shintani4}, which we do in the largely self contained \S\ref{sec:slice}. Unfortunately, the sieving techniques used in the cubic case fail for us at the very first step. This is why, previous to our work, even power saving error terms were not known for the counting function of $|\Sel_2(E)|$. The primary reason is that to prove Theorem \ref{Theorem1}, it is necessary for us to count $\GL_2(\Z)$-orbits on $V(\Z)$, with height bounded by $X$, where each orbit $f$ is weighted by $\phi(f)$, where $\phi:V(\Z)\to\R$ is a $\GL_2(\Z)$-invariant function. It is possible to write $\phi$ as a product over $p$ of some functions $\phi_p:V(\Z_p)\to\R$. However, the functions $\phi_p$ are not defined modulo $p^k$ for any integer $k$. This is quite in contrast to the cubic case, where the function analogous to $\phi_p$ is the characteristic function of the set of binary cubic forms corresponding to maximal orders over $\Z_p$. This set (and hence its characteristic function) is defined modulo $p^2$. 
Resolving this difficulty of $\phi_p$ not being a periodic function requires many new ideas and tools, on both algebraic and analytic aspects of the arithmetic of binary quartic forms.

To begin, we develop approximation techniques in Section \ref{sec:locsol}, in order to analyze these functions $\phi_p$. We prove that they can be written as an infinite sum of functions $\phi_p^{(k)}$ for $k\geq0$, where $p$ is an odd prime, $\phi_p^{(k)}$ is defined modulo $p^{2k}$ and supported on the set of binary quartic forms whose discriminants are divisible by $p^{2k}$. We prove something analogous (though slightly weaker) for $p=2$ as well.
We can then write
\begin{equation}
\phi(f)=\prod_p \phi_p(f)
=\prod_p\sum_{k\geq0}  \phi_p^{(k)}(f)=\sum_{n\geq 1}\phi(n;f),
\end{equation}
where $\phi(n;\cdot):V(\Z)\to\R$ is given by
$\phi(n;f)=\prod_{p^k\parallel n}\phi_p^{(k)}(f)$. Our results on $\phi_p^{(k)}$ imply that $\phi(n;\cdot)$ is defined modulo $n^2$, and its support is on a sparse set: namely the set of elements in $V(\Z)$ whose discriminants are divisible by $n^2$ (up to absolutely bounded powers of $2$).

To sum $\phi(n;\cdot)$ over $\GL_2(\Z)$-orbits on $V(\Z)$ with height bounded by $X$ when $n$ is small (i.e., $n\ll X^{1/12+\delta}$ for small positive $\delta$), we use equidistribution methods. More specifically, we combine nontrivial bounds on orbital exponential sums on $V(\Z/p^2\Z)$ with twisted Poisson summation. This allows us to carry out the analogous sum of $\phi(n;\cdot)$, when the $H(f)<X$ condition is replaced with a smooth approximation. We then prove that this approximation is good enough to carry out the weighted sharp sum.

Finally, it remains to show that the contribution from $n\gg X^{1/12+\delta}$ is negligible. To do this, it is necessary for us to prove uniformity (or tail) estimates on the number of $\PGL_2(\Z)$-orbits on integral binary quartic forms of bounded height with discriminant divisible by $n^2$, on average over $n$. Moreover, unlike in the cubic case, we need these estimates for all $n$, not just squarefree $n$. In Section 6, we prove the following result, giving us a sufficiently strong uniformity estimate:
\begin{Theorem}\label{Thm:Unif}
Let $\W_n\subset V(\Z)$ denotes the set of irreducible integral binary quartic forms whose discriminants are divisible by~$n^2$, and whose associated Galois group is $S_4$ or $A_4$.
For positive real numbers $X$ and $M$,
we have
\begin{equation}\label{eq:UnifMain}
\sum_{n\geq M}\#\Bigl\{
f\in \frac{\W_n}{\PGL_2(\Z)}:H(f)<X\Bigr\}
\ll_\epsilon 
\frac{X^{5/6+\epsilon}}{M}+X^{11/15+\epsilon}.
\end{equation}
\end{Theorem}
It is of course crucial for us that $11/15=3/4-1/60<3/4$. The primary term in the above uniformity estimate should be optimal (up to a factor of $X^\epsilon$). Thus, we expect \eqref{eq:UnifMain} to be optimal (up to a factor of $X^\epsilon$) in the range $M\ll X^{1/10}$.

Theorem \ref{Thm:Unif} is used to bound the contribution from $n\gg X^{1/12+\delta}$. Combining this with the methods and results described above to handle smaller $n$ yields Theorem \ref{Theorem1}.

\medskip

This paper is organized as follows. We begin by introducing some notation and stating versions of our main results for congruence families of elliptic curves and binary quartic forms in \S\ref{sec:CongFam}. In \S\ref{sec:param}, we collect an assortment of parametrization results regarding elements in the $2$-Selmer groups of elliptic curves and regarding quartic fields. We prove our quartic analogue of Shintani's theorem in \S\ref{sec:slice}, using the ``slicing method'' developed in \cite{MR3090184}. We construct periodic approximations to the relevent local functions in \S\ref{sec:locsol}. Then in \S\ref{sec:FT}, we obtain density estimates required to prove that sum describing our secondary term constant converges. We also obtain bounds on the Fourier transforms of various subsets of $V(\Z/n\Z)$; this will be necessary to obtain equidistribution results. Our main uniformity estimate (Theorem \ref{Thm:Unif}) is proven in \S\ref{sec:Unif}, and this is combined with the previous results to prove all the main theorems in \S\ref{sec:proofs}. Finally, we compute some local densities appearing in our secondary terms in the appendix.

\subsection*{Acknowledgments}
AS was supported by an NSERC discovery grant, and is also supported by a Simons fellowship.
TT was supported by JSPS KAKENHI Grant Number 22H01115.
It is a pleasure to thank Manjul Bhargava, Dick Gross, Wei Ho, Ananth Shankar, and Jacob Tsimerman for many helpful conversations. The authors would like to thank Hiroyuki Ochiai for many useful comments.

\section{Statements of results for congruence families}\label{sec:CongFam}

Versions of our main results, Theorems \ref{Theorem1} and \ref{Th:Shintani4} also hold if we restrict to families of elliptic curves and binary quartic forms, respectively, satisfying congruence conditions modulo a fixed finite integer. In this section, we state these results, and also describe certain infinite sets of congruence conditions that we may impose on these families.

Let $p$ be a prime number, and let $\phi:V(\Z_p)\to \R$ (resp.\ $\phi:\Z_p^2\to\R$) be a function. For a positive integer $k$, we say that $\phi$ is {\it periodic with period $p^k$} if $\phi$ can be written as a composition of functions 
\begin{equation*}
V(\Z_p)\xrightarrow{} V(\Z_p)/p^kV(\Z_p)=V(\Z/p^k\Z)
\xrightarrow{\overline{\phi}}\R \quad \mbox{(resp. }
\Z_p^2\xrightarrow{} \Z_p^2/p^k\Z_p^2=(\Z/p^k\Z)^2
\xrightarrow{\overline{\phi}}\R),
\end{equation*}
where the first map is reduction modulo $p^k$.
We say that a function $\phi:V(\Z_p)\backslash\{\Delta=0\}\to\R$ (resp.\ $\phi:\Z_p^2\backslash\{\Delta=0\}\to\R$) is {\it well approximated by periodic functions} if for $k\geq 0$, there exist functions $\phi^{(k)}:V(\Z_p)\to\R$ (resp.\ $\phi:\Z_p^2\to\R$) bounded by $1$ in absolute value, with $\phi^{(0)}$ being identically $1$, satisfying the following properties:
\begin{itemize}
\item[{\rm (a)}] For every element $f$ in $V(\Z_p)\backslash\{\Delta=0\}$ (resp.\ $\Z_p^2\backslash\{\Delta=0\}$), we have
\begin{equation*}
\phi(f)=\sum_{k=0}^\infty\phi^{(k)}(f)=1+\sum_{k=1}^\infty\phi^{(k)}(f),
\end{equation*}
where the convergence is absolute.
In fact, Condition (c) implies that the sum is finite for each $f$.
\item[{\rm (b)}] The function $\phi^{(k)}$ is periodic with period $p^{2k}$.
\item[{\rm (c)}] The support of $\phi^{(k)}$ is contained within the set of elements $f\in V(\Z_p)$ with $p^{2k}\mid\Delta(f)$ (resp.\ $(I,J)\in \Z_p^2$ with $p^{2k}\mid\Delta(I,J)$).
\end{itemize}
If instead of Property (c), we only have that the functions $\phi^{(k)}$ are supported on elements $f$ in $V(\Z_p)$ or $\Z_p^2$ with $p^{2k-O(1)}\mid\Delta(f)$, then we say that $\phi$ is {\it almost well approximated by periodic functions}. A bounded function $\phi:\Z^2\to\R$ is said to be {\it large and locally well approximated} if for every $p$ there exist bounded functions $\phi_p:\Z_p^2\backslash\{\Delta=0\}\to\R$ such that $\phi(I,J)=\prod_p\phi_p(I,J)$ for every $(I,J)\in \Z^2\backslash\{\Delta=0\}$, the functions $\phi_p$ are all almost well approximated, and for all large enough $p$, the functions $\phi_p$ are well approximated.

We will need an analogous notion not only for functions from $V(\Z)\to\R$ but for {\it $\PGL_2(\Z)$-invariant} functions. Here, for a ring $R$, the action of $\PGL_2(R)$ on $V(R)$ comes from descending the following twisted action of $\GL_2(R)$ on $V(R)$: for $\gamma\in \GL_2(R)$ and $f\in V(R)$, define 
\begin{equation}\label{eq:PGL2action}
\gamma\cdot f(x,y):=f((x,y)\cdot\gamma)/(\det(\gamma)^2).
\end{equation}
(Note that over $\Z$, the twisted action and the ordinary action coincide.)
We say that a function $\psi:V(\Z/p^{2k}\Z)\to\R$ is {\it strongly invariant} if $\psi$ is $\PGL_2(\Z/p^{2k}\Z)$-invariant, and $\psi(t^2f)=\psi(f)$ for $t\in(\Z/p^{2k}\Z)^\times$.
Then a $\PGL_2(\Z)$-invariant function $\phi:V(\Z)\to\R$ is said to be {\it large and locally well approximated}
if there exist 
almost well approximated functions
$\phi_p:V(\Z_p)\backslash\{\Delta=0\}\to\R$
for each $p$,
such that $\phi(f)=\prod_p\phi_p(f)$ for every $f\in V(\Z)\backslash\{\Delta=0\}$, 
and for all large enough $p$, the functions $\phi_p$ are well approximated via strongly invariant periodic functions $\phi_p^{(k)}$.

\begin{Remark}{\em 
A periodic function $\phi:V(\Z_p)\to\R$ is always almost well approximated by periodic functions. Thus, every periodic $\PGL_2(\Z)$-invariant function $\phi:V(\Z)\to\R$ (i.e., a lift to $V(\Z)$ of a $\PGL_2(\Z/n\Z)$-invariant function $V(\Z/n\Z)\to\R$ for some integer $n$) is large and locally well approximated. Moreover, it is easy to check that the characteristic functions of many other natural sets in $V(\Z)$ are large and well approximated. For example, the set of $f$ having squarefree discriminant, and the set of $f$ such that the associated elliptic curve is semistable.
}
\end{Remark}

The following result is a congruence version of Theorem \ref{Theorem1}.

\begin{Theorem}\label{Theorem1:cong}
Let $\phi:\Z^2\to\Z$ be large and locally well approximated.
\begin{equation*}
\sum_{E_{AB}\in\mathcal{E}_X^\pm}|\Sel_2(E)|\phi(A,B)=
3\Bigl(\sum_{E_{AB}\in\mathcal{E}_X^\pm} \phi(A,B)\Bigr)+C(\EE;\phi)^\pm X^{3/4}+O_\epsilon(X^{3/4-\alpha+\epsilon}),
\end{equation*}
for constants $C(\EE;\phi)^\pm$ and the same $\alpha$ as in Theorem \ref{Theorem1}.
\end{Theorem}
\noindent In the theorem above we assume that $\phi$ is fixed (and suppress the dependence on $\phi$ in the error term$)$.

We note that the characteristic function of the set of pairs $(A,B)$ such that $E_{AB}$ is a semistable elliptic curve (resp.\ has squarefree discriminant) is large and locally well approximated. Hence, in addition to families of elliptic curves whose coefficients satisfy finitely many congruence conditions, the above result also includes natural families of elliptic curves defined by the imposition of infinitely many congruence conditions.

We next describe two generalizations of Theorem \ref{Th:Shintani4}. In our first result, we impose finitely many {\it splitting conditions} on the set of binary quartic forms being counted. An element $f$ of $V(\Z)$, $V(\Z_p)$, of $V(\F_p)$ is said to have {\it splitting type $(f_1^{e_1}\ldots f_k^{e_k})$} if the reduction of $f$ modulo $p$ splits as the product of irreducible polynomials of degrees $f_i$ raised to the powers $e_i$. As convention, we will define the splitting type of the zero form in $V(\F_p)$ by $(0)$.
For nonzero $\sigma_p=(f_1^{e_1}\ldots f_k^{e_k})$, we define its index by $\ind(\sigma_p):=\sum(e_i-1)f_i$.
The splitting types that arise this way are called {\it quartic splitting types}. Let $S$ be a finite set of primes, and for each $p\in S$, let $\sigma_p$ be a quartic splitting type. We denote this collection of splitting types $\{\sigma_p\}_{p\in S}$ by $\Sigma$. For $i\in\{0,1,2+,2-\}$ and $(I,J)\in\Z^2$, we let $h^{(i)}_{\Sigma}(I,J)$ denote the number of $\GL_2(\Z)$-orbits on irreducible elements $f\in V(\Z)$, having invariants $I$ and $J$, such that $f(x,y)$ has $r-i$ real roots, $i$ pairs of complex conjugate roots, and such that for all $p\in S$, the splitting type of $f$ at $p$ is $\sigma_p$. Then we have the following result.

\begin{Theorem}
\label{th:mainCongCountFin}
Let notation be as above. For $i\in\{0,1,2+,2-\}$, we have
\begin{equation*}
\begin{array}{rcl}
\displaystyle\sum_{\substack{(I,J)\in\Z^2\\H(I,J)<X}}h^{(i)}_\Sigma(I,J)&=&
\displaystyle\frac{2\zeta(2)}{27\sigma_i}\kappa_{5/6}(\Sigma)C_{5/6}^{\circ}\cdot X^{5/6}
+
\frac{\zeta(1/2)}{108\sigma_i}\kappa_{3/4}(\Sigma)C_{3/4}^{\circ}\cdot X^{3/4}
\\[.15in]
&&+O_\epsilon\left(\ell(\Sigma)X^{2/3+\epsilon}+\ell'(\Sigma)X^{1/2+\epsilon}\right),
\end{array}
\end{equation*}
where we take $\circ$ to be $\Delta>0$ when $i\in\{0,2+,2-\}$ and $\Delta<0$ when $i=1$, and
 we have $\kappa_\ast(\Sigma)=\prod_{p\in S}\kappa_\ast(\sigma_p)$, $\ell(\Sigma)=\prod_p \ell_p(\sigma_p)$, and $\ell'(\Sigma)=\prod_p \ell'_p(\sigma_p)$. The values $\kappa_{5/6}(\sigma_p)$ are listed in Table \ref{tab:primaryden}, while the values $\kappa_{3/4}(\sigma_p)$ can be computed from Table \ref{tablenu}. We have $\ell_p(\sigma_p)=p^{1-\ind_p(\sigma_p)}$ and $\ell'_p(\sigma_p)=p^{2-\ind_p(\sigma_p)}$, for nonzero splitting types $\sigma_p$. We take $\ell_p(0)=\ell'_p(0)=1$.
\end{Theorem}

\noindent We note that Theorem \ref{th:mainCongCountFin} can be easily generalized to collections $\Sigma=(\Sigma_p)_p$, where $\Sigma_p$ is a set of quartic splitting types (not just a singleton set). Moreover, for some natural choices for $\Sigma_p$, the error term can be easily improved. For example, when $\Sigma_p=\Sigma_p^{{\rm ur}}$ is the set of all unramified splitting types, the error can be improved to the error coming from the case when $\Sigma_p=\Sigma_p^{{\rm ram}}$ is the 
set of all ramified splitting types. This is because the count for $\Sigma_p^{{\rm ur}}$ is the difference between the count when there is no condition on $p$, and the count for $\Sigma_p^{{\rm ram}}$.

\medskip

Our next result involves summing a large and locally well approximated function $\phi:V(\Z)\to \R$ over $\PGL_2(\Z)$-orbits on irreducible integral binary quartic forms. To state it, we need some additional notation. First let $\nu(\phi)$ denote the {\it density} of $\phi$ and, for $a\in\Z$, let $\nu_a(\phi)$ denote the density of $\phi$ restricted to $V_a(\Z)$, the set of binary quartic forms whose $x^4$-coefficients are $a$. That is, we define 
\begin{equation}\label{eq:nuden}
\nu(\phi):=\int_{V(\widehat{\Z})}\phi(v)dv;\quad\quad
\nu_a(\phi):=\int_{V_a(\widehat{\Z})}\phi(v)dv.
\end{equation}
Above, the measures $dv$ are normalized so that $V(\widehat{\Z})$ and $V_a(\widehat{\Z})$ have volume $1$. Next, define the Dirichlet series
\begin{equation}\label{eq:Dpmdef}
D^\pm(\phi,s):=\sum_{a>0}\frac{\nu_{\pm a}(\phi)}{a^s}.
\end{equation}
We will prove in \S\ref{sec:FT} that $D^\pm(\phi,s)$ has an analytic continuation to the region $\Re(s)>1/3$, with at most a simple pole at $s=1$. In particular, the value $D^\pm(\phi,1/2)$ is well defined. For $i\in\{0,1,2+, 2-\}$ and $(I,J)\in\Z^2$, we let $h^{(i)}_{\phi}(I,J)$ denote the sum of $\phi(f)$ over $\GL_2(\Z)$-orbits on irreducible elements $f\in V(\Z)$, having invariants $I$ and $J$, such that $f(x,y)$ has $r-i$ real roots and $i$ pairs of complex conjugate roots. We have the following result.

\begin{Theorem}\label{th:mainCongCountFull}
Let $\phi:V(\Z)\backslash\{\Delta\neq 0\}\to\R$ be a large and locally well approximated function. For $i\in\{0,1,2+,2-\}$, let $V(\Z)^{(i)}$ $($resp.\ $V(\R)^{(i)})$ denote the set of elements $f\in V(\Z)$ $($resp.\ $f\in V(\R))$ having nonzero discriminant and exactly $4-2i$ real roots in $\P^1_\R$. Then for any positive constant $\alpha$ acceptable in Theorem \ref{Theorem1}, we have
\begin{equation*}
\begin{array}{rcccccl}
\displaystyle\sum_{\substack{(I,J)\in\Z^2\\H(I,J)<X}}h_\phi^{(0)}(I,J)\!\!&=&\!\!\displaystyle
\nu(\phi)\frac{2\zeta(2)}{27\sigma_0}C_{5/6}^{\Delta>0}X^{5/6}
\!\!&+&\!\!\displaystyle \frac{D^+(\phi,\frac12)+D^-(\phi,\frac12)}{216\sigma_0}C_{3/4}^{\Delta>0} X^{3/4}
\!\!&+&\!\!\displaystyle O_\epsilon(X^{3/4-\alpha+\epsilon});
\\[.25in]
\displaystyle\sum_{\substack{(I,J)\in\Z^2\\H(I,J)<X}}h_\phi^{(1)}(I,J)\!\!&=&\!\!\displaystyle
\nu(\phi)\frac{2\zeta(2)}{27\sigma_1}C_{5/6}^{\Delta<0}X^{5/6}
\!\!&+&\!\!\displaystyle \frac{D^+(\phi,\frac12)+D^-(\phi,\frac12)}{216\sigma_1}C_{3/4}^{\Delta<0} X^{3/4}
\!\!&+&\!\!\displaystyle O_\epsilon(X^{3/4-\alpha+\epsilon});
\\[.25in]
\displaystyle\sum_{\substack{(I,J)\in\Z^2\\H(I,J)<X}}h_\phi^{(2\pm)}(I,J)\!\!&=&\!\!\displaystyle
\nu(\phi)\frac{2\zeta(2)}{27\sigma_2}C_{5/6}^{\Delta>0}X^{5/6} 
\!\!&+&\!\!\displaystyle \frac{D^{\pm}(\phi,\frac12)}{108\sigma_2}C_{3/4}^{\Delta>0} X^{3/4}
\!\!&+&\!\!\displaystyle O_\epsilon(X^{3/4-\alpha+\epsilon}).
\end{array}
\end{equation*}
\end{Theorem}
\noindent In the theorem above we assume that $\phi$ is fixed (and suppress the dependence on $\phi$ in the error term$)$.

The average sizes of the $2$-torsion in the class groups $\Cl(K)$ and narrow class groups $\Cl^+(K)$ of monogenic cubic fields $(K,\alpha)$ ordered by height are determined in \cite{BSHMC}. The above result would imply that this bound of $2$-torsion in these class groups admits a secondary term growing as $X^{3/4}$.

Suppose that a large and locally well approximated function $\phi=\prod_p\phi_p$ is such that $\phi_p:V(\Z_p)\to\R$ is invariant under multiplication by units in $\Z_p$. This is the case, for instance, for the functions $\phi$ arising by imposing splitting conditions on the family of binary quartic forms, for the characteristic function of elements in $V(\Z)$ with squarefree discriminant, and the function $\phi$ corresponding to the $2$-torsion in the class groups of monogenic cubic fields. It is easy to see that $D^+(\phi,s)=D^-(\phi,s)$ for such functions $\phi$, giving a uniform description of the leading second order constants for the four possible splitting types at infinity.
We will also see in the appendix that for such functions $\phi$, the Dirichlet series $D^\pm(\phi,s)$ admits an Euler product. However, we do not believe this to be true of general functions $\phi$. Finally, we note that the characteristic function of the set of {\it soluble} binary quartic forms over $\Z_p$ is not invariant under multiplication by units. This is the main reason our secondary term constants of Theorems \ref{Theorem1} and \ref{Theorem1:cong} are not explicit.

\section{Parametrization results}\label{sec:param}

In this section, we collect an assortment of results parametrizing various arithmetic objects using (weighted) group orbits on lattices. Specifically, in \S 3.1, we  parametrize elements in the $2$-Selmer group of elliptic curves, following work of Birch and Swinnerton-Dyer \cite{BSDEC1}, Cremona \cite{MR1628193}, and Bhargava and the first named author \cite{BS2Sel}. Then in \S 3.2, we use results of Bhargava \cite{BHCL3} and Wood \cite{WoodThesis} to parametrize quartic rings along with a monogenized cubic resolvent ring. We combine this latter parametrization with a construction of Heilbronn \cite{HCF} to parametrize $2$-torsion elements in the dual of the class group of monogenized cubic fields.

\subsection{The $2$-Selmer groups of elliptic curves}

For $E=E_{AB}\in\EE$, we define the following two invariants of $E$:
\begin{equation*}
I(E):=-2^4\cdot 3\cdot A;\quad\quad J(E):=-2^6\cdot 3^3\cdot B.
\end{equation*}
We denote the elliptic curve having invariants $I$ and $J$ by $E^{IJ}$; note that the height of $E^{IJ}$ is given by $H(E^{IJ})=H(I,J)/(2^{10}3^3)$.
We recall the following two results, proved originally by Birch--Swinnerton-Dyer \cite{BSDEC1} and further refined by Cremona \cite{MR1628193}, regarding the connection between the $2$-Selmer groups of elliptic curves and $\PGL_2$-orbits on $V$. The results are stated in the notation in~\cite[\S3.1]{BS2Sel} (see \cite[Theorem 3.2 and Proposition 3.3]{BS2Sel}). A binary quartic form $f(x,y)$ with coefficients in a field $K$ is said to be {\it $K\!$-soluble} if the equation $z^2=f(x,y)$ has a nontrivial solution. A binary quartic form $f(x,y)$ in $V(\Q)$ is said to be {\it locally soluble} if $f$ is soluble over $\R$ and over $\Q_p$ for every prime $p$.

\begin{theorem}\label{thparame2e}
  Let $K$ be a field having characteristic not $2$ or $3$. Let
  $E:y^2=x^3-\frac{I}3x-\frac{J}{27}$ be an elliptic curve over $K$.
  Then there exists a bijection between elements in $E(K)/2E(K)$ and
  $\PGL_2(K)$-orbits of $K\!$-soluble binary quartic forms having
  invariants equal to $I$ and $J$, given by
\begin{equation*}
(\xi,\eta)+2E(K)\mapsto \PGL_2(K)\cdot\left(\frac{1}{4}x^4-\frac{3}{2}\xi x^2y^2+2\eta xy^3+\left(\frac{I}{3}-\frac{3}{4}\xi^2\right)y^4\right).
\end{equation*}
Under this bijection, the identity element in
$E(K)/2E(K)$ corresponds to the $\PGL_2(K)$-orbit of binary quartic
forms having a linear factor over $K$.

Furthermore, the stabilizer in $\PGL_2(K)$ of any $($not necessarily
$K\!$-soluble$)$ binary quartic form $f$ in $V_K$, having nonzero
discriminant and invariants $I$ and $J$, is isomorphic to $E(K)[2]$,
where $E$ is the elliptic curve defined by
$y^2=x^3-\frac{I}{3}x-\frac{J}{27}$.
\end{theorem}

This leads to the following parametrization of the $2$-Selmer groups of elliptic curves over~$\Q$.

\begin{theorem}\label{2spar}
  Let $E=E^{IJ}$ be an elliptic curve over $\Q$. Then the elements
  of the $2$-Selmer group of $E$ are in one-to-one correspondence with
  $\PGL_2(\Q)$-equivalence classes of locally soluble integral binary
  quartic forms having invariants equal to $I$ and $J$.

  Furthermore, the set of integral binary quartic forms that have a
  rational linear factor and invariants equal to $I$ and $J$ lie
  in a single $\PGL_2(\Q)$-equivalence class, and this class corresponds to
  the identity element in the $2$-Selmer group of $E$.
\end{theorem}

Next, we express the number of $\PGL_2(\Q)$-equivalence classes of locally soluble integral binary quartic forms having fixed invariants as the weighted count of $\PGL_2(\Z)$-orbits on integral binary quartic forms. Moreover, we show that this weight is {\it local}, i.e., it can be expressed as a product of weights defined over $\Z_p$. The condition of local solubility is clearly local: let $\ell_p:V(\Z_p)\to\R$ be the characteristic function of the set of elements in $V(\Z_p)$ that are soluble over $\Q_p$, and let $\ell:V(\Z)\to\R$ be given by $\ell(f):=\prod\ell_p(f)$. That is, $\ell$ is the characteristic function of locally soluble elements in $V(\Z)$. 

Given a binary quartic form $f\in V(\Z)$ (resp.\ $f\in V(\Z_p)$ for some prime $p$), let $B(f)$ (resp.\ $B_p(f)$) denote a set of representatives for the action of
$\PGL_2(\Z)$ (resp.\ $\PGL_2(\Z_p)$) on the $\PGL_2(\Q)$-equivalence class of $f$ (resp. $\PGL_2(\Q_p)$-equivalence class of $f$) in
$V(\Z_p)$. For an integral binary quartic form $f$ in  $V(\Z)$ or $V(\Z_p)$, respectively define 
\begin{equation*}
\begin{array}{rcl}
\Aut_\Q(f):=\Stab_{\PGL_2(\Q)}(f);&\quad\quad&
\Aut_\Z(f):=\Stab_{\PGL_2(\Z)}(f);
\\[.1in]
\Aut_{\Q_p}(f):=\Stab_{\PGL_2(\Q_p)}(f);
&\quad\quad&
\Aut_{\Z_p}(f):=\Stab_{\PGL_2(\Z_p)}(f).
\end{array}
\end{equation*}
We define the {\it global weight} $m(f)$ for $f$ in $V(\Z)$ and the {\it local weight} $m_p(f)$ for $f$ in $V(\Z_p)$ to be:
\begin{equation}\label{eq:mp}
\begin{array}{rcccl}
m(f)&:=&\displaystyle\sum_{f'\in B(f)}\frac{\#\Aut_\Q(f')}{\#\Aut_\Z(f')}&=&\displaystyle\sum_{f'\in B(f)}\frac{\#\Aut_\Q(f)}{\#\Aut_\Z(f')};
\\[.2in]
m_p(f)&:=&\displaystyle\sum_{f'\in B_p(f)}\frac{\#\Aut_{\Q_p}(f')}{\#\Aut_{\Z_p}(f')}&=&\displaystyle\sum_{f'\in B_p(f)}\frac{\#\Aut_{\Q_p}(f)}{\#\Aut_{\Z_p}(f')}.
\end{array}
\end{equation}
Moreover, given $f\in V(\Z)$ (resp.\ $f\in V(\Z_p)$), we define $\PGL_2(\Q)_f$ (resp.\ $\PGL_2(\Q_p)_f$) to be the set of elements $\gamma\in \PGL_2(\Q)$ (resp.\ $\gamma \in\PGL_2(\Q_p)$), such that $\gamma\cdot f$ belongs to $V(\Z)$ (resp.\ $V(\Z_p)$). Then the following result follows from \cite[Proposition 3.6]{BS2Sel} and its proof.
\begin{proposition}\label{prop:mp}
For elements $f$ in $V(\Z)$ and $V(\Z_p)$, we have the equalities
\begin{equation*}
m(f)=\#[\PGL_2(\Z)\backslash\PGL_2(\Q)_f],\quad
m_p(f)=\#[\PGL_2(\Z_p)\backslash\PGL_2(\Q_p)_f],
\end{equation*}
respectively. Moreover, we have
\begin{equation*}
m(f)=\prod_p m_p(f)
\end{equation*}
for every $f\in V(\Z)$.
\end{proposition}

Suppose an elliptic curve $E^{IJ}$ over $\Q$ is such that $E^{IJ}(\Q)[2]=\{0\}$. Then it follows by Theorem \ref{thparame2e} that if $f\in V(\Q)$ has invariants $I$ and $J$, then $\Aut_\Q(f)=\Aut_\Z(f)=\{{\rm id}\}$. As a consequence, $m(f)=\#B(f)$. First, for integers $I$ and $J$, let $V(\Z)_{IJ}$ denote the set of elements in $V(\Z)$ with invariants $I$ and $J$. An element in $V(\Q)$ is said to be {\it generic} if it is irreducible and its corresponding elliptic curves has trivial $2$-torsion. For any set $S\subset V(\Q)$ let $S^\gen$ denote the set of  generic elements in $S$. Then Theorem \ref{2spar} immediately implies the following result.

\begin{theorem}\label{th:ellipselpar}
Let $E$ be an elliptic curve over $\Q$ with trivial $2$-torsion and invariants $I$ and $J$. Then we have
\begin{equation*}
|\Sel_2(E)|=1+\sum_{f\in \frac{V(\Z)_{IJ}^\gen}{\PGL_2(\Z)}}\frac{\ell(f)}{m(f)}.
\end{equation*}
\end{theorem}

\subsection{Quartic rings with monogenic cubic resolvent rings}

In this section, we recall results that parametrize certain cubic and quartic rings over $\Z$.
We begin with the following parametrization of cubic rings due to Levi \cite{Levi}, Delone--Faddeev \cite{DFCF}, and Gan-Gross-Savin \cite{GGSCF}. Let $U=\Sym^3(\Z^2)$ denote the space of binary cubic forms. The group $\GL_2$ acts on $U$ via the following twisted action: $\gamma\cdot f(x,y):=f((x,y)\cdot\gamma)/\det(\gamma)$. Then we have the following result.
\begin{proposition}[\cite{Levi},\cite{DFCF},\cite{GGSCF}]\label{prop:DF}
There is a natural bijection between the set of cubic rings over $\Z$ upto isomorphism and the set of $\GL_2(\Z)$-orbits on binary cubic forms.
\end{proposition}

\noindent In fact, this bijection has a very explicit description. Let $C$ be a cubic ring over $\Z$, and let $L:=C/\Z$ be the associated two dimensional lattice. Consider the {\it index form} $\ind:L\to \Z$ which sents $\alpha\in L$, to the signed index of $\Z[\alpha]$ in $C$. More precisely, this is defined by $\ind\colon L\ni \alpha\mapsto \alpha\wedge\alpha^2\in\wedge^2L$.
Since $L\cong\Z^2$ and $\ind$ is cubic, this yields an integral binary cubic form $f\in U(\Z)$ upto the action of $\GL_2(\Z)$.  A single element $f\in U(\Z)$ corresponds to a cubic ring $C$ along with a basis $\{\alpha_1,\alpha_2\}$ of $C/\Z$. These can be lifted in a unique way to two elements $\beta_1$ and $\beta_1$ of $C$ whose traces belong to $\{-1,0,1\}$, and it is then true that $\{1,\beta_1,\beta_2\}$ is a basis for $C$ as a $\Z$-module.

\medskip

Next, let $W=2\otimes\Sym^2(3)$ denote the space of pairs of ternary quadratic forms. There is a natural action of the group $\GL_2\times\SL_3$ on $W$. We represent elements in $W(\Z)$ by pairs $(A,B)$ of $3\times 3$ symmetric matrices with coefficients $\frac12 a_{ij}$ and $\frac12 b_{ij}$ with $1\leq i\leq j\leq 3$, where $a_{ii},b_{ii}\in2\Z$ and $a_{ij},b_{ij}\in\Z$ for $i\neq j$. Then the following landmark result of Bhargava \cite[Theorem 1]{BHCL3} parametrizes pairs $(Q,C)$, where $Q$ is a quartic ring and $C$ is a cubic resolvent ring of $Q$.
\begin{theorem}[\cite{BHCL3}]\label{th:QCparam}
There is a canonical bijection between the set of $\GL_2(\Z)\times\SL_3(\Z)$-orbits on $W(\Z)$ and the set of isomorphism classes of pairs $(Q,C)$, where $Q$ is a quartic ring and $C$ is a cubic resolvent ring of $Q$.
\end{theorem}

\noindent There is a natural {\it resolvent} map $\Res: W\to U$ given by $\Res(A,B):=4\det(Ax-By)$. This map is $\SL_3$-invariant and $\GL_2$-covariant. Moreover, if $(A,B)\in W(\Z)$ corresponds to the pair $(Q,C)$ of rings, the binary cubic form corresponding to $C$ is $\Res(A,B)\in U(\Z)$.

\medskip

Let $C$ be a rank-$n$ ring over $\Z$. An element $\alpha\in C$ is said to be a {\it monogenizer} for $C$ if $C=\Z[\alpha]$. We then say that the pair $(C,\alpha)$ over $\Z$ is an ($n$-ic) {\it monogenized rank-$n$ ring}. Two monogenized rings $(C,\alpha)$ and $(C',\alpha')$ are said to be {\it isomorphic} if there is an isomorphism from $C$ to $C'$ such that it sends $\alpha$ to $\alpha'+n$ for some $n\in\Z$.
For a ring $R$, let $U_1(R)$ denote the set $x^3+tx^2+sx+r$, where $r,s,t\in R$. There is a natural action of $R$ on $U_1(R)$ given by $(r\cdot f)(x):=f(x+r)$ for $r\in R$ and $f\in U_1(R)$.
Then we have the following result.
\begin{proposition}\label{prop:monCparam}
The map $f(x)\mapsto\left(\Z[x]/(f(x)),\overline x\right)$ gives a bijection between the set of $\Z$-orbits on $U_1(\Z)$ and the set of pairs $(C,\alpha)$, where $C$ is a cubic ring and $\alpha\in C/\Z$ is a monogenizer for $C$.
\end{proposition}

Let $Q$ be a quartic ring and let $C$ be a monogenized cubic resolvent of $Q$. Then there exists a monic binary cubic form representing $C$. Hence, there exists a pair $(A,B)\in W(\Z)$, corresponding to $(Q,C)$, such that $\det(A)=1/4$. The set of integral symmetric $3\times 3$-matrices with determinant $1/4$ forms a single $\SL_3(\Z)$-orbit. Hence, the pair $(Q,C)$ can be represented by $(A_0,B)\in W(\Z)$, wher $A_0$ is an anti-diagonal $3\times 3$ matrix with coefficients $1/2$, $-1$, and $1/2$. Wood \cite{WoodThesis} considers the following embedding of $V(\Z)$ into $W(\Z)$:
\begin{equation}\label{eq:iota}
\iota:ax^4+bx^3y+cx^2y^2+dxy^3+ey^4\mapsto
\left[
\begin{pmatrix}
&&\frac12\\&-1&\\\frac12&&
\end{pmatrix},
\begin{pmatrix}
a&\frac{b}2&\\[.02in]
\frac{b}{2}&c&\frac{d}{2}\\
&\frac{d}{2}&e
\end{pmatrix}
\right].
\end{equation}
The special orthogonal group $\SO_{A_0}$ is naturally isomorphic to $\PGL_2$. Moreover, the map $\iota$ respects the two actions of $\PGL_2(\Z)$ on $V(\Z)$ and of $\SO_{A_0}(\Z)$ on the set of pairs $(A_0,B)$ in $W(\Z)$. Combining the construction of $\iota$ with Theorem \ref{th:QCparam} and Proposition \ref{prop:monCparam}, Wood proves the following result.

\begin{theorem}\label{th:bqfQC}
There is a natural bijection between the set of $\PGL_2(\Z)$-orbits on $V(\Z)$ and the set of isomorphism classes of pairs $(Q,C,\alpha)$, where $Q$ is a quartic ring and $(C,\alpha)$ is a monogenized cubic resolvent ring of $Q$.
\end{theorem}

\noindent Given a binary quartic form $f(x,y)$ in $V(\Z)$ or $V(\Z_p)$ for some prime $p$, we let $Q_f$ and $C_f$ denote the quartic ring and monogenic cubic ring associated to $f$. Note that $C_f$ only depends on the invariants $I$ and $J$ of $f$ since the coefficients of $\Res(\iota(f))$ are invariants for the action of $\PGL_2$ on $V$. We denote this cubic ring by $R_{IJ}$.

\section{Secondary terms in the count of integral binary quartic forms}\label{sec:slice}

In this section, we determine secondary terms for the counting function of $\PGL_2(\Z)$-orbits on generic elements of $V(\Z)$, where each orbit is weighted by $\phi$, a $\PGL_2(\Z)$-invariant periodic function on $V(\Z)$. First, in \S4.1, we set up notation, and introduce the averaging method from \cite{BS2Sel}. In \S4.2, we give bounds for the number of nongeneric $\PGL_2(\Z)$-orbits on integral binary quartic forms with bounded height and nonzero discriminants. A bound of $O(X^{3/4+\epsilon})$ was proved in \cite[Lemma 2.4]{BS2Sel}, but in order to recover a secondary term, we need to improve this. Then in \S4.3, we use the ``slicing method'' developed in \cite{MR3090184} to obtain primary and secondary terms for our counting functions and complete the proofs of Theorems \ref{Th:Shintani4} and \ref{th:mainCongCountFin} (modulo a volume computation, carried out in \S8).

\subsection{Preliminaries}

Recall that we say an element $f\in V(\Q)$ is {\it generic} if $f$ is irreducible over $\Q$ and the corresponding elliptic curve has trivial $2$-torsion. The latter condition is equivalent to the stabilizer of $f$ in $\PGL_2(\Q)$ being trivial. By Theorem \ref{thparame2e}, it follows that $f$ is generic if and only if both $f$ and the cubic resolvent polynomial of $f$ are irreducible.

We partition the set of elements in $V(\R)$ with nonzero discriminant into the following sets:
\begin{equation*}
V(\R)\backslash\{\Delta=0\}=
V(\R)^{(0)}\cup V(\R)^{(1)} \cup V(\R)^{(2)}=V(\R)^{(0)}\cup V(\R)^{(1)} \cup V(\R)^{(2+)} \cup V(\R)^{(2-)}.
\end{equation*}
Above, for $i\in\{0,1,2\}$, the set $V(\R)^{(i)}$ consists of elements $f\in V(\R)$ having $4-2i$ distinct real roots in $\P^1(\R)$ and $i$ pairs of distinct complex conjugate roots in $\P^1(\C)$. The set $V(\R)^{(2)}$ consists of definite binary quartic forms, and we further partition it into the set of positive definite binary quartics $V(\R)^{(2+)}$ and the set of negative definite binary quartics $V(\R)^{(2-)}$. For a set $S\subset V(\R)$ and $i\in\{0,1,2+,2-\}$, we define $S^{(i)}$ to be $S\cap V(\R)^{(i)}$.

For a $\PGL_2(\Z)$-invariant function $\phi:V(\Z)\to\R$, and $i\in\{0,1,2+,2-\}$, we define the counting function $N^{(i)}(\phi;X)$ to be
\begin{equation*}
N^{(i)}(\phi;X):=\sum_{\substack{f\in\PGL_2(\Z)\backslash V(\Z)^{(i),\gen}\\H(f)<X}}\phi(f).
\end{equation*}
We start with a slightly modified treatment of \cite[\S2.1,2.3]{BS2Sel}, and provide fundamental domains for the action of $\PGL_2(\Z)$ on $V(\R)$. The only difference is that we work with the action of $\PGL_2$ on $V$, while \cite[\S2]{BS2Sel} considers the action of $\GL_2$ on $V$. We fix $i\in\{0,1,2+,2-\}$, and note that $V(\R)^{(i)}$ contains only points with positive discriminant when $i\in\{0,2+,2-\}$ and only points with negative discriminant when $i=1$. Given any pair $(I,J)$ with $\Delta(I,J)>0$ (resp.\ $\Delta(I,J)<0$), the set of elements in $V(\R)^{(i)}$, for $i\in \{0,2+,2-\}$ (resp.\ $i=1$), having invariants $I$ and $J$ consists of a single $\PGL_2(\R)$-orbit.
Fundamental sets $L^{(i)}$ for the action of $\PGL_2(\R)$ on the set of elements in $V(\R)^{(i)}$ having height $1$ are given in \cite[Table 1]{BS2Sel}.
Let $R^{(i)}_\infty$ denote the set $\R_{>0}\cdot L^{(i)}$. Then the sets $R^{(i)}_\infty$ are fundamental sets for the action of $\PGL_2(\R)$ on $V(\R)^{(i)}$. Moreover $R^{(i)}_\infty$ contains exactly one element with invariants $I$ and $J$ with $\Delta(I,J)>0$ when $i\in\{0,2+,2-\}$ and exactly one element with invariants $I$ and $J$ with $\Delta(I,J)<0$ when $i=1$.

Let $\FF$ be Gauss' fundamental domain for the action of $\PGL_2(\Z)$ on $\PGL_2(\R)$. Explicitly, we write
$\FF=\bigl\{u(t)k:u\in N(t),\,(t)\in A',\,k\in K
\bigr\}$, where
\begin{equation*}
    N(t):=\left\{u:=\begin{pmatrix} 1&\\u&1
\end{pmatrix}:u\in {\rm i}(t)\right\},\quad
A':=\left\{(t):=\begin{pmatrix} t^{-1}&\\&t
\end{pmatrix}:t\geq\frac{\sqrt[4]{3}}{\sqrt{2}}\right\},\quad
K:=\SO_2(\R).
\end{equation*}
Above ${\rm i}(t)$ is a subset of $[-1/2,1/2]$ depending on $t$, and is all of $[-1/2,1/2]$ when $t\geq1$.
For elements $u\in N(t)$ and $(t)\in A'$, we denote the product $u(t)$ by $(u,t)$.

Let $G_0$ be a nonempty semialgebraic left $K$-invariant bounded open subset of $\PGL_2(\R)$. 
Denote the set of elements in $R^{(i)}_\infty$ (resp.\ $G_0\cdot R^{(i)}_\infty$) with height less than $X$ by $R^{(i)}_X$ (resp.\ $S^{(i)}_X$).
Then Bhargava's averaging method, developed in \cite{dodqf,dodpf} yields the following equality for every $\PGL_2(\Z)$-invariant function $\phi:V(\Z)\to\R$:
\begin{equation}\label{eq:avg}
\begin{array}{rcl}
\displaystyle N^{(i)}(\phi;X)&=&\displaystyle
\frac{1}{\sigma_i\Vol(G_0)}
\int_{\gamma\in\FF}\Bigl( \sum_{
f\in\gamma S^{(i)}_X\cap V(\Z)^\gen}
\phi(f)\Bigr)d\gamma
\\[.3in]&=&\displaystyle
\frac{1}{\sigma_i\Vol(G_0)}
\int_{t\geq \frac{\sqrt[4]{3}}{\sqrt{2}}}\int_{u\in N(t)}\Bigl( \sum_{
f\in(u,t) S^{(i)}_X\cap V(\Z)^\gen}
\phi(f)\Bigr)t^{-2}dud^\times t.
\end{array}
\end{equation}
Above, $\sigma_i=4$ when $i\in\{0,2+,2-\}$ and $\sigma_i=2$ when $i=1$, $d\gamma=t^{-2}dud^\times t dk$ is a Haar measure on $\PGL_2(\R)$ with $dk$ normalized so that the volume of $K$ is $1$, and the volume of $G_0$ is computed with respect to $d\gamma$. A version of \eqref{eq:avg}, using the group $\GL_2(\R)$ rather than $\PGL_2(\R)$, and where $\phi$ is taken to be $1$, is proved in \cite[Theorem 2.5]{BS2Sel}; the proof in our case is identical.

\subsection{Bounding the number of nongeneric elements}

In this subsection, we bound the number of $\PGL_2(\Z)$-orbits on nongeneric integral orbits on integral binary quartic forms having bounded height. By construction, we have $R^{(i)}_X=X^{1/6}R^{(i)}_1$.
Since the sets $L^{(i)}$ of \cite[Table 1]{BS2Sel} (and hence the sets $R^{(i)}_1$) are absolutely bounded, it follows that the absolute values of coefficients of the elements in $R^{(i)}_X$ are $\ll X^{1/6}$. Since $G_0$ is bounded, the same is true of the coefficients of elements in $S^{(i)}_X$. This means that when $t>CX^{1/24}$, for some sufficiently large $C$, every element in $(u,t)S^{(i)}_X\cap V(\Z)$ has $x^4$-coefficient in $(-1,1)$, and hence $x^4$-coefficient equal to $0$. Such an element is not generic implying that $(u,t)S^{(i)}_X\cap V(\Z)^\gen$ is empty. That is, every integral element $f(x,y)$ deep enough in the cusp is nongeneric. Moreover, it is nongeneric because $y$ is a factor of $f(x,y)$.

Let $a$, $b$, $c$, $d$, and $e$, denote the coeffients of elements in $V(\R)$. That is, for $f(x,y)\in V(R)$, the $x^4$-, $x^3y$-, $x^2y^2$-, $xy^3$-, and $y^4$-coefficients of $f(x,y)$ are denoted by $a(f)$, $b(f)$, $c(f)$, $d(f)$, and $e(f)$, respectively. Let $V(\Z)^\red$ denote the set of reducible elements in $V(\Z)$. In the next two lemmas, we bound the number of elements in $V(\Z)^\red$ that lie within the main body of the fundamental domain.

\begin{lemma}\label{lem:red0}
We have
\begin{equation*}
\int_{1\ll t\ll X^{1/24}}\int_{u\in [-1/2,1/2]}\# \bigl\{(u,t)S_X^{(i)}\cap V(\Z)^\red:a(f)\neq 0\bigr\}t^{-2}dud^\times t\ll_\epsilon X^{2/3+\epsilon}.
\end{equation*}
\end{lemma}
This follows immediately from the proof of \cite[Lemma 2.3]{BS2Sel}.

\begin{lemma}\label{lem:red1}
For a positive real number $\delta\leq 1/24$, we have
\begin{equation*}
\int_{t\gg 1}^{X^\delta}\int_{u\in [-1/2,1/2]}\# \bigl\{f(x,y)\in (u,t)S_X^{(i)}\cap V(\Z):a(f)=0\bigr\}t^{-2}dud^\times t\ll X^{2/3+2\delta}.
\end{equation*}
\end{lemma}
\begin{proof}
For $u\in[-1/2,1/2]$ and $1\ll t<X^{\delta}$
we have
\begin{equation*}
\# \bigl\{f\in (u,t)S_X^{(i)}\cap V(\Z):a(f) = 0\bigr\}\ll t^4X^{4/6},
\end{equation*}
since $(u,t) S_X^{(i)}$ is contained within the set of elements $f(x,y)=ax^4+bx^3y+cx^2y^2+dxy^3+ey^4$ with $|a|\ll t^{-4}X^{1/6}$, $|b|\ll t^{-2}X^{1/6}$, $|c|\ll X^{1/6}$, $|d|\ll t^{2}X^{1/6}$, and $|e|\ll t^{4}X^{1/6}$. Thus
\begin{equation*}
\begin{array}{rcl}
\displaystyle\int_{t\gg 1}^{X^\delta}\int_{u\in [-1/2,1/2]}\# \bigl\{f\in (u,t)S_X^{(i)}\cap V(\Z):a(f)=0\bigr\}t^{-2}dud^\times t
\ll
X^{4/6}\int_{t\gg 1}^{X^\delta} t^{2} d^\times t
\ll
X^{2/3+2\delta},
\end{array}
\end{equation*}
and the lemma follows.
\end{proof}

\noindent The next result, which follows immediately from the method of Lemma \ref{lem:red1}, is needed in the sequel.
\begin{lemma}\label{lem:b0bound}
We have
\begin{equation}
\int_{t\gg 1}\int_{u\in [-1/2,1/2]}\# \bigl\{f(x,y)\in (u,t)S_X^{(i)}\cap V(\Z)^\gen:b(f)=0\bigr\}t^{-2}dud^\times t\ll X^{2/3}.
\end{equation}
\end{lemma}

Suppose an element $f\in V(\Z)$ is irreducible but nongeneric. Then its cubic resolvent polynomial must be reducible. To bound the number of such elements in the fundamental domain, we need the following result which follows from work of Nakagawa (see \cite[Theorem 1]{MR1342021}).
\begin{proposition}\label{prop:Nak}
Let $L$ be a quartic \'etale algebra over $\Q$ with ring of integers $\O_L$. Then the number $N(\O_L,k)$ of suborders of $\O_L$ having index $k$ is bounded by 
\begin{equation}
N(\O_L,k)\ll_\epsilon
k^\epsilon\prod_{\substack{p^2\nmid\Disc(\O_L)\\p^3||k}}p
\prod_{\substack{p^2\nmid\Disc(\O_L)\\p^e||k,\ e\geq 4}}p^{\lfloor e/2\rfloor}
\prod_{\substack{p^2\mid\Disc(\O_L)\\p^e||k,\ e\geq 2}}p^{\lfloor e/2\rfloor}
\end{equation}
In particular, we have $N(\O_L,k)\ll_\epsilon k^\epsilon N(k)$
where $N(k):= \prod_{p^e\parallel k}p^{\lfloor e/2\rfloor}$, a bound that is independent of $L$.
\end{proposition}

We have the following lemma improving \cite[Lemma 2.4]{BS2Sel}.
\begin{lemma}\label{lem:red2}
The number of $\PGL_2(\Z)$-orbits $f$ on $V(\Z)$ with $\Delta(f)\neq 0$ and $H(f)<X$ such that the stabilizer of $f$ in $\PGL_2(\Q)$ is nontrivial is bounded by $O_\epsilon(X^{2/3+\epsilon})$.
\end{lemma}
\begin{proof}
If the stabilizer of $f(x,y)$ in $\PGL_2(\Q)$ is nontrivial, then the cubic resolvent polynomial $g(x,y)$ is reducible. Let $Q$ and $C$ denote the quartic ring and cubic ring corresponding to $f$ and $g$, respectively. Let $L=Q\otimes\Q$ and $K = C\otimes \Q$. It follows from a result of Baily (see \cite{MR0564533}) that the number of possible quartic $\Q$-algebras associated to a cubic algebra $K$ is $\ll_\epsilon |\Cl(K)[2]||\Delta(K)|^\epsilon$. Since $K$ is isomorphic to $\Q\times F$, for a quadratic field $F$, we have $|\Cl(K)[2]|\ll_\epsilon|\Delta(K)|^\epsilon$. Therefore, once a reducible integer cubic polynomial $g$ is fixed, the associated quartic algebra $L$ has $\ll_\epsilon |\Delta(g)|^\epsilon$ choices.

A quick application of \cite[Lemma 4.3]{BSHMC} implies that the number of reducible integral traceless monic cubic polynomials $g$ with $H(g)<X$ is $\ll_\epsilon X^{1/2+\epsilon}$. Fix such a polynomial $g$ as well as one of the $O_\epsilon(X^\epsilon)$ choices for the quartic \'etale algebra $L$. The set of $\PGL_2(\Z)$-orbits $f$ such that the cubic resolvent polynomial of $f$ is $g$ and the \'etale quartic algebra corresponding to $f$ is $L$ injects into the set of quartic suborders $Q$ of $L$ with $\Disc(Q)=\Delta(g)$. In particular, the index of $Q$ in $\O_L$, the ring of integers of $L$, is determined by the pair $(g,L)$, and is equal to $k=(|\Delta(g)|/|\Disc(L)|)^{1/2}$. The previous lemma now implies that the number of possible options for $Q$ given $(g,L)$ is $\ll_\epsilon k^\epsilon N(k)\leq k^{1/2+\epsilon}$.

Fix $\delta>0$ to be optimized later. We will use partition the relevant set of forms $f$ into two sets: those whose corresponding \'etale quartic algebra $L$ satisfies $|\Disc(L)|\geq X^\delta$, and those for which $|\Disc(L)|<X^\delta$. Putting together the previous observations, we see that the number of $\PGL_2(\Z)$-orbits $f$ on $V(\Z)$ with $\Delta(f)\neq 0$, $H(f)<X$, $\Stab_{\PGL_2(\Q)}(f)\neq 1$, and such that the \'etale quartic algebra $L$ corresponding to $f$ satisfies $|\Disc(L)|\geq X^{\delta}$ is $\ll_\epsilon X^{1/2+\epsilon}\cdot X^{(1-\delta)/4}= X^{3/4-\delta/4+\epsilon}$.

Meanwhile, the number of quartic \'etale algebras $L$ with $|\Disc(L)|<X^\delta$ is bounded by $O(X^\delta\log X)$. Each algebra with discriminant $D$ has $\ll_\epsilon (X/D)^{1/2+\epsilon}$ suborders with discriminant less than $X$ (again using Nakagawa's result). This gives a bound of $O_\epsilon(X^{1/2+\delta/2+\epsilon})$ on the number of possible quartic orders with discriminant less than $X$ and whose \'etale algebras have discriminant bounded by $X^\delta$. Since a quartic order $Q$ occurs for an absolutely bounded number of $\PGL_2(\Z)$-orbits of binary quartic forms (see \cite[Proposition 2.16]{BS2Sel}), we see that the number of $\PGL_2(\Z)$-orbits $f$ on $V(\Z)$ with $\Delta(f)\neq 0$, $H(f)<X$, and such that the \'etale quartic algebra $L$ corresponding to $f$ satisfies $|\Disc(L)|< X^{\delta}$ is $\ll_\epsilon X^{1/2+\delta/2+\epsilon}$.
Optimizing, we pick $\delta=1/3$, yielding the result.
\end{proof}

\subsection{Slicing and secondary terms}

For ease of notation, we let $Y$ denote $X^{1/6}$ throughout this subsection. Let $\phi:V(\Z)\to\R$ be a $\PGL_2(\Z)$-invariant function. From \eqref{eq:avg}, we have
\begin{equation*}
N^{(i)}(\phi,X)=\frac{1}{\sigma_i\Vol(G_0)}\cI(\phi,Y),
\end{equation*}
where $\cI(\phi,Y)$ is defined by
\begin{equation}\label{eq:tcIdef}
\cI(\phi,Y):=\int_{t\geq \frac{\sqrt[4]{3}}{\sqrt{2}}}\int_{u\in N(t)}\Bigl(\sum_{f\in Y\cdot (u,t)S^{(i)}\cap V(\Z)^\gen}\phi(f)\Bigr) du\frac{d^\times t}{t^{2}},
\end{equation}
where we denote the bounded set $S_1^{(i)}$ by $S^{(i)}$. 
We break up the integral over $t$ above into the ``small $t$ range'' and the ``large $t$ range'' as follows.
Let $\psi_1:\R_{\geq 0}\to [0,1]$ be a smooth function such that $\psi(x)=1$ for $x\leq  2$ and $\psi_1(x)=0$ for $x\geq 3$, and denote $1-\psi_1$ by $\psi_2$. Let $0<\kappa<1/4$ be a real number to be optimized later. We define
\begin{equation}\label{eq:cIdef}
\begin{array}{rcl}
\displaystyle\cI^{(1)}(\phi,Y;\kappa)
&:=&\displaystyle
\int_{t\geq \frac{\sqrt[4]{3}}{\sqrt{2}}}\int_{u\in N(t)}\psi_1\Big(\frac{t}{Y^\kappa}\Big)\Bigl(\sum_{f\in Y\cdot (u,t)S^{(i)}\cap V(\Z)}\phi(f)\Bigr) du\frac{d^\times t}{t^{2}},
\\[.15in]
\displaystyle\cI^{(2)}(\phi,Y;\kappa)
&:=&\displaystyle
\int_{t\geq \frac{\sqrt[4]{3}}{\sqrt{2}}}\int_{u\in N(t)}\psi_2\Big(\frac{t}{Y^\kappa}\Big)\Bigl(\sum_{f\in Y\cdot (u,t)S^{(i)}\cap V(\Z):a(f)\neq 0}\phi(f)\Bigr) du\frac{d^\times t}{t^{2}},
\end{array}
\end{equation}
and note that Lemmas \ref{lem:red0}, \ref{lem:red1}, and \ref{lem:red2} imply the estimate
\begin{equation}\label{eq:II12}
\cI(\phi,Y)=
\cI^{(1)}(\phi,Y;\kappa)+\cI^{(2)}(\phi,Y;\kappa) +O_\epsilon(X^{2/3+\kappa/3}+X^{2/3+\epsilon}).
\end{equation}
We use the next result of Davenport to estimate the number of lattice points in bounded regions.
\begin{proposition}[\cite{MR43821}] \label{prop-davenport}
Let $\RR$ be a bounded, semi-algebraic multiset in $\R^n$ having maximum multiplicity $m$ that is defined by at most $k$ polynomial inequalities, each having degree at most $\ell$. Let $\RR'$ denote the image of $\RR$ under any $($upper or lower$)$ triangular, unipotent transformation of $\R^n$. Then the number of lattice points $($counted with multiplicity$)$ contained in the region $\RR'$ is given by
$$\Vol(\RR) + O (\max\{\overline{\Vol}(\RR)\},1\}),$$
where $\overline{\Vol}(\RR)$ denotes the greatest $d$-dimensional volume of any projection of $\RR$ onto a coordinate subspace obtained by equating $n-d$ coordinates to zero, where $d$ ranges over all values in $\{1, \dots, n-1\}$. The implied constant in the second summand depends only on $n$, $m$, $k$, and $\ell$.
\end{proposition}

For the rest of this subsection, we restrict our function $\phi$ to be defined via congruence classes modulo some positive integer $m$, and will keep track of the dependence of our error terms on $m$. We handle the cases of small and large $t$ separately.

\subsubsection*{The small $t$ case: estimates for $\cI^{(1)}(\phi,Y;\kappa)$}

Let $m$ be a positive integer, 
let $\overline{\phi}:V(\Z/m\Z)\to\R$ be a bounded (independent of $m$) function, and denote the lift of $\overline{\phi}$ to $V(\Z)$ by $\phi$. Recall the definitions of $\nu(\phi)$ and $\nu_a(\phi)$ from \eqref{eq:nuden}. We also define $\supp(\phi)$ to denote the size of the support of $\overline{\phi}$ in $V(\Z/m\Z)$. Then we have the following lemma.
\begin{lemma}\label{lem:ptcountMB}
Let $\phi:V(\Z)\to\R$ be an absolutely bounded function defined modulo a positive integer $m$. For $u\in[-1/2,1/2]$ and $t\gg 1$, we have
\begin{equation*}
\sum_{
f\in Y\cdot(u,t) S^{(i)}\cap V(\Z)}\phi(f)=
\nu(\phi)\Vol(S^{(i)})X^{5/6}+O\Bigl(\frac{\supp(\phi)t^4X^{2/3}}{m^4}+\frac{\supp(\phi)t^6X^{1/2}}{m^3}+\supp(\phi)t^6\Bigr).
\end{equation*}
\end{lemma}
\begin{proof}
Partition the support of $\phi$ in $V(\Z)$ into $\supp(\phi)$ translated lattices $L_v:=v+mV(\Z)$, where $v$ ranges over lifts of elements $\bar{v}$ in the support of $\overline{\phi}$ in $V(\Z/m\Z)$. The coefficients of an element $f\in Y\cdot(u,t) S^{(i)}$ satisfy 
\begin{equation*}
|a(f)|\ll t^{-4}Y;\quad
|b(f)|\ll t^{-2}Y;\quad
|c(f)|\ll Y;\quad
|d(f)|\ll t^{2}Y;\quad
|e(f)|\ll t^{4}Y.
\end{equation*}
Applying Proposition \ref{prop-davenport} to the set $(Y\cdot (u,t)S^{(i)}-v)/m$ yields
\begin{equation*}
\#\bigl(Y\cdot (u,t)S^{(i)}\cap L_v\bigr)=
\frac{\Vol(S^{(i)})Y^5}{m^{5}}+O\Bigl(\frac{t^4Y^4}{m^4}+\frac{t^6Y^3}{m^3}+\frac{t^6Y^2}{m^2}+\frac{t^4Y}{m}+1\Bigr).
\end{equation*}
Since we assume that $t\gg 1$, the first two summands in the above error term added to $t^6$ dominate the others. The lemma then follows by summing the above expression over $\bar{v}\in V(\Z/m\Z)$ that are in the support of $\overline{\phi}$.
\end{proof}

Estimating the inner sum of $f$ using Lemma \ref{lem:ptcountMB}, we obtain
\begin{equation}\label{eq:smallt}
\begin{array}{rcl}
\displaystyle\cI^{(1)}(\phi,Y;\kappa)&=&\displaystyle\nu(\phi)\Vol(S^{(i)})X^{5/6}\int_{t\geq \frac{\sqrt[4]{3}}{\sqrt{2}}}\int_{u\in N(t)}\psi_1\Bigl(\frac{t}{Y^{\kappa}}\Bigr)du\frac{d^\times t}{t^{2}}
\\[.2in]&&\displaystyle
+O\Bigl(\frac{\supp(\phi)X^{\frac23+\frac{\kappa}{3}}}{m^4}+\frac{\supp(\phi)X^{\frac12+\frac{2\kappa}{3}}}{m^3}+\supp(\phi)X^{\frac{2\kappa}{3}}\Bigr),
\end{array}
\end{equation}
where we recall that $Y=X^{1/6}$.

\subsubsection*{The large $t$ case: estimates for $\cI^{(2)}(\phi,Y;\kappa)$}

To evaluate the contribution from the large range of $t$, we fiber by the $x^4$-coefficients of integral binary quartic forms. That is, we write
\begin{equation}\label{eq:I2first}
\cI^{(2)}(\phi,Y;\kappa)=\displaystyle
\sum_{\substack{a\neq 0}}
\int_{t\geq \frac{\sqrt[4]{3}}{\sqrt{2}}}\int_{u\in N(t)}
\psi_2\Bigl(\frac{t}{Y^\kappa}\Bigr)\Bigl(\sum_{f\in Y\cdot (u,t)S^{(i)}\cap V_a(\Z)}\phi(f)\Bigr)du\frac{d^\times t}{t^2}.
\end{equation}
As before, we assume that $\phi$ is the lift of some function $\overline{\phi}:V(\Z/m\Z)\to\R$ which is absolutely bounded (independent of $m$). For any integer $a$ and any ring $R$, let $V_{a}(R)$ denote the set of elements in $V(R)$ whose $x^4$-coefficients are equal to $a$. For $a\in\Z$, recall that $\nu_a(\phi)$ denotes the density of $\phi$ restricted to $V_a(\Z)$, and for $a\in\Z$, let $\supp(\phi,a)$ denote the size of the support of $\overline{\phi}$ restricted to $V_a(\Z/m\Z)$.
For a set $T\subset V(\R)$, we let $T|_a$ denote $T\cap V_a(\R)$. Then
from a similar argument in the proof of Lemma \ref{lem:ptcountMB}, applying Proposition \ref{prop-davenport} yields the following result.
\begin{lemma}\label{lem:elemslct}
With notation as above, we have
\begin{equation*}
\sum_{f\in Y\cdot(u,t) S^{(i)}\cap V_a(\Z)}\phi(f)=\nu_a(\phi)t^4X^{2/3}\Vol\Bigl(S^{(i)}|_{t^4a/X^{1/6}}\Bigr)+O\Bigl(
\frac{\supp(\phi,a)t^6X^{1/2}}{m^3}+\supp(\phi,a)t^6\Bigr).
\end{equation*}
\end{lemma}

We return to \eqref{eq:I2first}, and note that for $a\neq 0$, the set $(Y\cdot(u,t) S^{(i)})|_a$ is empty unless $t\ll (Y/|a|)^{1/4}$. Since
$\psi_2(t/Y^\kappa)=0$ unless $t\gg Y^\kappa$,
only $a\neq0$ with $|a|\ll Y^{1-4\kappa}$ contributes to \eqref{eq:I2first}.
We truncate the sum over $a$ and integral over $t$ in \eqref{eq:I2first} using these observations, and apply Lemma \ref{lem:elemslct} to estimate the sum over $f$. Integrating the resulting error term over $t,u$, and summing over $a$ gives an error that is $\ll$
\begin{equation*}
\bigl(\frac{X^{1/2}}{m^3}+1\bigr)\sum_{\substack{a\neq 0\\|a|\ll Y^{1-4\kappa}}}\supp(\phi,a)\int_{t\ll (Y/|a|)^{1/4}}t^4d^\times t\ll \Bigl(\frac{X^{2/3}}{m^3}+X^{1/6}\Bigr)\sum_{\substack{a\neq 0\\|a|\ll Y^{1-4\kappa}}}\frac{\supp(\phi,a)}{|a|}.
\end{equation*}
Meanwhile, the main term contribution to $\cI^{(2)}(\phi,Y;\kappa)$ is 
\begin{equation}\label{eq:SMT1}
\begin{array}{rcl}
&& \displaystyle X^{2/3}\sum_{a\neq 0}\nu_a(\phi)\int_{t\geq 1}
\psi_2\Bigl(\frac{t}{Y^\kappa}\Bigr)\Vol\Big(S^{(i)}|_{t^4a/Y}\Big)t^2d^\times t
\\[.2in]&=&\displaystyle
\frac{X^{3/4}}{4}\sum_{a\neq 0}\frac{\nu_a(\phi)}{|a|^{\frac{1}{2}}}
\int_{\substack{\alpha\in\R^\times\\a\alpha>0}}
\psi_2\Bigl(\frac{|\alpha|^{\frac14}Y^{\frac14-\kappa}}{|a|^{\frac14}}\Bigr)\Vol\bigl(S^{(i)}|_\alpha\bigr)|\alpha|^{\frac12}d^\times\alpha\end{array}
\end{equation}
by the change of variables $\alpha=t^4a/Y$.

We next have the following lemma which follows immediately since $\nu_a(\phi)$
only depends on the residue class of $a$ modulo $m$.
\begin{lemma}
Define the functions 
\begin{equation*}
D^\pm(\phi,s):=\sum_{a>0}\frac{\nu_{\pm a}(\phi)}{a^s}.
\end{equation*}
Then the functions $D^\pm(\phi,s)$ admit meromorphic continuations to $\C$, and are holomorphic with the exception of possible simple poles at $s=1$ with residue $\nu(\phi)$.
\end{lemma}

Then with a computation identical to \cite[(Equation 41)]{MR3090184}, we obtain for any $M>1$
\begin{equation}\label{eq:suma}
\begin{array}{rcl}
\displaystyle\sum_{\pm a> 0}\frac{\nu_a(\phi)}{|a|^{\frac12}}
\psi_2\Bigl(\frac{|\alpha|^{\frac14}Y^{\frac{1}{4}-\kappa}}{|a|^{\frac14}}\Bigr)&=&
\displaystyle
D^\pm(\phi,1/2)+4\nu(\phi)\widetilde{\psi_2}(-2)(|\alpha|Y^{1-4\kappa})^{\frac12}
\\[.2in] &&\displaystyle
+O_M(\min((|\alpha|Y^{1 -4\kappa})^{-M},1)).
\end{array}
\end{equation}
Above, $\widetilde{\psi_2}$ denotes the Mellin transform of $\psi_2$.
We will choose $\kappa<1/4$ and so the error term is negligible. The contribution of the second summand of \eqref{eq:suma} to the right hand side of \eqref{eq:SMT1} is equal~to
\begin{equation*}
\nu(\phi) X^{5/6}\widetilde{\psi_2}(-2)Y^{-2\kappa}\int_{\alpha\in\R^\times}
\Vol\bigl(S^{(i)}|_\alpha\bigr)d\alpha
=\nu(\phi)\Vol(S^{(i)})X^{5/6}\int_{t>0}\psi_2\Bigl(\frac{t}{Y^\kappa}\Bigr)t^{-2}d^\times t,
\end{equation*}
where the equality follows from the definition of the Mellin transform.
Finally, the contribution of the first summand in \eqref{eq:suma} to \eqref{eq:SMT1} is
\begin{equation*}
\begin{array}{rcl}
&&\displaystyle\frac{1}{4}\Bigl(D^+(\phi,1/2)\mathcal{V}^+(S^{(i)})+D^-(\phi,1/2)\mathcal{V}^-(S^{(i)})\Bigr)X^{3/4},\mbox{ where }
\\[.2in]&&\displaystyle
\quad \quad\mathcal{V}^\pm(S^{(i)}):=
\int_{\pm \alpha>0}
\Vol(S^{(i)}|_\alpha)\frac{1}{|\alpha|^{1/2}}d\alpha.
\end{array}
\end{equation*}
Therefore, we have
\begin{equation}\label{eq:larget}
\begin{array}{rcl}
\cI^{(2)}(\phi,Y;\kappa)&=&\displaystyle
\nu(\phi)\Vol(S^{(i)})X^{5/6}\int_{t>0}\psi_2\Bigl(\frac{t}{Y^\kappa}\Bigr)t^{-2}d^\times t
\\[.25in]&&+\displaystyle
\frac{1}{4}\Bigl(D^+(\phi,1/2)\mathcal{V}^+(S^{(i)})+D^-(\phi,1/2)\mathcal{V}^-(S^{(i)})\Bigr)X^{3/4}
\\[.25in]&&+\displaystyle
O\left(
\Bigl(\frac{X^{2/3}}{m^3}+X^{1/6}\Bigr)\sum_{0\neq |a|\ll Y^{1-4\kappa}}\frac{\supp(\phi,a)}{|a|}\right).
\end{array}
\end{equation}
We define the quantities
\begin{equation*}
M_{{5/6}}^{(i)}(\phi):=\frac{\nu(\phi)\Vol(\FF)\Vol(S^{(i)})}{\sigma_i\Vol(G_0)};\quad 
M_{ {3/4}}^{(i)}(\phi):=\frac{D^+(\phi,1/2)\mathcal{V}^+(S^{(i)})+D^-(\phi,1/2)\mathcal{V}^-(S^{(i)})}{4\sigma_i\Vol(G_0)}.
\end{equation*}
Adding the right hand sides of \eqref{eq:smallt} and \eqref{eq:larget} yields
\begin{equation}\label{eq:Shintanifinal}
N^{(i)}(\phi,X)= M_{{5/6}}^{(i)}(\phi)X^{5/6}+M_{{3/4}}^{(i)}(\phi)X^{3/4}
+O_\epsilon(X^{2/3+\epsilon}+X^{2/3+\kappa/3}+ E(\phi;X;\kappa)),
\end{equation}
for every $\kappa$, where the error term $E(\phi;X)$ is given by
\begin{equation*}
E(\phi;X;\kappa)=\Bigl(\frac{X^{\frac23+\frac{2\kappa}{3}}}{m^4}+\frac{X^{\frac12+\frac{4\kappa}{3}}}{m^3}+
X^{\frac{4\kappa}{3}}\Bigr)\supp(\phi)+
\Bigl(\frac{X^{\frac23}}{m^3}+X^{\frac{1}{6}}\Bigr)\!\!\sum_{\substack{0\neq a\\|a|\ll Y^{1-4\kappa}}}\frac{\supp(\phi,a)}{|a|}.
\end{equation*}
For any fixed $m$ and $\phi$, \eqref{eq:Shintanifinal} holds for all positive $X$. This automatically implies that $M_1^{(i)}(\phi)$ and $M_2^{(i)}(\phi)$ are independent of the choice of $G_0$.

\medskip

We are now ready to prove the analogues, Theorems \ref{Th:Shintani4} and \ref{th:mainCongCountFin}, of Shintani's result in the case of binary quartic forms.

\medskip

\noindent {\bf Proof of Theorem \ref{Th:Shintani4}:} The theorem requires us to provide a secondary term in the estimate of $N^{(i)}(\phi,X)$ when $\phi(f)=1$ for all $f$. In this case, we have $m=1$, $\nu(\phi)=\nu(a,\phi)=1$ for all $a$ and hence $D^\pm(\phi,s)=\zeta(s)$. The result now follows by taking $\kappa=\epsilon$ in \eqref{eq:Shintanifinal} in conjunction with the computations of $M_{{5/6}}^{(i)}(\phi)$ and $M_{{ 3/4}}^{(i)}(\phi)$ in \S8. $\Box$
\medskip

\noindent {\bf Proof of Theorem \ref{th:mainCongCountFin}:} For this result, we take $m=\prod_{p\in S}p$ and write $\phi:V(\Z/m\Z)\to\R$ as $\prod_{p\in S}\phi_p$, where $\phi_p:V(\Z/p\Z)\to \R$ is the characteristic function of the set of elements having splitting type $\sigma_p$. The value of the main term follows from \eqref{eq:Shintanifinal}. The value of the secondary term follows in conjunction with the special values of $D(\phi,1/2)$ computed in Table \ref{tablenu}. The error term is easily seen to also follow from \eqref{eq:Shintanifinal}. $\Box$

\begin{remark}{\rm
We note that the error terms for Theorem \ref{th:mainCongCountFin} in the $\Sigma$-aspects can be improved using the equidistribution results we prove in \S6, specifically, Corollary \ref{prop:FT}. In fact, since the functions $\phi_p$ under consideration are more than strongly invariant--they are invariant under multiplication by all units (not merely squares) in $\F_p$--we can get even better error bounds by improving Corollary \ref{prop:FT}. Finally, for many applications, it is necessary to take $\phi_p$ to be certain natural linear combinations of characteristic functions of elements with a particular splitting types. For such functions, it is often the case that further cancellations in the Fourier transform, leading to further improvements to the error terms, may be obtained. For some such examples, see \cite{ITTX}, \cite{BTT}, \cite{vanderwarden}.
}\end{remark}

\section{Constructing periodic approximations of local functions}\label{sec:locsol}

In this section, we prove that $m_p$ and $\ell_p$ are and almost well approximated by periodic functions (and well approximated except in the case of $\ell_2$). We do this by explicitly constructing the periodic approximations $m_p^{(k)}$ and $\ell_p^{(k)}$ in \S5.1 and \S5.3, respectively. In \S5.2, for $p\geq 5$, we give a natural interpretation of $m_p^{(k)}$ using the Bruhat--Tits tree of $\PGL_2(\Q_p)$. Finally, at the end of \S5.3, we prove that $\ell_p/m_p$ is large and almost well appoximated, and well approximated for $p\geq 3$.

\subsection{Approximating the number of $\PGL_2(\Z_p)$-orbits in a $\PGL_2(\Q_p)$-equivalence class}

\begin{proposition}\label{prop:mpgood}
    The function $m_p$ is well approximated by periodic functions.
\end{proposition}
\begin{proof}
For $k\geq0$, let
$\mathcal C^{(k)}=\PGL_2(\Z_p)
\footnotesize\begin{pmatrix}p^k&0\\0&1\end{pmatrix}
	\PGL_2(\Z_p)\subset \PGL_2(\Q_p)$,
where we use the same notation
$\footnotesize\begin{pmatrix}p^k&0\\0&1\end{pmatrix}$
for an element in $\GL_2(\Q_p)$ and its image in $\PGL_2(\Q_p)$.
Then it is well known that
$\PGL_2(\Q_p)=\bigsqcup_{k\geq0}\mathcal C^{(k)}$ and that
the size of $\PGL_2(\Z_p)\backslash \mathcal C^{(k)}$ is finite for each $k$.
For $f\in V(\Z_p)$, we put
$\mathcal C^{(k)}_f=\{g\in\mathcal C^{(k)}\mid g\cdot f\in V(\Z_p)\}$ and define $m_p^{(k)}(f)$
to be the size
of $\PGL_2(\Z_p)\backslash \mathcal C^{(k)}_f$.
Then $m_p=\sum_{k\geq0}m_p^{(k)}$
and $m_p^{(0)}$ is identically $1$.

We show that the function $m_p^{(k)}$ is periodic with period $p^{2k}$. Suppose $f_1,f_2\in V(\Z_p)$ are congruent modulo $p^{2k}$. We let $h=f_1-f_2\in p^{2k}V(\Z_p)$. For $g\in \mathcal C^{(k)}$, $g\cdot h\in V(\Z_p)$, so
$g\cdot f_1\in V(\Z_p)$
if and only if $g\cdot f_2\in V(\Z_p)$.
Therefore $\mathcal C^{(k)}_{f_1}=\mathcal C^{(k)}_{f_2}$,
implying $m_p^{(k)}(f_1)=m_p^{(k)}(f_2)$.

Suppose that $m_p^{(k)}(f)\neq0$. Then
$\mathcal C^{(k)}_f\neq\emptyset$.
Take $g=\gamma_1\footnotesize\begin{pmatrix}p^k&0\\0&1\end{pmatrix}\gamma_2\in \mathcal C^{(k)}_f$, for elements $\gamma_1,\gamma_2\in \PGL_2(\Z_p)$ and let $f'=\gamma_2 f\in V(\Z_p)$.
Then $\footnotesize\begin{pmatrix}p^k&0\\0&1\end{pmatrix}f'\in V(\Z_p)$.
Let $f'=(a_0,b_0,c_0,d_0,e_0)$. Then
$\footnotesize\begin{pmatrix}p^k&0\\0&1\end{pmatrix}f'=
(a_0/p^{2k},b_0/p^k,c_0,d_0p^k,e_0p^{2k})$
and so $p^{2k}\mid a_0$ and $p^k\mid b_0$.
By \eqref{eq:disc-cong-cubic}, this implies $p^{2k}\mid \Delta(f')$.
Therefore  $p^{2k}\mid \Delta(f)$ as well since $\Delta(f)=\Delta(f')$.
\end{proof}

\subsection{Counting nodes in the Bruhat--Tits tree of $\PGL_2(\Q_p)$}

Let $p\geq 5$ be prime. In this subsection, which is unnecessary for the proofs of the main results, we give an alternative and more natural description of $m_p^{(k)}$ in terms of the Bruhat--Tits tree $\mathcal{T}=(\mathcal{V},\mathcal{D})$ of $\PGL_2(\Q_p)$. Recall that $\mathcal{T}$ is an undirected simple graph whose nodes $\mathfrak{n}\in\mathcal{V}$ correspond to homothety classes of lattices in $\Q_p^2$, where we recall that two lattices $L_1$ and $L_2$ are {\it homothetic} if there exists $\theta\in\Q_p^\times$ with $\theta L_1=L_2$.
Given a pair $(L,L')$ of lattices in $\Q_p^2$, there exist integers $a$ and $b$ and a basis $\{ v_1,v_2\}$ of $L$ such that $\{ p^{a}v_1,p^bv_2\}$ is a basis of $L'$. Then it is easy to check that $a$ and $b$ are invariants of the pair $(L,L')$, and that $|a-b|$ remains invariant even when we replace $L$ and $L'$ by homothetic lattices. This yields a function
\begin{equation*}
\inv:\mathcal{V}\times\mathcal{V}\to\Z_{\geq 0}.
\end{equation*}
Moreover, $\inv(\mathfrak{n}_1,\mathfrak{n_2})=0$ if and only if $\mathfrak{n}_1=\mathfrak{n_2}$. We say that two nodes $\mathfrak{n}_1,\mathfrak{n_2}\in\mathcal{V}$ are {\it neighbors} if $\inv(\mathfrak{n}_1,\mathfrak{n_2})=1$, and there is an edge in $\mathcal{D}$ between two nodes if and only if they are neighbors. 
It is well known that this set of vertices and edges makes $\mathcal{T}$ into a regular tree, where each vertex has degree $p+1$. Furthermore, the distance between two nodes $\mathfrak{n}_1$ and $\mathfrak{n_2}$ is $\inv(\mathfrak{n}_1,\mathfrak{n_2})$.

The group $\GL_2(\Q_p)$ acts on the set of lattices in $\Q_p^2$ via $\gamma\cdot L:=\{\gamma v:v\in L\}$. This action descends to an action of $\PGL_2(\Q_p)$ on the set of homothety classes of lattices, and hence an action of $\PGL_2(\Q_p)$ on $\mathcal{V}$. This action is transitive since the aforementioned action of $\GL_2(\Q_p)$ is transitive. Denote the node corresponding to the homothety class of the lattice $\Z_p^2$ by $\mathfrak{o}$.
Then it is easy to see that the stabilizer in $\PGL_2(\Q_p)$ of $\mathfrak{o}$ is $\PGL_2(\Z_p)$. This gives a bijection
\begin{equation*}
\Phi:\PGL_2(\Q_p)/\PGL_2(\Z_p)\to \mathcal{V},
\end{equation*}
sending the $\PGL_2(\Z_p)$-orbit of $\gamma\in\PGL_2(\Q_p)$ to the homothety class of $\gamma\cdot\mathfrak{o}$.
For an integer $k\geq 1$, let $\mathfrak{o}_k$ denote the node corresponding to the homothety class of $p^k\Z_p\oplus\Z_p$. Then $\mathfrak{o}_k$ has distance $k$ from $\mathfrak{o}$. It is well known that $\PGL_2(\Z_p)$ acts transitively on the set of nodes having distance $k$ from $\mathfrak{o}$, and hence every node having distance $k$ from $\mathfrak{o}$ is $\PGL_2(\Z_p)$-equivalent to $\mathfrak{o}_k$.

To describe the connection between $\mathcal{T}$ and binary quartic forms, we define the following notion: Denote the determinant of $L$ by $d(L)$. Then we say that $f\in V(\Q_p)$ is {\it integral-valued with respect to $L$} if $d(L)^2\mid f(v)$ for every $v\in L$. Since this notion is homothety invariant, it descends to a notion of $f$ being {\it integral valued} with respect to nodes in $\mathcal{V}$. We next prove that this notion of integrality respects the action of $\PGL_2(\Q_p)$.
\begin{proposition}
Let $f\in V(\Q_p)$, $\mathfrak{n}\in\mathcal{V}$ and $\gamma\in\PGL_2(\Q_p)$, be any elements. Then $f$ is integral-valued with respect to $\gamma\cdot \mathfrak{n}$ if and only if $\gamma\cdot f$ is integral with respect to $\mathfrak{n}$.
\end{proposition}
\begin{proof}
Let $L\subset\Q_p^2$ and $\widetilde{\gamma}\in\GL_2(\Q_p)$ be representatives of $\mathfrak{n}$ and $\gamma$, respectively. By definition, $\gamma\cdot f$ is integral-valued with respect to $\mathfrak{n}$ if and only if $d(L)^2\mid (\gamma\cdot f)(x,y)$ for every $(x,y)\in L$. However, we have $(\gamma\cdot f)(x,y)=f((x,y)\cdot\widetilde{\gamma})/(\det(\widetilde{\gamma}))^2$. Therefore, we have that $\gamma\cdot f$ is integral-valued with respect to $\mathfrak{n}$ if and only if $(\det(\widetilde{\gamma})d(L))^2\mid f(x,y)$ for every $(x,y)\in \widetilde{\gamma}L$, which is true if and only if $f$ is integral-valued with respect to $\gamma\cdot\mathfrak{n}$, as necessary.
\end{proof}

We next have the following consequence describing the image of $\PGL_2(\Q_p)$ under $\Phi$.
\begin{corollary}\label{cor:integralnodes}
Let $f\in V(\Q_p)$ be any element.
Then $\gamma\in\PGL_2(\Q_p)$ belongs to $\PGL_2(\Q_p)_f$ $($equivalently, $\gamma\cdot f\in V(\Z_p))$ if and only if $f$ is integral with respect to $\Phi(\gamma)$.
\end{corollary}
\begin{proof}
We know that $f$ is integral with respect to $\Phi(\gamma)=\gamma\cdot\mathfrak{o}$ if and only if $\gamma\cdot f$ is integral with respect to $\mathfrak{o}$ which is true if and only if $\gamma\cdot f\in V(\Z_p)$ as necessary.
\end{proof}

\noindent It is only here that we use $p\geq 5$. It is needed since for $p=2$ and $p=3$, there exist elements $f\in V(\Q_p)\backslash V(\Z_p)$ for which $f(v)\in\Z_p$ for all $v\in\Z_p$.

\medskip

Finally, denote the set of nodes in $\mathcal{V}$ with respect to which $f\in V(\Q_p)$ is integral by $\mathcal{V}(f)$.
An immediate consequence of Corollary \ref{cor:integralnodes} is that $m_p(f)$ is equal to $\#\mathcal{V}(f)$. For an integer $k\geq 0$, let $\mathcal{V}_k$ denote the set of nodes in $\mathcal{V}$ that have distance $k$ from $\mathfrak{o}$. It is easy to check that $m_p^{(k)}=\#(\V(f)\cap\V_k)$.

\subsection{Approximating local solubility}\label{sec:alt-locsol}

Let $p$ be a prime and recall that we defined $\ell_p:V(\Z_p)\to\{0,1\}$ to be the characteristic function of the set of elements in $V(\Z_p)$ that are soluble. In this section, we prove that $\ell_p$ is well approximated by periodic functions if $p$ is odd, and is almost well approximated by periodic functions if $p=2$.

We first show the following:

\begin{proposition}\label{prop:mainls-equiv}
Let $k\geq0$.
\begin{enumerate}
\item Suppose $p$ is odd.
If $f\in V(\Z_p)$ is insoluble
but there exist a soluble $f_0\in V(\Z_p)$
such that $f\equiv f_0\pmod{p^{2k}}$,
then $p^{2k+2}\mid \Delta(f)$.
\item Suppose $p=2$.
If $f\in V(\Z_p)$ is insoluble
but there exist a soluble $f_0\in V(\Z_p)$
such that $f\equiv f_0\pmod{p^{2k+2}}$,
then $p^{2k+2}\mid \Delta(f)$.
\end{enumerate}
\end{proposition}

For the proof, we begin with some preliminary results.
\begin{lemma}\label{lem:hensel}
Let $g(x)\in\Z_p[x]$ and $a\in\Z_p$.
Let $\lambda=\ord_p(g(a))$ and
$\mu=\ord_p(g'(a))$.
If $\lambda>2\mu$, then
$\{g(a+\xi)\mid \xi\in p^{\lambda-\mu}\Z_p\}=p^\lambda\Z_p$.
\end{lemma}
This is a version
of Hensel's lemma.
We have $g(a+\xi)=g(a)+\xi g'(a)+\xi^2 h(a,\xi)$
for a polynomial $h(x,y)\in\Z_p[x,y]$, and the proof is standard with this identity.
Next, we have the following lemma, which will also be used in the next section, regarding congruences for the discriminant polynomial
$\Delta(f)$ for $f=(a,b,c,d,e)\in V$:
\begin{lemma}\label{lem:disc-cong}
We have
\begin{align}
\Delta(f)&\equiv
    b^2(c^2d^2+18bcde-4bd^3-4c^3e-27b^2e^2)\pmod{(ae,ad^2)},\label{eq:disc-cong-cubic}\\
\Delta(f)&\equiv 4ac^3(4ce-d^2)\pmod{(a^2,ab,b^2)},\label{eq:disc-cong-quad}
\\
\Delta(f)&\equiv0\pmod{(a^2,abc,abd,ac^2,b^4,b^3d,b^2c^2)}.\label{eq:disc-cong-abcd}
\end{align}
\end{lemma}
This immediately follows from the concrete form
the discriminant:
\begin{align*}
    \Delta(f)=&    256 a^3 e^3\\
    &- 192 a^2 b d e^2 - 128 a^2 c^2 e^2
    + 144 a^2 c d^2 e - 27 a^2 d^4\\
    &+ 144 a b^2 c e^2 - 6 a b^2 d^2 e - 80 a b c^2 d e + 18 a b c d^3 + 16 a c^4 e - 4 a c^3 d^2\\
    &- 27 b^4 e^2 + 18 b^3 c d e - 4 b^3 d^3 - 4 b^2 c^3 e + b^2 c^2 d^2.
\end{align*}

As an immediate consequence, we obtain the following:

\begin{lemma}\label{lem:discdiv}
Let $k\geq1$.
Suppose the coefficients
of $f(x,y)=ax^4+bx^3y+cx^2y^2+dxy^3+ey^4\in V(\Z_p)$  satisfy one of the following divisibility properties:
\begin{itemize}
\item $p^{2k+1}\mid a$, $p^{k+1}\mid b$, $p\mid d$, and $p\mid e$;
\item $p\mid b$, $p\mid c$, $p^{k}\mid d$, and $p^{2k}\mid e$.
\end{itemize}
Then $p^{2k+2}\mid\Delta(f)$.
\end{lemma}
\begin{proof}
This follows from \eqref{eq:disc-cong-cubic} and \eqref{eq:disc-cong-abcd}.
Note that $\Delta(a,b,c,d,e)=\Delta(e,d,c,b,a)$.
\end{proof}

We now prove Proposition \ref{prop:mainls-equiv}.

\medskip

\begin{proof}
Recall from the proof of \cite[Proposition 3.18]{BS2Sel}
that if $f$ is insoluble,
the splitting type of $f\pmod p$
is either $(1^21^2)$, $(2^2)$, $(1^4)$
or $(0)$. In particular $p^2\mid \Delta(f)$
so we have the assertion for $k=0$.
For the rest of the proof we assume $k\geq1$.
By suitably multiplying an element in $\Q_p^\times$ to the variables and
further transforming the variables by an elemnt in $\PGL_2(\Z_p)$, we may assume that $e_0=f_0(0,1)$
is a squared element.
Let $f(x,y)=ax^4+bx^3y+cx^2y^2+dxy^3+ey^4$.
Then $e\equiv e_0\pmod {p^{2k}}$.
If $p^{2k}\nmid e_0$,
then the congruence condition
implies that
$e/e_0\in 1+p\Z_p$ if $p$ is odd and
$e/e_0\in 1+8\Z_2$ if $p=2$.
In particular $e/e_0=u^2$ for some $u\in\Z_p^\times$
and so $e=u^2e_0$ is a squared element.
This contradicts to $f$ being not locally soluble.
Thus $p^{2k}\mid e_0$
and so $p^{2k}\mid e$.
If $p^k\nmid d$,
then Lemma \ref{lem:hensel}
implies that $f$ is locally soluble,
again a contradiction. Thus $p^k\mid d$.

Since $f$ is insoluble,
the splitting type of $f\pmod p$
is either $(1^21^2)$, $(2^2)$, $(1^4)$
or $(0)$. If the splitting type is $(2^2)$,
then $e=f(0,1)$ must be in $\Z_p^\times$, so this can not occur.
If the splitting type is either $(1^4)$
or $(0)$, then $p\mid b$ and $p\mid c$,
so by Lemma \ref{lem:discdiv},
$p^{2k+2}\mid \Delta(f)$ as needed.
Suppose the splitting type of $f$ is $(1^21^2)$.
Applying a $\PGL_2(\Z_p)$-transformation if necessary,
we may assume that the coefficients
of $f$ satisfy $p\mid a$, $p\mid b$, $p\nmid c$, $p^{k}\mid d$, and $p^{2k}\mid e$. Since $f$ is insoluble and local solubility is a $\PGL_2(\Q_p)$-invariant property, it follows that $f_1=\tiny{\begin{pmatrix}p^k &\\&1 \end{pmatrix}} f$ is insoluble. Now the coefficients $(a_1,b_1,c_1,d_1,e_1)=(p^{2k}a,p^{k}b,c,d/p^k,e/p^{2k})$ are all integral and satisfy $p^{2k+1}\mid a_1$, $p^{k+1}\mid b_1$, and $p\nmid c_1$. Since $f_1$ is insoluble, the splitting type of $f_1$ must again be $(1^21^2)$ (the only other options, $(1^22)$ or $(1^211)$, would imply that $f_1$ is soluble). Therefore, applying another $\PGL_2(\Z_p)$-transformation, we may assume that in addition we also have $p\mid d_1$ and $p\mid e_1$. Applying Lemma \ref{lem:discdiv}, it follows that  $p^{2k+2}\mid\Delta(f_1)=\Delta(f)$ as necessary.
\end{proof}

We are now ready to prove the following:
\begin{proposition}\label{prop:mainls}
If $p$ is an odd prime, then the function $\ell_p$ is well approximated by periodic functions.
If $p=2$, then the function $\ell_p$ is
almost well approximated by periodic functions.
\end{proposition}

\begin{proof}
Let $k\geq0$. For $f\in V(\Z_p)\setminus\{\Delta=0\}$,
we let $L_p^{(k)}(f)=0$ if any $g\in V(\Z_p)$ satisfying $g\equiv f\pmod{p^{2k}}$
is not locally soluble,
and $L_p^{(k)}(f)=1$ otherwise.
Then we define the functions
$\ell_p^{(k)}$ on $V(\Z_p)$ by setting $\ell_p^{(0)}:=L_p^{(0)}$ and
$\ell_p^{(k)}:=L_p^{(k)}-L_p^{(k-1)}$ for $k\geq 1$.
By definition $L_p^{(k)}$
is periodic with period $p^{2k}$
and is $\PGL_2(\Z_p)$-invariant,
so is $\ell_p^{(k)}$.
We show that
$\ell_p$ is well approximated by $\ell_p^{(k)}$ for odd $p$,
and almost well approximated by $\ell_p^{(k)}$ for $p=2$.

By definition, $\ell_p^{(0)}=L_p^{(0)}$ is identically $1$,
and $L_p^{(k)}(f)\geq L_p^{(k+1)}(f)$ for all $k$ and $f$.
If $f$ is soluble then $L_p^{(k)}(f)=1$ for all $k$,
meanwhile if $f$ is insoluble and $p^{2k+2}\nmid\Delta(f)$,
then Proposition \ref{prop:mainls-equiv} implies that
$L_p^{(k)}(f)=0$ if $p$ is odd, and
$L_p^{(k+1)}(f)=0$ if $p=2$.
Thus $\lim_{k\to\infty}L_p^{(k)}(f)=\ell_p(f)$ for all $f\in V(\Z_p)\setminus\{\Delta=0\}$.
Let $k\geq1$, and suppose $\ell_p^{(k)}(f)\neq0$.
This happens only when
$L_p^{(k-1)}(f)=1$ and $L_p^{(k)}(f)=0$.
Then Proposition \ref{prop:mainls-equiv}
implies that
$p^{2k}\mid \Delta(f)$ if $p$ is odd, and
$p^{2k-2}\mid \Delta(f)$ if $p=2$.
\end{proof}

\medskip

Finally, we deduce that $\ell_p/m_p$ is well approximated for all odd primes, and almost well approximated for $p=2$. To do this, we prove that the product of two  (almost) well approximated functions is (almost) well approximated, and that under certain conditions, the inverse of a large and (almost) well approximated function is also (almost) well approximated.

\begin{lemma}\label{lem:lawainv}
Suppose that $\phi:V(\Z_p)\backslash\{\Delta=0\}\to\R$ and $\phi':V(\Z_p)\backslash\{\Delta=0\}\to\R$ are large and $($almost$)$ well approximated. Then their product $\phi\phi'$ is also large and $($almost$)$ well approximated. 

Suppose $\phi:V(\Z_p)\backslash\{\Delta=0\}\to\R_{>0}$ is bounded away from $0$, large and $($almost$)$ well approximated via the series of functions $\phi^{(k)}$. Let $\Phi^{(k)}$ denote the partial sums of $\phi^{(k)}$, i.e., $\Phi^{(k)}:=\sum_{n=0}^k\phi^{(k)}$, and assume that the $\Phi^{(k)}$ are also bounded away from $0$ by an absolute constant. That is, we have $\Phi^{(k)}(f)\geq c>0$ for all $f$, where $c$ is an absolute constant. Then $1/\phi$ is large and $($almost$)$ well approximated.
\end{lemma}
\begin{proof}
We define the partial sums $\Phi^{(k)}:=\sum_{n=0}^k\phi^{(k)}$ and $\Phi'^{(k)}:=\sum_{n=0}^k\phi'^{(k)}$; define $\Psi^{(k)}:=\Phi^{(k)}\Phi'^{(k)}$; and finally define $\psi^{(0)}$ to be identically $1$ and $\psi^{(k)}:=\Psi^{(k)}-\Psi^{(k-1)}$ for $k\geq 1$.
It is then easy to check that $\phi\phi'$ is large and (almost) well approximated via the functions $\psi^{(k)}$.

Consider the functions $\Psi^{(k)}=(\Phi^{(k)})^{-1}$, and define $\psi^{(0)}$ to be identically $1$, and $\psi^{(k)}:=\Psi^{(k)}-\Psi^{(k-1)}$ for $k\geq 1$. We claim that $1/\phi$ is large and (almost) well approximated via the functions $\psi^{(k)}$. Indeed, Properties (a) and (b) are immediately seen to be satisfied. To verify Property (c) in the case when $\phi$ is well approximated, note that if $f\in V(\Z)$ is an element with $p^{2k}\nmid\Delta(f)$, then $\phi^{(k)}(f)=0$, which implies that $\Phi^{(k)}(f)=\Phi^{(k-1)}(f)$, which in turn implies that $\psi^{(k)}(f)=1/\Phi^{(k)}(f)-1/\Phi^{(k)}(f)=0$ as necessary. The proof in the case when $\phi$ is almost well approximated is identical.
\end{proof}

Suppose $\chi:\Z_p^2\backslash\{\Delta\neq 0\}\to\R$ is large and (almost) well approximated by periodic functions. Then it is clear that the function $\Inv_\chi:=\chi\circ\Inv:V(\Z_p)\backslash\{\Delta=0\}\to\R$ is also (almost) well approximated. Here, for a ring $R$, the function $\Inv:V(R)\to R^2$ is defined by setting $\Inv(f)=(I(f),J(f))$.
We then have the following consequence of Propositions \ref{prop:mpgood}, \ref{prop:mainls}, and Lemma \ref{lem:lawainv}
\begin{corollary}
Let $\chi:\Z_p^2\backslash\{\Delta\neq 0\}\to\R$ be $($almost$)$ well approximated by periodic functions. Then $\Inv_\chi\cdot \ell_p/m_p$ is large and $($almost$)$ well approximated by periodic functions when $p\geq 3$, and almost well approximated when $p=2$.
\end{corollary}

\section{Density estimates and Fourier analysis on $V(\Z/n\Z)$}\label{sec:FT}

As we will see in \S8, the error terms obtained in Section 4 are insufficiently strong for the purposes of summing $\PGL_2(\Z)$-orbits on irreducible binary quartic forms weighted by a locally well approximated function. In fact, the obtained error terms in both the main body and the cuspidal region are insufficiently strong. The main goals of this section are: 1: to prove that the previously introduced Dirichlet series $D^\pm(\phi,s)$ have analytic continuation to the left of $\Re(s)=1$, so as to show that their values at $1/2$ are well defined, and 2: to obtain cancellation in the Fourier transforms of $\PGL_2(\Z/p^2\Z)$-invariant functions on $V(\Z/p^2\Z)$. These results will then be combined in \S 8 with equidistribution techniques to improve the main body count.

This section is organized as follows. In Section 6.1, we obtain upper bounds on the density of the set of elements in $V(\Z_p)$ whose discriminants are divisible by $p^{2k}$ for $k\geq 1$. This will allow us to analytically continue the Dirichlet series $D^\pm(\phi,s)$ to the left of $\Re(s)=1$. Then in Section 6.2, we obtain nontrivial cancellation in the Fourier transform of $\PGL_2(\Z/p^2\Z)$-invariant functions on $V(\Z/p^2\Z)$.

For a function $\phi\colon V(\Z_p)\rightarrow\R$, similar to \eqref{eq:nuden},
we define
\begin{equation}\label{eq:nuden_l}
\nu(\phi):=\int_{V(\Z_p)}\phi(v)dv;\quad\quad
\nu_a(\phi):=\int_{V_a(\Z_p)}\phi(v)dv.
\end{equation}
Above, the measures $dv$ are normalized so that $V(\Z_p)$ and $V_a(\Z_p)$ have volume $1$.

\subsection{Bounds on the density of elements in $V(\Z_p)$ with small discriminant}

For $k\geq 1$, let $\chi_{p^{2k}}:V(\Z/p^{2k}\Z)\to\R$ denote the characteristic function of any set of elements $f\in V(\Z/p^{2k}\Z)$ with $\Delta(f)=0$. In this subsection, we begin by proving the following result.

\begin{proposition}\label{prop:ppden}
We have
\begin{equation*}
\nu(\chi_{p^2})\ll\frac{1}{p^2},\quad\quad \nu(\chi_{p^4})\ll\frac{1}{p^4},\quad\quad
\nu(\chi_{p^{2k}})\ll\frac{k}{p^{3k/2}},
\end{equation*}
for $k\geq 3$.
\end{proposition}
\begin{proof}
We begin with the case $k\geq 3$. Fix $0\neq a\in \Z_p$, and denote the set of binary quartic forms $f\in V(\Z_p)$ with $a(f)=a$ by $V_a(\Z_p)$. Let $S_a(2k)\subset V_a(\Z_p)$ denote the set of elements $f\in V_a(\Z_p)$ with $p^{2k}\mid\Delta(f)$. To obtain a bound on the density of $S_a(2k)$, we consider the map
\begin{equation*}
\begin{array}{rcl}
\mon_a: V_a(\Z_p)&\to& V_1(\Z_p)\\[.05in]
ax^4+bx^3y+cx^2y^2+dxy^3+ey^4&\mapsto&
x^4+bx^3y+acx^2y^2+a^2dxy^3+a^3ey^4.
\end{array}
\end{equation*}
The Jacobian change of variables of this map is clearly $a^6$, and it is easily seen that $\Delta(\mon_a(f))=a^6\Delta(f)$. If $p^\ell\parallel a$, then it follows that for $g(x,y)\in\mon_a(S_a(2k))$, we have $p^{2k+6\ell}\mid\Delta(g)$. In particular, we have $\mon_a(S_a(2k))\subset S_1(2k+6\ell)$.

It is proved in \cite[Lemma 5.1]{2204.01651} that the density of $S_1(2k+6\ell)$ is $\ll p^{-3k/2-9\ell/2}$.
We claim that the density of $\mon_a(S_a(2k))\subset S_1(2k+6\ell)$ is in fact smaller, and bounded by $\ll p^{-3k/2-5\ell}$. For this, we use the additional fact that every $f(x,y)\in\mon_a(S_a(2k))$ has $x^2y^2$-coefficient divisible by $p^\ell$. The set $S_1(2k+6\ell)$ is preserved by the transformation $f(x,y)\mapsto f(x+ry,y)$ for any $r\in\Z_p$. This yields a surjective map $\Z_p\times S_1^{(0)}(2k+6\ell)\to S_1(2k+6\ell)$, where $S_1^{(0)}(2k+6\ell)$ is the set of elements in $S_1(2k+6\ell)$ whose $x^3y$-coefficient belongs to $\{0,1,2,3\}$. The density of $S_1^{(0)}(2k+6\ell)$ satisfies the same bound as the density of $S_1(2k+6\ell)$, and is thus $\ll p^{-3k/2-9\ell/2}$. We have the following lemma
\begin{lemma}
Fix any $f(x,y)\in S_1^{(0)}(2k+6\ell)$. Then the density of $r\in\Z_p$ such that $f(x+ry,y)$ belongs to $\mon_a(S_a(2k))$ is $\ll p^{-\ell/2}$.
\end{lemma}
\begin{proof}
We prove the lemma in the case when the $x^4y$-coefficient is $0$; the other cases are identical. In this case, we need to upper bound the set of $r\in\Z_p$ such that $p^\ell\mid (6r^2+c)$ for any fixed $c\in\Z_p$, and the set of such $r$ clearly has density $\ll p^{-\ell/2}$.
\end{proof}

\noindent We therefore have 
\begin{equation*}
\nu(S_a(2k))=p^{6\ell}\nu(\mon_a(S_a(2k)))\ll p^{6\ell}p^{-\ell/2}\nu(S_1^{(0)}(2k+6\ell))\ll
p^{6\ell}p^{-\ell/2}p^{-3k/2-9\ell/2}= p^{\ell-3k/2}.
\end{equation*}
The required result (for $k\geq 3$) follows immediately by integrating over $a$.

\medskip

For $k=2$ and $k=3$, it is sufficient to bound the densities of the sets $T_\delta:=\{f\in V(\Z_p):p^\delta\parallel\Delta(f)\}$ for $\delta=2$, $3$, $4$, and $5$. For this, we use the Jacobian change of variables formula of \cite[Proposition 3.11]{BS2Sel} (also stated in Proposition \ref{prop:Jac}) to obtain
\begin{equation}\label{eq:vTdelta}
\nu(T_\delta)\ll \int_{\substack{(I,J)\in\Z_p^2\\p^\delta\parallel\Delta(I,J)}}\#\frac{\Inv^{-1}(I,J)}{\PGL_2(\Z_p)}dIdJ.
\end{equation}
Let $g_{I,J}$ denote the monic cubic polynomial with invariants $I$ and $J$, and let $R_{I,J}$ denote the corresponding cubic ring. Let $\mathcal{Q}_{I,J}$ denote the set of  quartic algebras over $\Z_p$ with resolvent $R_{I,J}$. We have the following result regarding the size of $\mathcal{Q}_{I,J}$.
\begin{lemma}
We have
\begin{equation*}
\#\frac{\Inv^{-1}(I,J)}{\PGL_2(\Z_p)}\leq \#\Stab_{\GL_2(\Z_p)}(g_{I,J})\#\mathcal{Q}_{I,J}.
\end{equation*}
\end{lemma}
\begin{proof}
The set $\mathcal{Q}_{I,J}$ is in bijection with $\GL_2(\Z_p)\times\SL_3(\Z_p)$-orbits on the set $S_{I,J}$ of pairs $(A,B)\in W(\Z_p)$ such that the ring associated to $4\det(Ax-By)$ is $R_{I,J}$. The map $\iota$ defined in \eqref{eq:iota} gives an injection from $\Inv^{-1}(I,J)$ to $S_{I,J}$. Since $\PGL_2$ is isomorphic to $\SO_{A_0}$ (see the discussion following the definition of $\iota$) this induces an injection $\PGL_2(\Z_p)\backslash\Inv^{-1}(I,J)\to \SL_3(\Z_p)\backslash S_{I,J}$. Moreover, the image of this injection is contained in $S'_{I,J}$, the set of pairs $(A,B)\in S_{I,J}$ with monic cubic resolvent. The lemma follows since it is easy to see that the size of a set of $\GL_2(\Z_p)$-equivalent elements in $\SL_3(\Z_p)\backslash S'_{I,J}$ is bounded by $\#\Stab_{\GL_2(\Z_p)}(g_{I,J})$.
\end{proof}

Set $K_{I,J}:=R_{I,J}\otimes \Q_p$, and note that there are $O(1)$ \'etale quartic algebras $K_4$ over $\Q_p$ with resolvent $K_{I,J}$. Then every element in $\mathcal{Q}_{I,J}$ is a suborder in some such $K_4$ having index $\leq p^2$, and index $\leq p$ if $\Delta(K_4)\geq p^2$. (This latter condition follows since we are assuming that $\delta\leq 5$.) The previously used result of Nakagawa \cite[Theorem 1]{MR1342021} implies that there are $O(1)$ such suborders in $K_4$. The result now follows from \eqref{eq:vTdelta} in conjunction with \cite[Proposition 3.8]{MR4585297}, which upper bounds the density of pairs $(A,B)$ such that $p^\delta\mid\Delta(x^3+Ax+B)$ by $O(p^{-2})$ when $\delta=2$ or $\delta=3$ and by $O(p^{-4})$ when $\delta=4$ or $\delta=5$.
\end{proof}

We now collect several consequences of Proposition \ref{prop:ppden}. The most important of these is the analytic continuation of $D(\phi,s)$ for large and locally well approximated functions $\phi:V(\Z)\to\R$. To establish this continuation, we need the following lemma.

\begin{lemma}\label{lem:need-Dancont}
Fix an integer $k\geq 1$ and an element $a\in \Z/p^{2k}\Z$. We write $a=up^\ell$, where $u\in\Z/p^{2k}\Z$ is a unit and $\ell\leq 2k$. Then we have the following bounds.
\begin{equation*}
\nu_u(\chi_{p^2}),\nu_{up}(\chi_{p^2})\ll\frac{1}{p^2},\;\;\nu_{up^2}(\chi_{p^2})\ll\frac{1}{p},\quad \nu_{up^\ell}(\chi_{p^4})\ll\frac{1}{p^{4-\ell}},
\quad
\nu_{up^\ell}(\chi_{p^{2k}})\ll\min\Bigl(\frac{k}{p^{3k/2-\ell}},1\Bigr),
\end{equation*}
for $k\geq 3$.
\end{lemma}
\begin{proof}
The bound on $\nu_{u}(\chi_{p^2})$ follows from an analysis of the possible splitting types, while the last two bounds are a direct consequence of Proposition \ref{prop:ppden}. We consider the remaining two bounds.

We first consider
$\nu_{up^2}(\chi_{p^2})$.
Let $f=(a,b,c,d,e)\in V(\Z_p)$ and assume $p^2\mid a$
Then by \eqref{eq:disc-cong-cubic},
\begin{equation*}
\Delta(f)\equiv b^2(c^2d^2+18bcde-4bd^3-4c^3e-27b^2e^2)
\pmod{p^2}.
\end{equation*}
Note that $\Delta_3:=c^2d^2+18bcde-4bd^3-4c^3e-27b^2e^2$
is the discriminant of $bx^3+cx^2y+dxy^2+ey^3$.
The density of $(b,c,d,e)\in\Z_p^4$
with $p^2\mid b^2$ is $O(p^{-1})$,
while with $p^2\mid \Delta_3$ is $O(p^{-2})$,
as desired.

Finally we consider $\nu_{up}(\chi_{p^2})$.
Let $f=(a,b,c,d,e)\in V(\Z_p)$ with
 $a=up$ for some $u\in\Z_p^\times$,
and assume $p^2\mid \Delta(f)$.
Then $(f\pmod p)\in V(\F_p)$ has a multiple root in $\mathbb P^1$.
Note that $(1:0)\in\mathbb P^1(\F_p)$ is
a root.
Suppose $(1:0)$ is a multiple root.
Then $p\mid b$. Thus by \eqref{eq:disc-cong-quad},
\[
\Delta(f)\equiv 4ac^3(4ce-d^2)
\pmod{p^2}.
\]
Thus $p^2\mid\Delta(f)$ if and only if $p\mid c$ or $p\mid (4ce-d^2)$.
The $(b,c,d,e)\in\Z_p^4$ satisfying the conditions have density $O(p^{-2})$ as desired.
Suppose $(1:0)\in\mathbb P^1(\F_p)$ is a simple root. Then $(f\pmod p)$ has one multiple root in $\mathbb P^1(\F_p)$. Further suppose that the multiple root is $(0:1)$. Then $p\mid d$ and $p\mid e$. Thus by \eqref{eq:disc-cong-quad}
with $\Delta(a,b,c,d,e)=\Delta(e,d,c,b,a)$, we have
\[
\Delta(f)\equiv -4b^2c^3e
\pmod{p^2}.
\]
So $p^2\mid\Delta(f)$ if and only if $p^2\mid e$ or     $p\mid bc$.
The density of $(b,c,d,e)\in\Z_p^4$
satisfying the conditions is
$O(p^{-3})$. There are $p$ possibilities
for the multiple root of $(f\pmod p)$,
and the densities all coincide.
Thus the total density of $(b,c,d,e)\in\Z_p^4$
for $p^2\mid\Delta(f)$ is $O(p^{-2})$,
completing the proof.
\end{proof}

Let $\phi:V(\Z)\to\R$ be a large and locally well approximated function via $\phi(\cdot)=\sum_n\phi(n;\cdot)$. We clearly have the equality $D^\pm(\phi,s)=\sum_{n}D^\pm(\phi(n;\cdot),s)$ in the region $\Re(s)>1$ of absolute convergence. Recall that $\phi(n;\cdot)$ is defined by congruence conditions modulo $n^2$, and is supported on the set of elements $f\in V(\Z_p)$ satisfying $n^2\mid c\Delta(f)$, for some positive integer $c$. We have the following result on $D(\phi(n;\cdot),s)$:

\begin{proposition}\label{lem:Dancont}
Keep the above notation. The functions $D^\pm(\phi(n;\cdot),s)$ have analytic continuation to the whole plane, with at most a simple pole at $s=1$ with residue $r_n=\nu(\phi(n;\cdot))$.
Moreover, we have the bound
\begin{equation}\label{eq:Dconvtemp}
\left|\ D^\pm(\phi(n;\cdot),s)-\frac{r_n}{s-1}\ \right|\ll_\epsilon \frac{n^\epsilon}{n_1^{3/2+\sigma/2}n_2^{4\sigma}m_3^{1/3+\sigma/6}}(1+|t|),
\end{equation}
for $s=\sigma+it$ with $\sigma>1/3$, where we write $n=n_1n_2^2m_3$ with $n_1$ and $n_2$ squarefree, $m_3$ cubefull $(p\mid m_3$ implies $p^3\mid m_3)$, and $n_1$, $n_2$, $m_3$ pairwise relatively prime.
\end{proposition}

\begin{proof}
We write $\psi:=\phi(n,\cdot)$.
We break up the Dirichlet series into summands corresponding to each divisor $d$ of $n^2$:
\begin{equation*}
D^\pm(\psi,s)=\sum_{a\geq 1}\frac{\nu_{\pm a}(\psi)}{a^s}=\sum_{d\mid n^2}\sum_{\substack{a\geq 1\\(a,n^2)=d}}\frac{\nu_{\pm a}(\psi)}{a^s}
=\sum_{d\mid n^2}\frac{1}{d^s}\sum_{\substack{a\geq 1\\(a,n^2)=1}}\frac{\nu_{\pm ad}(\psi)}{a^s}.
\end{equation*}
Since $\psi$ is $\PGL_2(\Z/n^2\Z)$ invariant, it follows that $\nu_{\pm a}(\psi)=\nu_{\pm u^2a}(\psi)$ for each $u\in\Z$ with $(u,n^2)=1$. It follows that 
for any $d$, the value of $\nu_{\pm ad}(\psi)$ only depends on the residue class of $a$ modulo $n^2/d$. Moreover, the inner product of $\nu(\psi|_{ad})$ with a character $\chi$ of $(\Z/(n^2/d)\Z)^\times$ is $0$ unless $\chi$ is quadratic.
Therefore, each summand in the rightmost term of the above equation is a weighted sum of $O(n^\epsilon)$ quadratic Dirichlet $L$-functions $L(s,\chi)$, where $\chi$ is a quadratic character on $(\Z/(n^2/d)\Z)^\times$. The analytic continuation of $D^\pm(\psi,s)$ follows immediately.

To prove the bound, we start with observing that when $\chi$ is a (quadratic) character on $(\Z/(n^2/d)\Z)^\times$, the conductor of $L(s,\chi)$ is $\ll \rad(n^2/d)$, where $\rad$ denotes the radical function. Applying the convexity bound, we obtain for $s=\sigma+it$ with $\sigma\in(0,1-\delta)$ for any positive $\delta$, the following estimate:
\begin{equation*}
\Bigl|\sum_{\substack{a\geq 1\\(a,n^2)=1}}\frac{\nu_{\pm ad}(\psi)}{a^s}\Bigr|\ll_\epsilon \max_{a}|\nu_{\pm ad}(\psi)|\rad(n^2/d)^{1/2-\sigma/2+\epsilon}(1+|t|).
\end{equation*}
Next, we bound $\max_{a}|\nu_{\pm ad}(\psi)|$ using Lemma \ref{lem:need-Dancont}. As above, we write $n=n_1n_2^2m_3$. We also write $d=d_1d_2d_3$, where  $d_1\mid n_1^{2}$, $d_2\mid n_2^4$, and $d_3\mid m_3^2$.
Therefore, we have
\begin{equation*}
\begin{array}{rcl}
&&\displaystyle\frac{1}{d^\sigma}\max_{a}|\nu_{\pm ad}(\psi)|\rad(n^2/d)^{1/2-\sigma/2+\epsilon}
\\[.2in]
&&\ll_\epsilon 
\displaystyle n^\epsilon\Bigl(
\prod_{p^2\mid d_1}p^{-1-2\sigma}\prod_{p\mid n_1^2/d_1}p^{-3/2-\sigma/2}
\Bigr)\cdot\Bigl(\frac{d_2^{1-\sigma}}{n_2^4}\prod_{p\mid n_2^4/d_2}p^{1/2-\sigma/2}\Bigr)\cdot d_3^{-\sigma}\min\Bigl(\frac{d_3}{m_3^{3/2}},1\Bigr)m_3^{(1-\sigma)/6}
\\[.2in]
&&\ll_\epsilon 
\displaystyle
\frac{n^\epsilon}{n_1^{3/2+\sigma/2}n_2^{4\sigma}m_3^{1/3+\sigma/6}},
\end{array}
\end{equation*}
where the last estimate follows from a straightforward check. The result now follows noting that the number of divisors of $n^2$ is bounded by $O_\epsilon(n^\epsilon)$.
\end{proof}

The above proposition has the following immediate corollary, which follows by noting that the sum over $n$ of the right hand side in \eqref{eq:Dconvtemp} converges absolutely for $\sigma>1/3$ and that the sum of the residues $r_n(=\nu(\phi(n;\cdot))$ also converges absolutely.

\begin{cor}\label{cor:anacont}
Let $\phi:V(\Z)\to\R$ be a large and locally well approximated function. Then $D^\pm(\phi,s)$ has an analytic continuation to $\Re(s)>1/3$ with only a possible simple pole at $s=1$. In particular, the value of $D(\phi,1/2)$ is well defined.
\end{cor}

\subsection{Equidistribution of strongly invariant sets in $V(\Z/p^2\Z)$}

For a positive integer $n$, and a function $\phi:V(\Z/n\Z)\to\C$, let $\widehat{\phi}:V^*(\Z/n\Z)\to\C$ be its Fourier transform normalized as follows:
\begin{equation}
\widehat{\phi}(h)=\frac{1}{n^5}\sum_{f\in V(\Z/n\Z)}\phi(f)e\Big(\frac{[f,h]}{n}\Big)=\frac{1}{n^5}\sum_{f\in V(\Z/n\Z)}\phi(f)\exp\Big(\frac{2\pi i[f,h]}{n}\Big),
\end{equation}
where $[\cdot,\cdot]$ is the natural bilinear form $V(\Z/n\Z)\times V^*(\Z/n\Z)\to\Z/n\Z$. For the rest of the section, we take $n=p^k$ to be a prime power. We need our result only for $k=2$, but since the argument is identical we work for general $k$. Let $S\subset V(\Z/n\Z)$ be a strongly invariant set, i.e., its characteristic function $\chi_S$ is strongly invariant. Clearly we have $\widehat{\chi_S}(0)=|S|/n^5$ and we have the ``trivial'' bound $|\widehat{\chi_S}(h)|\leq |S|/n^5$ for any $h\in V^*(\Z/n\Z)$. (In fact, these trivial bounds clearly hold for every subset $S$.) The first result of this subsection improves upon this trivial bound for strongly invariant sets in $V(\Z/p^k\Z)$ not intersecting $pV(\Z/p^k\Z)$. To simplify notation for the next proposition, we write $V$, $V^*$, $\GL_1$, $\PGL_2$ for $V(\Z/p^k\Z)$, $V^*(\Z/p^k\Z)$, $\GL_1(\Z/p^k\Z)$, and $\PGL_2(\Z/p^k\Z)$, respectively.
We state the result in terms of bounding the {\it orbital exponential sum} $\mathcal{G}_{p^k}(f,h)$ associated to $f\in V$ and $h\in V^*$ defined as
\begin{equation}\label{eq:oes}
\mathcal{G}_{p^k}(f,h):=
\frac{1}{|\GL_1||\PGL_2|}\sum_{t\in\GL_1}\sum_{g\in\PGL_2}\exp\Bigl(2\pi i\cdot\frac{t^2[gf,h]}{p^k}\Bigr).
\end{equation}
Then we prove the following proposition.
\begin{proposition}\label{prop:orbitalintegral}
Let notation be as above, and let $f\in V$ be an element which is not a multiple of~$p$. Then we have
\begin{equation}\label{eq:oes-estimate}
    \mathcal{G}_{p^k}(f,h)\ll
        \begin{cases}
            1   &   h=0,\\
            p^{-1/2} & h\in p^{k-1}V^*, h\neq0,\\
            p^{-1} & h\notin p^{k-1}V^*.
        \end{cases}
\end{equation}
\end{proposition}
\begin{proof}
The classical quadratic Gauss sum is defined by
\begin{equation}
    \mathcal Q_{p^k}(a):=\frac{1}{|\GL_1|}\sum_{t\in\GL_1}\exp\Bigl(2\pi i\cdot\frac{t^2a}{p^k}\Bigr),
    \qquad a\in\Z/p^k\Z.
\end{equation}
The explicit formula of the quadratic Gauss sum is well known. In particular, we have
\begin{equation}\label{eq:qgs-estimate}
    \mathcal{Q}_{p^k}(a)\ll
        \begin{cases}
            1   &   a=0,\\
            p^{-1/2} & p^{k-1}\mid a, a\neq0,\\
            p^{-1} & p^{k-1}\nmid a.
        \end{cases}
\end{equation}
Then by definition, the orbital exponential sum is expressed as
\begin{equation}
\mathcal{G}_{p^k}(f,h)=
\frac{1}{|\PGL_2|}\sum_{g\in\PGL_2}\mathcal Q_{p^k}\bigl([gf,h]\bigr).
\end{equation}
For simplicity we first consider the case $h\notin pV^\ast$.
We show that except for $O(p^{-1})$-proposition of $g$ in $\PGL_2$,
$[gf,h]$ is not divisible by $p$. Then \eqref{eq:qgs-estimate} implies that
$\mathcal{G}_{p^k}(f,h)\ll p^{-1}$.
To study $[gf,h]$, it will be convenient to have an explicit description of $V^\ast$.
Since the estimate \eqref{eq:oes-estimate} holds automatically for $p=2,3$,
we assume $p\neq2,3$.
For $f=(f_0,f_1,f_2,f_3,f_4), h=(h_0,h_1,h_2,h_3,h_4)\in V$, let
$[f,h]=f_0h_0+4^{-1}f_1h_1+6^{-1}f_2h_2+4^{-1}f_3h_3+f_4h4$.
This bilinear form is $\PGL_2$-invariant in the sense that $[gf,h]=[f,g^Th]$ holds for all $f,h\in V$ and $g\in\PGL_2$,
where $g^T$ is the matrix transpose of $g$. Via the map $V\ni h\mapsto[\cdot,h]\in V^*$ we identify $V^*$ with $V$,
and regard $h$ as an element in $V$.

Let $\PGL_2^\bullet\subset \PGL_2$ consists of elements
whose $(1,2)$-entry (i.e., the upper right entry) is in
$(\Z/p^k\Z)^\times$. Then the proportion of elements $g\in\PGL_2$
not in $\PGL_2^\bullet$ is $O(p^{-1})$,
and hence has a negligible contribution to  \eqref{eq:oes-estimate}.
Any $g\in\PGL_2^\bullet$ is expressed uniquely in the form
\[
g=
l_ma_swl_n:=
\begin{pmatrix}
1&0\\m&1
\end{pmatrix}
\begin{pmatrix}
s&0\\0&1
\end{pmatrix}
\begin{pmatrix}
0&1\\1&0
\end{pmatrix}
\begin{pmatrix}
1&0\\n&1
\end{pmatrix}
\]
for
$m,n\in\Z/p^k\Z$ and $s\in(\Z/p^k\Z)^\times$.
We consider $[gf,h]=[a_swl_nf,l_m^Th]$.
The $x^4$-coefficient of $wl_nf$ is $(wl_nf)(1,0)=f(n,1)$, and 
the $x^4$-coefficient of $l_m^Th$ is $(l_m^Th)(1,0)=h(1,m)$.
Therefore, we have
\begin{equation}\label{eq:bilinear-in-s}
    [gf,h]=\frac{1}{s^2}(f(n,1)h(1,m)s^4+p_{m,n}(s)).
\end{equation}
for a degree $\leq3$ polynomial $p_{m,n}(s)$ in $s$
with coefficients in $(\Z/p^k\Z)[m,n]$.
Since $f\pmod p$ and $h\pmod p$ are both not the zero form by our assumption, they have at most $4$ roots in $\mathbb{P}^1(\F_p)$.
Therefore, except for $O(p^{-1})$-propotion of $g$, $f(n,1)h(1,m)$ is not divisible by $p$.
For such majority of $g$, the reduction modulo $p$
of the polynomial of $s$ in \eqref{eq:bilinear-in-s} is
of degree $4$, and has at most $4$ roots in $\F_p$.
Therefore, except for $O(p^{-1})$-propotion of $s$ (and hence of $g$),
$[gf,h]$ is not divisible by $p$. This finishes the proof of \eqref{eq:oes-estimate} for $h\notin pV$.

This argument works for general $h\neq0$:
Let us write $h=p^lh'$,
where $l<k$ and not all the coefficients of $h'$ are
divisible by $p$. Then $[gf,h]=p^l[gf,h']$.
(Here $h'$ is well defined as an element in $V(\Z/p^{k-l}\Z)$
and we also regard $[gf,h']\in \Z/p^{k-l}\Z$.)
Exactly by the same argument,
we confirm that except for $O(p^{-1})$-proportion of $g$,
$[gf,h']$ is not divisible by $p$. Thus again \eqref{eq:qgs-estimate} implies \eqref{eq:oes-estimate}, concluding the proof of the proposition.
\end{proof}

Proposition \ref{prop:orbitalintegral} combined 
 with the first bound of Proposition \ref{prop:ppden} immediately implies the following result.

\begin{cor}\label{prop:FT}
Let $p$ be a prime, and let $\phi:V(\Z/p^2\Z)\to\R$ be a strongly invariant bounded function whose support is contained within the set of elements $f\in V(\Z/p^2\Z)$ with $\Delta(f)=0$. Then we have
\begin{equation*}
\widehat{\phi}(0)\ll\frac{1}{p^2},\quad\quad
\widehat{\phi}(ph)\ll\frac{1}{p^{2+1/2}},\quad\quad
\widehat{\phi}(h)\ll\frac{1}{p^{3}},
\end{equation*}
for $h\in V^*(\Z/p^2\Z), h\notin pV^*(\Z/p^2\Z)$.
\end{cor}

\section{Uniformity estimates}\label{sec:Unif}

Recall that for a positive integer $m$, the set of generic elements in $V(\Z)$ whose discriminants are divisible by $m^2$ is denoted by $\W_m$. In this section, we prove Theorem \ref{Thm:Unif} by giving a bound for the number of $\PGL_2(\Z)$-orbits on $\W_m$ (on average over $m$) having bounded height.

\subsection{Fibering over quartic fields}

Let $f\in V(\Z)$ be an integral binary quartic form, and let $(Q_f,C_f,\alpha_f)$ be the triple corresponding to the $\PGL_2(\Z)$-orbit of $f$ under the bijction of Theorem~\ref{th:bqfQC}. If $f$ is generic, then $Q_f$ and $C_f$ are integral domains, and thus are orders in a quartic field and a cubic field, respectively. We define $L_f$ to be $Q_f\otimes\Q$, and let $\O_f$ denote the maximal order in $L_f$. Let $R_f$ be the (unique) cubic resolvent ring of $\O_f$. Define the {\it index} $\ind(f)$ of $f$ to be the index of $Q_f$ in $\O_f$ (equivalently, the index of $C_f$ in $R_f$), and define the {\it strongly divisible factor} $\sd(f)$ of $f$ to be the product of primes $p$ such that $p^2\mid\Delta(\O_f)$. For positive integers $m_1$ and $m_2$ we define
\begin{equation*}
\W_{m_1,m_2}:=\bigl\{
f\in \W_{m_1m_2}:m_1\mid\ind(f),\,
m_2\mid \sd(f)
\bigr\}.
\end{equation*}
For positive real numbers $M_1$ and $M_2$, let $N(X;M_1,M_2)$ denote the following quantity:
\begin{equation}
N(X;M_1,M_2):=\sum_{\substack{m_1\in[M_1,2M_1]\\ m_2\geq M_2}}|\mu(m_2)|\#\frac{\{f\in\W_{m_1,m_2}:H(f)<X\}}{\PGL_2(\Z)}.
\end{equation}
Note that $\Delta(f)/(\ind(f)^2\sd(f)^2)=\Delta(\O_f)/\sd(f)^2$ is squarefree upto some bounded product $C_6$ of $2$'s and $3$'s.
As a consequence, we have the inclusion
$\W_m\subset \bigcup_{\substack{m\mid C_6m_1m_2}}\W_{m_1,m_2}$,
where the $m_2$ are squarefree. 
Therefore, for a real number $M>1$, we have
\begin{equation}\label{eq:firstinc}
\sum_{m\geq M}\#
\frac{\{f\in\W_m:H(f)<X\}}{\PGL_2(\Z)}
\ll \sum_{\substack{2^k=M_1\leq X^{1/2}\\M_2=M/(2C_6M_1)}}N(X;M_1,M_2),
\end{equation}
where $M_1$ ranges over powers of $2$ betewen $1$ and $X^{1/2}$, while $M_2=M/(2C_6M_1)$. In particular, the sum over $M_1,M_2$ has length $O(\log(X))$.

\medskip

We repackage the count $N(X;M_1,M_2)$ by fibering over maximal quartic orders. To this end, let $V(\Z)^\gen_{H<X}$ denote the set of generic elements in $V(\Z)$ having height less than $X$. We consider the map\begin{equation}\label{eq:ftoringmap}
\begin{array}{rcl}
\Psi_X:\displaystyle\PGL_2(\Z)\backslash V(\Z)^\gen_{H<X}&\to&
\{\mbox{maximal quartic order with its (unique) cubic resolvent  order}\}
\\[.1in]
f&\mapsto & (\cO_f,R_f).
\end{array}
\end{equation}
    The ring $C_f$ is a monogenized cubic subring of $R_f$, with monogenizer $\alpha_f\in R_f$. For a cubic ring $R$, we let $R^\red$ denote the set of elements $\alpha\in R$ such that $\Tr(\alpha)\in\{-1,0,1\}$. Since the monogenizer of $C_f$ is defined upto addition by an element of $\Z$, we may in fact assume that $\alpha_f\in R_f^\red$. We introduce a {\it height} on $R_f$ by setting $h(\alpha):=\max(|\alpha'|_v)$, where $\alpha'$ is the unique $\Z$-translate of $\alpha$ with $\alpha'\in R^\red$, 
and $v$ ranges over the archimedian valuations of $K=C\otimes \Q$; when $v$ is a complex place, we take $|x|_v$ to be the absolute value of $x$. (An equivalent height is to take the length of $\alpha'$ in $R_f\otimes_\Z\R$.) A standard computation reveals that we have $h(\alpha_f)\ll H(f)^{1/6}$.

We study the fibers of the map \eqref{eq:ftoringmap}. Clearly, we have an injection
\begin{equation}\label{eq:Psiinjection}
\Psi_X^{-1}(\O,R)\hookrightarrow\bigl\{(Q,\alpha):Q\subset \O,\,\alpha\in R^\red,\,h(\alpha)\ll X^{1/6},\,[\O:Q]=[R:\Z[\alpha]]\bigr\}.
\end{equation}
In fact, the above map remains an injection even if we also specify that $\Z[\alpha]$ is a cubic resolvent ring of $Q$. However, we will not be using this fact in our estimates. Recall that by Bhargava's work \cite{BHCL3}, the set of pairs $(\O,R)$, where $\O$ is quartic ring, and $R$ is a cubic resolvent ring of $\O$, is in bijection with $\GL_2(\Z)\times\SL_3(\Z)$-orbits on $W(\Z)$. We denote the set of elements in $W(\Z)$ corresponding to pairs $(\O,R)$, where $\O$ is a maximal integral domain and $R$ is an integral domain, by $W(\Z)^\mg$. For $(A,B)\in W(\Z)^\mg$, let $w_X(A,B)$ to be the size of the right hand side of \eqref{eq:Psiinjection},
where (where $(A,B)$ to the pair of rings $(\O,R)$. 
For positive real numbers $Y$ and $M_2$, let $S_W(Y,M_2)$ denote the set of $\GL_2(\Z)\times\SL_3(\Z)$-orbits on the set of elements $(A,B)\in W(\Z)^\mg$ such that $|\Delta(A,B)|\ll Y$, and $m_2^2\mid \Delta(A,B)$ for some squarefree positive integer $m_2\geq M_2$. Then we have the following result.
\begin{proposition}\label{prop:NXM1M2}
With notation as above, we have
\begin{equation}\label{eq:uniftemp2}
N(X;M_1,M_2)\ll \sum_{(A,B)\in S_W(X/M_1^2,M_2)}w_X(A,B).
\end{equation}
\end{proposition}
\begin{proof}
For integers $m_1>M_1$ and $m_2>M_2$ Let $f\in \W_{m_1,m_2}$ be a binary quartic form with $H(f)\leq X$. Note that we have $|\Delta(Q_f)|=|\Delta(f)|\ll H(f)$. Hence, we have $|\Delta(R_f)|=|\Delta(\O_f)|\ll X/M_1^2$, and moreover, $m_2^2\mid\Delta(\cO_f)$. Hence, $(\O_f,R_f)\in S_W(X/M_1^2,M_2)$. Therefore, we have
\begin{equation*}
N(X;M_1,M_2)\ll \sum_{(A,B)\in S_W(X/M_1^2,M_2)}\#\Psi_X^{-1}(A,B)\ll \sum_{(A,B)\in S_W(X/M_1^2,M_2)}w_X(A,B),
\end{equation*}
as needed.
\end{proof}

\subsection{Preliminary results on cubic rings}

Evaluating $w_X(A,B)$ requires counting elements $\alpha\in R^\red$, where $R$ is a cubic ring. To this end, we collect some basic results on cubic rings and their shapes.

We begin by introducing some notation. Let $g$ be an irreducible integral binary cubic form, whose $\GL_2(\Z)$-orbit corresponds to the cubic order $R$ under the Dalone--Faddeev parametrization. Let $K$ denote the fraction field of $R$. We identify $R\otimes\R$ with $\R^3$ by choosing the isomorphism $\C\cong \R^2$ given by $a+ib\mapsto (a+b,a-b)$ when $K$ is complex. The embedding $R\to R\otimes\R$ maps $R$ into a lattice. Denote the lengths of successive minima of this lattice by $1$ (corresponding the to element $1\in R$), $\ell_1(R)$, and $\ell_2(R)$. (Here, we have normalized the length on $R\otimes\R$ such that $1\in R$ has length $1$ in $R\otimes\R$.) We define the {\it skewness} of $R$ to be $\sk(R):=\ell_2(R)/\ell_1(R)$. 
We begin with the following lemma.

\begin{lemma}\label{lem:crp}
Let $R$ be a cubic order. Then we have
\begin{equation*}
\ell_1(R)\asymp \frac{|\Disc(R)|^{\frac14}}{\sk(R)^{\frac12}};\quad
\ell_2(R)\asymp |\Disc(R)|^{\frac14}\sk(R)^{\frac12};\quad \ell_2(R)\ll\ell_1(R)^2;\quad 1\leq \sk(R)\ll |\Disc(R)|^{\frac16}.
\end{equation*}
\end{lemma}
\begin{proof}
First note that by Minkowski's theorem, we have $\ell_1(R)\ell_2(R)\asymp \sqrt{|\Disc(R)|}$. The first two claims follow immediately from this and the definition of $\sk(R)$. Suppose $\beta\in R\backslash\Z$ is the element having length $\ell_1(C)$. Then $\beta^2$ has length $\asymp\ell_1^2(R)$. Since $\ell_2(R)$ is the second successsive minima, we have $\ell_1(R)^2\gg \ell_2(R)$, which yields the third claim. The final claim is a consequence of the first three claims.
\end{proof}

Next, recall that an integral irreducible binary cubic form $g(x,y)\in U(\Z)$ corresponds to a cubic order $R$, together with a basis $\{\overline{\beta_1},\overline{\beta_2}\}$ of $R/\Z$. We say that $g(x,y)$ is {\it almost Minkowski reduced} if the basis $\{\beta_1,\beta_2\}$ is almost Minkowski reduced, where we say that some set of linearly independent vectors in $\R^n$ is {\it almost Minkowski reduced} if the angles between every pair of them
is bounded below by a positive absolute constant. 
Note that if $\{1,\beta_1,\beta_2\}$ is an almost Minkowski reduced basis of $R$, then $\{\overline{\beta_1},\overline{\beta_2}\}$ is an almost Minkowski reduced basis of $R/\Z$. Conversely, if $\{\overline{\beta_1},\overline{\beta_2}\}$ is an almost Minkowski reduced basis of $R/\Z$, then $\overline{\beta_1}$ and $\overline{\beta_2}$ admit lifts to $\beta_1$ and $\beta_2$, respectively, such that $\{1,\beta_1,\beta_2\}$ is an 
almost Minkowski reduced basis of $R$. Indeed, to ensure this, $\beta_1$ and $\beta_2$ should simply be taken to be elements in $R^\red$.

The notion of being almost Minkowsky reduced can be generalized to elements of $U(\R)$: Namely, given $h(x,y)\in U(\R)$ with nonzero discriminant, an analogue of the Delone--Faddeev parametrization over $\R$ (see for example \cite{BSHMC}) yields an \'etale cubic algebra $R_h$ over $\R$ along with a basis $\{\overline{\beta_1},\overline{\beta_2}\}$ for $R_h/\R$. The space $R_h/\R$ can be naturally identified with the set of traceless elements in $R_h$, and this identification allows us to restrict the standard norm and inner product on $\R^3$ to $R_h/\R$. We say that 
$h(x,y)$ is {\it almost Minkowski reduced} if $\{\overline{\beta_1},\overline{\beta_2}\}$ is almost Minkowski reduced. We define $\overline{\sk}(h)$ to be $|\overline{\beta_1}|/|\overline{\beta_2}|$, and note that if $h(x,y)\in U(\Z)$, corresponding to the cubic order $R$, is almost Minkowski reduced, then $\overline{\sk}(h)\asymp\sk(R)$.
We now have the following result.
\begin{lemma}
Let $B\subset U(\R)$ be a bounded set whose boundary does not intersect the discriminant zero locus. Then every element in $\FF\cdot B$ is almost Minkowski reduced. Moreover, for $\gamma=utk\in\FF$ and $h\in B$, we have $\sk(\gamma\cdot h)\asymp t^2$.
\end{lemma}
\begin{proof}
We start by noting three facts. First, the assumptions on $B$ ensure that the angles between the two basis vectors corresponding to each element of $B$ are uniformly bounded away from $0$. Hence every element in $B$ is almost Minkowski reduced. Second, let $\gamma\in \GL_2(\R)$ and let $h(x,y)\in U(\R)$ be an element corresponding to the basis $\{\overline{\beta_1},\overline{\beta_2}\}$ of $R_h$. Then $\{\gamma\overline{\beta_1},\gamma\overline{\beta_2}\}$ is the basis corresponding to $\gamma\cdot h$. Third, if $\{\overline{\beta_1},\overline{\beta_2}\}$ is almost Minkowski reduced, then so is $\{\gamma\overline{\beta_1},\gamma\overline{\beta_2}\}$ for every $\gamma\in\FF$. The first assertion follows from these three facts. The second assertion is an immediate consequence of the second fact.
\end{proof}

We conclude this subsection by explaining how these results can be used to help evaluate the right hand side of \eqref{eq:uniftemp2}.
Let $g(x,y)$ be an irreducible integral binary cubic form corresponding to the cubic order $R$ be a cubic order and basis $\{1,\beta_1,\beta_2\}$, with $\beta_1,\beta_2\in R^\red$. There is a natural bijection between $R^\red$ and $\Z^2$, where we map $(r,s)\in\Z^2$ to the unique $\Z$-translate $\alpha$ of $r\beta_1+s\beta_2$ such that $\tTr(\alpha)\in\{-1,0,1\}$. Suppose the basis $\{1,\beta_1,\beta_2\}$ is almost Minkowsky reduced. Then the condition $h(\alpha)\ll X^{1/6}$ translates to $|r|\ll X^{1/6}/\ell_1(R)$ and $|s|\ll X^{1/6}/\ell_2(R)$. 
Suppose $(A,B)\in W(\Z)^\mg$ with cubic resolvent $g(x,y)$ corresponds to the pair $(\O,R)$. Assume that $g(x,y)$ corresponds to an almost Minkowski reduced basis $\{1,\beta_1,\beta_2\}$ of $R$. The quantity $w_X(A,B)$ is equal to the number of pairs $(Q,\alpha)$, where $Q$ is a suborder of $\cO$, $\alpha\in R$ with $h(\alpha)\ll X^{1/6}$ and $[\cO:Q]=[R:\Z[\alpha]]$. The height condition implies that there exist integers $r_1$ and $r_2$ with $|r_1|\ll X^{1/6}/\ell_1(R)$ and $|r_2|\ll X^{1/6}/\ell_2(R)$ such that $\alpha=r_1\beta_1+r_2\beta_2+n$ for some integer $n$. Moreover, since $g(x,y)$ is the index form on $R/\Z$, it follows that $[R:\Z[\alpha]]=|g(r_1,r_2)|$.
Therefore, with notation as above, we have the following bound:
\begin{equation}\label{eq:wX}
w_X(A,B)\ll\sum_{\substack{(r_1,r_2)\in\Z^2\\|r_1|\ll X^{1/6}/\ell_1(R)\\|r_2|\ll X^{1/6}/\ell_2(R)}}\#\{Q\subset\O:[\O:Q]=|g(r_1,r_2)|\}.
\end{equation}
We turn now to the right hand side of \eqref{eq:uniftemp2}. It is necessary to sum over $(A,B)\in S_W(X/M_1^2,M_2)$. Such an element $(A,B)$ corresponds to a pair $(\O,R)$, where $\O$ is a maximal quartic order, and $R$ is the cubic resolvent of $\O$. We will break the ranges of the possible values of $|\Delta(R)|$ and $\sk(R)$ into dyadic ranges. By Lemma \ref{lem:crp}, a pair of such dyadic ranges $[Y,2Y]$ and $[Z,2Z]$ determines the sizes of $\ell_1(R)$ and $\ell_2(R)$ up to absolutely bounded multiplicative constants. Furthermore, $Y$ will be bounded above by $X/M_1^2$. 

Let $S_W(Y,Z;M_2)$ denote the set of $\GL_2(\Z)\times\SL_3(\Z)$-orbits on the set of elements $(A,B)\in W(\Z)^\mg$, corresponding to $(\O,R)$, such that $Y\leq|\Delta(A,B)|<2Y$, $Z\leq\sk(R)<2Z$, and the discriminant of $(A,B)$ is strongly divisible by $m_2^2$ for some squarefree positive integer $m_2\geq M_2$. Then an application of Bhargava's averaging method in \cite{dodqf} yields an absolutely bounded ball $\cB\subset W(\R)$, with support bounded away from the discriminant-$0$ locus of $W(\R)$, such that
\begin{equation}\label{eq:unifQavg}
\sum_{(A,B)\in S_W(Y,Z;M_2)}w_X(A,B)\ll
\int_{\substack{\lambda\asymp Y^{1/12}\\t\asymp Z^{1/2}\\s_1,s_2\gg 1}}
\Bigl(\sum_{(A,B)\in (t,s)\lambda\cB\cap W(\Z)^\mg_{M_2}}w_X(A,B)\Bigr)\frac{d^\times s  d^\times t d^\times \lambda}{s_1^6s_2^6t^2},
\end{equation}
where $s=(s_1,s_2)$, $d^\times s:=d^\times s_1 d^\times s_2$, $(t,s_1,s_2)$ is the diagonal element $(g_2,g_3)\in\GL_2(\R)\times\SL_3(\R)$ with $g_2=\diag(t^{-1},t)$ and $g_3=\diag(s_1^{-2}s_2^{-1},s_1s_2^{-1},s_1s_2^2)$, and $W(\Z)^\mg_{M_2}$ denotes the subset of $W(\Z)^\mg$ whose discriminants are strongly divisible by $m_2^2$ for some squarefree $m_2\geq M_2$. In the next subsection, we use \eqref{eq:wX} and \eqref{eq:unifQavg} to obtain the uniformity estimate.

\subsection{Bounds in the small $Y$ range}

Let $X$, $Y$, and $Z$ be positive real numbers with $Y\ll X/M_1^2$ and $Z\ll Y^{1/6}$. (This bound on $Z$ comes from Lemma \ref{lem:crp}.) In this subsection, we will break up our uniformity estimate into three parts (the sum over $r_1,r_2$ with $r_1r_2\neq0$, the sum over $r_2$ with $r_1=0$, and the sum over $r_1$ with $r_2=0$). For the first two parts to be nonzero, we will show that we must have $Y\ll X^{2/3}$. In this subsection, we obtain bounds on these two parts. 

We begin with the following lemma. Recall that we define $N(k)$ for positive integers $k$ in Proposition \ref{prop:Nak}.
For convenience we define it for negative integers $k$
in the same manner.
In the sequel we use the following bound.
\begin{lemma}\label{lem:sumdar}
We have
\begin{equation*}
\sum_{\substack{0\neq |d|\ll D\\0\neq|r|\ll R}}N(dr^3)\ll_\epsilon D^{1+\epsilon}R^{2+\epsilon}.
\end{equation*}
\end{lemma}
The above lemma follows from standard methods, and we omit the proof.

For $\lambda\asymp Y^{1/12}$, $t\asymp Z^{1/2}$, and $s_1,s_2\gg 1$, the cubic resolvent ring $R$ associated to every element in $(t,s_1,s_2)\lambda\cB\cap W(\Z)^\mg$ satisfies (by design) $|\Delta(R)|\asymp Y$ and $\sk(R)\asymp Z$. Therefore, from Lemma \ref{lem:crp}, we also have $\ell_1(R)\asymp Y^{1/4}/Z^{1/2}$ and $\ell_2(R)\asymp Y^{1/4}Z^{1/2}$.
We define the quantities
\begin{equation}\label{eq:R1R2}
R_1:=X^{1/6}Z^{1/2}/Y^{1/4};\quad\quad
R_2:=X^{1/6}/(Y^{1/4}Z^{1/2}).
\end{equation}
Then the ranges of $|r_1|$ and $|r_2|$ in the right hand side of \eqref{eq:wX}, for every $(A,B)\in (t,s_1,s_2)\lambda\cB\cap W(\Z)^\mg$, are $|r_1|\ll R_1$ and $|r_2|\ll R_2$. 

We define the functions $w_X^{(1)}$, $w_X^{(2)}$ and $w_X^{(3)}$ from $W(\Z)^\mg\to\R$ by
\begin{equation*}
w_X^{(1)}(A,B)=\sum_{\substack{r_1r_2\neq 0\\r_i\ll R_i}}N(g(r_1,r_2)),\;\;
w_X^{(2)}(A,B)=\sum_{\substack{r_2\ll R_2}}N(g(0,r_2)),\;\;
w_X^{(3)}(A,B)=\sum_{\substack{r_1\ll R_1}}N(g(r_1,0))
\end{equation*}
where $g(x,y)$ is the cubic resolvent form corresponding to $(A,B)$ and the sum is over integers $r_1$ and $r_2$, not both zero.
By Proposition \ref{prop:Nak}, we have
\begin{equation}\label{eq:wxex1wx2}
w_X(A,B)\ll_\epsilon X^\epsilon \left(w_X^{(1)}(A,B)+w_X^{(2)}(A,B)+w_X^{(3)}(A,B)\right).
\end{equation}

\subsubsection*{Bounding the contribution of $w_X^{(1)}(A,B)$}
Let $X$, $Y$, and $Z$ be fixed, and take $\lambda\asymp Y^{1/12}$, $t\asymp Z^{1/2}$, and $s_1,s_2\gg 1$. Set $s=(s_1,s_2)$. For
\begin{equation}\label{eq:contT1}
\sum_{(A,B)\in (t,s)\lambda\cB\cap W(\Z)^\mg}w_X^{(1)}(A,B)
\end{equation}
to be nonzero, we must have $R_2\gg 1$, which in turn implies
$X^{1/6}\gg Y^{1/4}Z^{1/2}$.
Thus, only pairs $(A,B)$ with $\Delta(A,B)\ll X^{2/3}$ are being counted. With such a strong bound on the discriminant, we will not need any savings from the fact that only pairs $(A,B)$ whose discriminants are strongly divisible by some $m_2^2$ (with $m_2>M_2$) are being counted.

Fix an element $(A,B)\in (t,s)\lambda\cB$, and let $g(x,y)=ax^3+bx^2y+cxy^2+dy^3$ denote its cubic resolvent form. The coefficients of $g(x,y)$ and the values of $g(r_1,r_2)$, for $|r_1|\ll R_1$ and $|r_2|\ll R_2$, satisfy the following bound:
\begin{equation}\label{eq:unifCRest1}
|a|\ll \frac{Y^{1/4}}{Z^{3/2}};\quad
|b|\ll \frac{Y^{1/4}}{Z^{1/2}};\quad
|c|\ll Y^{1/4}Z^{1/2};\quad
|d|\ll Y^{1/4}Z^{3/2};\quad
|g(r_1,r_2)|\ll \frac{X^{1/2}}{Y^{1/2}}.
\end{equation}
We denote the coefficients of elements $(A,B)\in W(\R)$ by $a_{ij},b_{ij}$. The action of torus elements $(t,s)\lambda$ on $W(\R)$ multiply each coefficient $c_{ij}$ by a factor which we denote by $w(c_{ij})$.
We have
\begin{equation}
w(a_{11})=\frac{\lambda} {ts_1^{4}s_2^{2}};\quad
w(a_{12})=\frac{\lambda} {ts_1s_2^{2}};\quad
w(a_{13})=\frac{\lambda s_2}{ts_1};\quad
w(a_{22})=\frac{\lambda s_1^2} {ts_2^{2}};\quad
w(b_{11})=\frac{\lambda t} {s_1^{4}s_2^{2}}.
\end{equation}
If $(A,B)\in W(\Z)$ satisfies either $a_{11}=b_{11}=0$, or $a_{11}=a_{12}=a_{13}=0$, or $a_{11}=a_{12}=a_{22}=0$, then $(A,B)$ is not generic (in the first case, $A$ and $B$ have a common rational zero, implying that the quartic ring corresponding to them is not an integral domain; in the second case, we have $\det(A)=0$ implying that the cubic resolvent ring is not an integral domain). Thus, for $(t,s_1,s_2)\lambda\cB\cap W(\Z)^\gen$ to be nonempty, we must have $w(b_{11})\gg 1$, $w(a_{13})\gg 1$ and $w(a_{22})\gg 1$ (which in turn implies that $w(c_{ij})\gg 1$ for all $c_{ij}$ other than $a_{11}$ and $a_{12}$). These weight conditions will be used to constrain the ranges of $s_1$ and $s_2$ in terms of $t$ and $\lambda$.

We obtain bounds on \eqref{eq:contT1} using two methods. In the first method, we use the trivial bound $N(k)\ll_\epsilon |k|^{1/2+\epsilon}$. This yields the estimate 
\begin{equation*}
w_X^{(1)}(A,B)\ll_\epsilon R_1R_2(X/Y)^{1/4+\epsilon}\ll_\epsilon \frac{X^{7/12+\epsilon}}{Y^{3/4}},
\end{equation*}
where we are using the bound on $|g(r_1,r_2)|$ from \eqref{eq:unifCRest1}. Therefore, using Proposition \ref{prop-davenport} to bound the number of generic elements in $(t,s)\lambda\cB\cap W(\Z)$, we obtain
\begin{equation*}
\begin{array}{rcl}
\displaystyle\sum_{(A,B)\in (t,s)\lambda\cB\cap W(\Z)^\mg} w_X^{(1)}(A,B)
&\ll_\epsilon& \displaystyle
X^{7/12+\epsilon}Y^{-3/4}\bigl(\lambda^{12}+\lambda^{11}ts_1^4s_2^2+\lambda^{10}t^2s_1^5s_2^4\bigr)\\
&\ll& \displaystyle
X^{7/12+\epsilon}\left(Y^{1/4}+ts_1^4s_2^2Y^{1/6}+t^2s_1^5s_2^4Y^{1/12}\right).
\end{array}
\end{equation*}
Noting that we have $Z\ll Y^{1/6}$, and integrating this over $t$, $s$, and $\lambda$ yields
\begin{equation}\label{eq:unifC1M1}
\begin{array}{rcl}
\displaystyle\int_{\substack{\lambda\asymp Y^{1/12}\\t\asymp Z^{1/2}\\s_1,s_2\gg 1}}
\Bigl(\sum_{(A,B)\in (t,s)\lambda\cB\cap W(\Z)^\mg}\!\!\!\!\!\!w_X^{(1)}(A,B)\Bigr)\frac{d^\times s d^\times t d^\times \lambda}{s_1^6s_2^6t^2}
&\ll_\epsilon& \displaystyle
X^{7/12+\epsilon}\Bigl(\frac{Y^{1/4}}{Z^2}+\frac{Y^{1/6}}{Z}+Y^{1/12}\Bigr)\\
&\ll& X^{7/12+\epsilon}Y^{1/4}.
\end{array}
\end{equation}
Since $Y\ll X^{2/3}$, we are already at the border of what is needed.

\medskip

Our second method is based on the following idea. As $A$, $B$, $r_1$, and $r_2$ vary, the value of $g(r_1,r_2)=4\det(Ar_1+Br_2)$ should equidistribute. That would mean that our estimate $N(k)\ll_\epsilon |k|^{1/2+\epsilon}$ is inefficient most of the time. (For example, when $k$ is squarefree, we have $N(k)=1!)$ To realize this idea, we fiber first over $r_1$ and $r_2$, then over possible values of $k=4\det(Ar_1+Br_2)$ (with each $k$ weighted by $N(k)$). For each triple $(r_1,r_2,k)$, we count the number of possible matrices $Ar_1+Br_2$ with determinant $k/4$ (using, crucially, \cite[Theorem 4.8]{unitmono}). Finally, for fixed $r_1$, $r_2$, and $Ar_1+Br_2$, we estimate the number of $(A,B)$: Namely, we bound by considering
\begin{equation*}
\begin{array}{rcl}
\displaystyle
\sum_{(A,B)}w_X^{(1)}(A,B)
&=&
\displaystyle
\sum_{(A,B)}\sum_{(r_1,r_2)}N(4\det(Ar_1+Br_2))
\\
&=&
\displaystyle
\sum_{(r_1,r_2)}\sum_{k}N(k)\sum_{\substack{C\in S(\Z)\\4|\det(C)|=k}}
\#\{(A,B): Ar_1+Br_2=C\}.
\end{array}
\end{equation*}

There are $\ll R_1R_2$ choices for the pair $(r_1,r_2)$, and from the last bound in \eqref{eq:unifCRest1},
$k$ ranges $0\neq|k|\ll X^{1/2}/Y^{1/2}$. We estimate the number of $C=Ar_1+Br_2$ with $\det(C)=k/4$ as follows:
Let $S$ denote the space of $3\times 3$ symmetric matrices. Let $\cB'\subset S(\R)$ be a bounded set of $3\times 3$-symmetric matrices. Let $k\neq 0$ be any fixed integer.  For $T>0$, we write 
\begin{equation*}
\begin{array}{lcl}
\displaystyle N_k(s_1,s_2;T)&:=&\bigl\{C\in(s_1,s_2)T\cB'\cap S(\Z):4\det(C)=k\bigr\};
\end{array}
\end{equation*}
Then we have the following result proved in \cite[Theorem 4.8]{unitmono}.
\begin{proposition}\label{prop:symcount}
For a real number $T>1$ and integer $k\neq 0$, we have
\begin{equation*}
\displaystyle N_k(s_1,s_2;T)\ll_\epsilon
(s_1^3+s_2^3)T^{3+\epsilon}+s_1^4s_2^5T^{2+\epsilon},
\end{equation*}
independent of $k$.
\end{proposition}
Note that for $(A,B)\in (t,1,1)\lambda\cB$, and $r_1\ll R_1$ and $r_2\ll R_2$, the coefficients of $r_1A+r_2B$ are $\ll X^{1/6}/Y^{1/6}$. Therefore, the number of integral symmetric matrices $C=r_1A+r_2B$, with $(A,B)\in (t,s)\lambda\cB$ satisfying $4\det(C)=k$, is
\begin{equation*}
\ll_\epsilon (s_1^3+s_2^3)\frac{X^{1/2+\epsilon}}{Y^{1/2}}+s_1^4s_2^5\frac{X^{1/3+\epsilon}}{Y^{1/3}}.
\end{equation*}
The number of choices for the coefficients
$b_{11}$, $b_{12}$, $a_{13}$, $a_{22}$, $a_{23}$, and $a_{33}$
of $(A,B)\in (t,s)\lambda\cB\cap W(\Z)^\mg$ is $\ll Y^{1/2}/Z$,
since the ranges of these coefficients are all $\gg 1$.
The condition $r_1r_2\neq 0$ implies that the choice of these coefficients, along with $Ar_1+Br_2$ determines $(A,B)$.
Putting this together and applying Lemma \ref{lem:sumdar}, we obtain
\begin{equation*}
\begin{array}{rcl}
\displaystyle\sum_{(A,B)\in (t,s)\lambda\cB\cap W(\Z)^\mg}\!\!\!\!\!\!w_X^{(1)}(A,B)
&\ll_\epsilon& \displaystyle R_1R_2\frac{X^{1/2}}{Y^{1/2}}\Big((s_1^3+s_2^3)\frac{X^{1/2+\epsilon}}{Y^{1/2}}+s_1^4s_2^5\frac{X^{1/3+\epsilon}}{Y^{1/3}}\Big)\frac{Y^{1/2}}{Z}
\\[.2in]&=&\displaystyle
(s_1^3+s_2^3)\frac{X^{4/3+\epsilon}}{YZ}+s_1^4s_2^5\frac{X^{7/6+\epsilon}}{Y^{5/6}Z}.
\end{array}
\end{equation*}
Integrating this over $t$, $s$, and $\lambda$ yields
\begin{equation}\label{eq:unifC1M2}
\int_{\substack{\lambda\asymp Y^{1/12}\\t\asymp Z^{1/2}\\s_1,s_2\gg 1}}
\Bigl(\sum_{(A,B)\in (t,s)\lambda\cB\cap W(\Z)^\mg}\!\!\!\!\!\!w_X^{(1)}(A,B)\Bigr)\frac{d^\times s d^\times t d^\times \lambda}{s_1^6s_2^6t^2}\ll_\epsilon 
\frac{X^{4/3+\epsilon}}{YZ^2}+\frac{X^{7/6+\epsilon}}{Y^{5/6}Z^2}\ll\frac{X^{4/3+\epsilon}}{Y}.
\end{equation}
We combine \eqref{eq:unifC1M1} and \eqref{eq:unifC1M2} to obtain
\begin{equation}\label{eq:unifCase1}
\begin{array}{rcl}
\displaystyle\int_{\substack{\lambda\asymp Y^{1/12}\\t\asymp Z^{1/2}\\s_1,s_2\gg 1}}
\Bigl(\sum_{(A,B)\in (t,s)\lambda\cB\cap W(\Z)^\mg}\!\!\!\!\!\!w_X^{(1)}(A,B)\Bigr)\frac{d^\times s d^\times t d^\times \lambda}{s_1^6s_2^6t^2}
&\ll_\epsilon &\displaystyle
X^\epsilon\min\{X^{7/12}Y^{1/4},\frac{X^{4/3}}{Y}\}
\\[.2in]\displaystyle
&\leq& X^{11/15+\epsilon},
\end{array}
\end{equation}
which is sufficiently small.

\subsubsection*{Bounding the contribution of $w_X^{(2)}(A,B)$}
As before, let $X$, $Y$, and $Z$ be fixed, and take $\lambda\asymp Y^{1/12}$, $t\asymp Z^{1/2}$, and $s_1,s_2\gg 1$. Set $s=(s_1,s_2)$, and consider $(A,B)\in (t,s)\lambda\cB\cap W(\Z)^\mg_{M_2}$. We have
\begin{equation}\label{eq:contT2}
w_X^{(2)}(A,B)=\sum_{|r|\ll R_2}N(4\det(B)r^3).
\end{equation}
Fiber over $r$ (going up to $R_2$ in size), and $d=\det(B)$ (going up to $Y^{1/4}Z^{3/2}$ in size). For each fixed $d$, the number of choices for $B$ is 
\begin{equation*}
\ll_\epsilon (s_1^3+s_2^3)Y^{1/4+\epsilon}Z^{3/2}+s_1^4s_2^5 Y^{1/6+\epsilon}Z
\end{equation*}
from Proposition~\ref{prop:symcount}. Meanwhile, the number of choices for $A$ is bounded depending on the ranges of $s_1$ and $s_2$. If $w(a_{11})\gg 1$, then there are $\ll Y^{1/2}/Z^3$ choices for $A$; if $w(a_{11})<1$ but $w(a_{12})\gg 1$, then there are $\ll Y^{1/2}/Z^3\cdot w(b_{11})/w(a_{11})$ choices for $A$; if $w(a_{11})<1$ and $w(a_{12})< 1$, then there are $\ll Y^{1/2}/Z^3\cdot w(b_{11})w(b_{12})/(w(a_{11})w(a_{12}))$ choices for $A$. Here, we are allowed to multiply by $w(b_{11})$ and $w(b_{12})$ since they are both $\gg 1$. This implies that the number of choices for $A$ is $\ll Y^{1/2}/Z$.
Therefore, applying Lemma \ref{lem:sumdar}, we have
\begin{equation*}
\begin{array}{rcl}
\displaystyle\sum_{(A,B)\in (t,s)\lambda\cB\cap W(\Z)^\mg}\sum_{|r|\ll R_2}N(4\det(B)r^3)
&\ll_\epsilon&\displaystyle
X^\epsilon R_2^2Y^{3/4}Z^{1/2}\bigl((s_1^3+s_2^3)Y^{1/4}Z^{3/2}+s_1^4s_2^5 Y^{1/6}Z\bigr)
\\[.2in]&=&\displaystyle
X^{1/3+\epsilon} \left((s_1^3+s_2^3)Y^{1/2}Z+s_1^4s_2^5Y^{5/12}Z^{1/2}\right).
\end{array}
\end{equation*}
Integrating this over $s$, $t$, and $\lambda$, and using the bound $Y\ll X^{2/3}$ implies that we have
\begin{equation}\label{eq:unifCase2P1}
\int_{\substack{\lambda\asymp Y^{1/12}\\t\asymp Z^{1/2}\\s_1,s_2\gg 1}}
\Bigl(\sum_{(A,B)\in (t,s)\lambda\cB\cap W(\Z)^\mg}w_X^{(2)}(A,B)\Bigr)\frac{d^\times s d^\times t d^\times \lambda}{s_1^6s_2^6t^2}\ll_\epsilon X^{2/3+\epsilon},
\end{equation}
which is sufficiently small.

\subsection{Bounding the contribution of $w_X^{(3)}$ using a uniform Ekedahl sieve}

Finally, we obtain upper bounds on
\begin{equation}\label{eq:UnifMT}
\sum_{(A,B)\in (t,s)\lambda\cB\cap W(\Z)^\mg_{M_2}}w_X^{(3)}(A,B) = \sum_{(A,B)\in (t,s)\lambda\cB\cap W(\Z)^\mg_{M_2}}\sum_{|r|\ll R_1}N(4\det(A)r^3).
\end{equation}
This is the region from which we expect the biggest contribution. In the previous estimates \eqref{eq:unifCase1} and \eqref{eq:unifCase2P1} we had $R_2\gg 1$ which implies $Y\ll X^{2/3}$. Thus, in particular we already have $M_1\gg X^{1/6}$ and did not need any savings from $M_2$. In this final term \eqref{eq:UnifMT}, the situation is very different. In particular, we have to consider values of $Y$ going all the way up to $X$, in which case all our savings will come from $M_2$.

Our strategy to bound \eqref{eq:UnifMT} is similar to our previous methods. We fiber over $r$ (going up to $R_1$)
and $a=4\det(A)$ (going up to $Y^{1/4}/Z^{3/2}$). 
As before, once $a\neq 0$ is fixed, the number of $A$ that can arise with $4\det(A)=a$ is bounded in Proposition \ref{prop:symcount} by
\begin{equation*}
(s_1^3+s_2^3)\frac{Y^{1/4+\epsilon}}{Z^{3/2}}+s_1^4s_2^5 \frac{Y^{1/6+\epsilon}}{Z}.
\end{equation*}
Once $A$ has been fixed (with $4\det(A)=a$), we bound the number of possible $B$'s with $(A,B)\in (t,s)\lambda\cB\cap W(\Z)_{M_2}^\mg$ using the Ekedhal sieve, as developed in \cite{geosieve}. The first thing we note is that if $p\mid A$ for some prime $p$, then $(A,B)$ is not maximal for any $B$. Therefore, we can assume that the $\F_p$-rank of $A$ is $1$, $2$, or $3$ for every prime $p$. Suppose first that $p\mid a=\det(A)$, implying that the $\F_p$-rank of $A$ is $1$ or $2$. In this case, the condition $p^2\mid\Delta(A,B)$ for mod $p$ reasons is at least a codimension $1$ condition on $B$. Moreover, it is easy to see that the polynomial cutting out this condition involves one of the coefficients $b_{22}$, $b_{23}$, and $b_{33}$. Next suppose that $p\nmid a$. In this case, the condition $p^2\mid\Delta(A,B)$ for mod $p$ reasons is a codimension $2$-condition on $B$, and we may apply the Ekedahl sieve from \cite[Proof of Theorem 3.3]{geosieve} without change. In other words, we proceed as follows. We have fixed $A$ with $4\det(A)=a$. We are interested in counting the number of $B$'s in a certain domain such that $m^2$ divides the discriminant of $(A,B)$ for mod $m$ reasons, where $m$ is squarefree and $m>M_2$.
We fiber over the possible values $m'$ of the gcd of $a$ and $m$. Each such $m'$ imposes a mod-$m'$ condition on one of $b_{22}$, $b_{23}$, and $b_{33}$ (the smallest range of which is that of $b_{22}$). Moreover, there must exist an $m''\geq M_2/m'$ which imposes a mod-$m''$ codimension-2 condition on $B$ (involving $b_{33})$. Carrying this out, we obtain
\begin{equation*}
\begin{array}{rcl}
&&\displaystyle\#\bigl\{B:(A,B)\in (t,s)\lambda\cB\cap W(\Z)^\mg_{M_2}\bigr\}
\\[.1in]&\ll_\epsilon&\displaystyle
X^\epsilon
\sum_{\substack{m'\mid a\\|\mu(m')|=1}}w(b_{11})w(b_{12})w(b_{13})\max\Bigl(\frac{w(b_{22})}{m'},1\Bigr)
w(b_{23})\max\Bigl(\frac{w(b_{33})}{\max(M_2/m',1)},1\Bigr)
\\[.1in]&\leq&\displaystyle
X^\epsilon
\sum_{\substack{m'\mid a\\|\mu(m')|=1}}
\Bigl(\frac{\lambda^6t^6}{M_2}+\frac{\lambda^5t^5s_1^{-2}s_2^{-4}}{m'}+\frac{\lambda^5t^5s_1^{-2}s_2^2}{\max(M_2/m',1)}+\lambda^4t^4s_1^{-4}s_2^{-2}\Bigr)
\\[.2in]&\ll_\epsilon&\displaystyle
X^\epsilon
\Bigl(
\frac{Y^{1/2}Z^3}{M_2}+s_1^{-2}s_2^2Y^{5/12}Z^{5/2}
\Bigr).
\end{array}
\end{equation*}
Combining the above discussion with the bound 
\begin{equation*}
\sum_{a\ll Y^{1/4}/Z^{3/2}}\sum_{r\ll R_1}N(ar^3)\ll_\epsilon X^\epsilon R_1^2Y^{1/4}/Z^{3/2}
\end{equation*}
implied by Lemma \ref{lem:sumdar}, we obtain
\begin{equation*}
\begin{array}{rcl}
&&\displaystyle\sum_{(A,B)\in (t,s)\lambda\cB\cap W(\Z)_{M_2}^\mg}\sum_{0\neq|r|\ll R_1}N(4\det(A)r^3)
\\[.2in]&\ll_\epsilon&\displaystyle
X^\epsilon \frac{R_1^2 Y^{1/4}}{Z^{3/2}}
\Bigl(s_1^3s_2^3\frac{Y^{1/4}}{Z^{3/2}}+s_1^4s_2^5\frac{Y^{1/6}}{Z}\Bigr)
\Bigl(\frac{Y^{1/2}Z^3}{M_2}+s_1^{-2}s_2^2Y^{5/12}Z^{5/2}
\Bigr)
\\[.2in]&=&\displaystyle
X^\epsilon \Bigl(
\frac{s_1^3s_2^3X^{1/3}Y^{1/2}Z}{M_2}+
s_1^4s_2^5\frac{X^{1/3}Y^{5/12}Z^{3/2}}{M_2}+
s_1s_2^5X^{1/3}Y^{5/12}Z^{1/2}+
s_1^2s_2^7X^{1/3}Y^{1/3}Z
\Bigr).
\end{array}
\end{equation*}
Recall that we always have $Z\ll Y^{1/6}$ and $R_1\gg 1$. The latter condition implies $Z^{1/2}\gg Y^{1/4}/X^{1/6}$. Multiplying the last summand in the above displayed equation by $1\ll \sqrt{w(a_{22})}\asymp s_1s_2^{-1}Y^{1/12}Z^{-1/2}$, and integrating over $\lambda$, $t$, and $s$ yields
\begin{equation}\label{eq:unifCase2P2}
\begin{array}{rcl}
&&\displaystyle\int_{\substack{\lambda\asymp Y^{1/12}\\t\asymp Z^{1/2}\\s_1,s_2\gg 1}}
\Bigl(\sum_{(A,B)\in (t,s)\lambda\cB\cap W(\Z)^\mg}\sum_{|r|\ll R_1}N(4\det(A)r^3)\Bigr)\frac{d^\times s d^\times t d^\times \lambda}{s_1^6s_2^6t^2}
\\[.2in]&\ll_\epsilon&\displaystyle
\frac{X^{1/3+\epsilon}Y^{1/2}}{M_2}+\frac{X^{1/3+\epsilon}Y^{5/12}}{Z^{1/2}}
\\[.2in]&\ll_\epsilon&\displaystyle
\frac{X^{1/3+\epsilon}Y^{1/2}}{M_2}+
X^{1/2+\epsilon}Y^{1/6}.
\end{array}
\end{equation}
We are now ready to prove the main result of this section.

\medskip

\noindent{\bf Proof of Theorem \ref{Thm:Unif}:} Combining \eqref{eq:unifCase1}, \eqref{eq:unifCase2P1}, and \eqref{eq:unifCase2P2} yields
\begin{equation*}
\int_{\substack{\lambda\asymp Y^{1/12}\\t\asymp Z^{1/2}\\s_1,s_2\gg 1}}
\Bigl(\sum_{(A,B)\in (t,s)\lambda\cB\cap W(\Z)^\mg_{M_2}}\!\!\!\!\!\!w_X(A,B)\Bigr)\frac{d^\times s d^\times t d^\times \lambda}{s_1^6s_2^6t^2}\ll_\epsilon  
\frac{X^{1/3+\epsilon}Y^{1/2}}{M_2}+
X^{1/2+\epsilon}Y^{1/6}+X^{11/15+\epsilon}.
\end{equation*}
From Proposition \ref{prop:NXM1M2}, \eqref{eq:unifQavg}, and summing the above equation over $Y$ and $Z$ in the dyadic ranges $Y\leq X/M_1^2$ and $Z\ll Y^{1/6}$, we obtain
\begin{equation*}
N(X;M_1,M_2)\ll_\epsilon  
\frac{X^{5/6+\epsilon}}{M_1M_2}+
\frac{X^{2/3+\epsilon}}{M_1^{1/3}}+X^{11/15+\epsilon}.
\end{equation*}
Finally, applying \eqref{eq:firstinc} completes the proof of Theorem \ref{Thm:Unif}.
$\Box$

\section{Proofs of the main results}\label{sec:proofs}

In this section, we prove our main results.

\subsection{Summing large and locally well approximated functions over $\PGL_2(\Z)\backslash V(\Z)$}

In this section, we prove the following result.

\begin{theorem}\label{thm:NphiXM}
Let $\phi$ be a large and locally well approximated function.  Then, for $i\in\{0,1,2+,2-\}$, we have
\begin{equation*}
N^{(i)}(\phi,X)=M_{5/6}^{(i)}(\phi)X^{5/6}+M_{3/4}^{(i)}(\phi)X^{3/4}+O_\epsilon(X^{3/4-1/3804+\epsilon}).
\end{equation*}
\end{theorem}
\noindent We assume that $\phi$ is fixed $($and suppress the dependence on this constant in all the error terms in this section$)$.

\medskip

\begin{proof}
The proof of this result combines the methods and results of \S4, \S6, and \S7, and is carried out in the following steps.

\medskip

\noindent {\bf Step 1: Cutting off the tail.}
We begin by noting that for
$f\in V(\Z)\backslash\{\Delta\neq 0\}$, we have
\begin{equation*}
\phi(f)=\prod_p\phi_p(f)=\sum_{n\geq 1}\phi(n;f),\quad\mbox{where}\quad \phi(n;f):=\prod_{p^k\parallel n}\phi_p^{(k)}(f).
\end{equation*}
The function $\phi(n;\cdot)$ is defined modulo $n^2$ and supported on the set of elements $f\in V(\Z)$ with $n^2\mid C\Delta(f)$ for a fixed positive integer $C$. Let $\delta>0$ be a constant to be optimized later. Applying the uniformity estimate in Theorem \ref{Thm:Unif}, we obtain
\begin{equation}\label{eq:Mtailcut}
\sum_{n\geq X^{1/12+\delta}}N^{(i)}(\phi(n;\cdot);X)\ll_\epsilon X^{3/4-\delta+\epsilon}+X^{3/4-1/60+\epsilon},
\end{equation}
since $\sum_{n\geq1}|\phi(n,f)|\leq\sum_{n^2\mid\Delta(f)}1\ll_\epsilon|\Delta(f)|^\epsilon$.
As in \S4, set $Y:=X^{1/6}$. Combining \eqref{eq:avg} and \eqref{eq:Mtailcut}, we have
\begin{equation}\label{eq:Mtailcut1}
\begin{array}{rcl}
N^{(i)}(\phi;X)&=&\displaystyle\frac{1}{\sigma_i\Vol(G_0)}\sum_{n< X^{1/12+\delta}}\cI(\phi(n;\cdot),Y)+O_\epsilon(X^{3/4-\delta+\epsilon}+X^{3/4-1/60+\epsilon})
\\[.2in]
&=&\displaystyle\frac{1}{\sigma_i\Vol(G_0)}\sum_{n< X^{1/12+\delta}}\bigl(\cI^{(1)}(\phi(n;\cdot),Y;\kappa)+\cI^{(2)}(\phi(n;\cdot),Y;\kappa)\bigr)
\\[.2in]
&&\displaystyle +O_\epsilon(X^{3/4-\delta+\epsilon}+X^{3/4-1/60+\epsilon}),
\end{array}
\end{equation}
for some $0<\kappa<1/4$ to be chosen later, where the quantities $\cI(\cdot,Y)$ and $I^{(i)}(\cdot,Y;\kappa)$ are defined in \eqref{eq:tcIdef} and \eqref{eq:cIdef}, respectively. We refer to the sum over $\cI^{(1)}(\phi(n;\cdot),Y;\kappa)$ as the main ball contribution to $N^{(i)}(\phi;X)$ and to the sum over $\cI^{(2)}(\phi(n;\cdot),Y;\kappa)$ as the cuspidal contribution.

\medskip
\medskip

\noindent {\bf Step 2: Estimates for the main ball. Part 1: $n$ has large squarefree part.} Let $n< X^{1/12+\delta}$ be a positive integer. Denote the squarefree part of $n$ by $n_1$.
We use the bounds on the Fourier coefficients of $\phi(n;\cdot)$ obtained in Proposition~\ref{prop:FT} to estimate the quantity
\begin{equation*}
N_n(S^{(i)};Y,(u,t),\phi):=
\sum_{f\in Y(u,t)S^{(i)}\cap V(\Z)}\phi(n;f)
\end{equation*}
for $u\in [-1/2,1/2]$ and $t\gg 1$. Note that we have $(u,t)S^{(i)}=(t)\cdot ((u/t^2,1)\cdot S^{(i)})$, and that $u/t^2$ is absolutely bounded.

Let $\eta>0$ be a sufficiently small real number, to be optimized later. Define the set $S^\sharp\subset S:=(u/t^2,1)S^{(i)}$ to be the (compact) set of elements $v\in \R^n$ such that $|v-w|\geq\eta$ for all elements $w$ in the boundary of $S$. Define the set $S^\flat \supset S$ to be the (open) set of elements $v\in \R^n$ such that $|v-w|<\eta$ for some $w$ in $\overline{S}$, the closure of $S$. Then there exist $C^\infty$ functions $\Psi^\sharp:=\Psi^\sharp_{S,\eta}:\R^n\to\R_{\geq 0}$ and $\Psi^\flat:=\Psi^\flat_{S,\eta}:\R^n\to\R_{\geq 0}$ satisfying the following properties.
\begin{itemize}
\item[{\rm (a)}] The function $\Psi^\sharp$ is $1$ in a neighbourhood of $S^\sharp$ and its support is contained in $S$.
\item[{\rm (b)}] The function $\Psi^\flat$ is $1$ in a neighbourhood of $\overline{S}$ and its support is contained in $S^\flat$.
\item[{\rm (c)}] The partial derivatives of $\Psi^\sharp$ and $\Psi^\flat$ satisfy the following bounds:
\begin{equation*}
|\partial^\alpha(\Psi^\sharp)|,|\partial^\alpha(\Psi^\flat)|\ll_{S,\alpha} \eta^{-|\alpha|}.
\end{equation*}
\item[{\rm (d)}] The Fourier transforms of $\Psi^\sharp$ and $\Psi^\flat$
satisfy the following bounds:
\begin{equation*}
|\widehat{\Psi^\sharp}(w)|,|\widehat{\Psi^\flat}(w)|\ll_M \min\big(1, (\eta|w|)^{-M}\big),
\end{equation*}
for $M>0$.
\end{itemize}
Properties (a), (b), and (c) are consequences of \cite[Theorem 1.4.1, Equation (1.4.2)]{Hormander}. Property~(d) is a standard consequence of Property (c). Furthermore, since $u/t^2$ is absolutely bounded, the above error terms are independent of $u$.
Define the auxilliary counting functions
\begin{equation*}
\begin{array}{rcl}
N_n^\sharp(S;Y,(t),\phi)&:=&\displaystyle\sum_{f\in V(\Z)}\Psi^{\sharp}\Bigl(\frac{(t)^{-1}f}{Y}\Bigr)\phi(n;f),\quad
\\[.3in]
N_n^{{\rm err}}(S;Y,(t),\phi)&:=&\displaystyle\sum_{f\in V(\Z)}(\Psi^{\flat}-\Psi^{\sharp})\Bigl(\frac{(t)^{-1}f}{Y}\Bigr)|\phi(n;f)|,
\end{array}
\end{equation*}
and note that we have
\begin{equation}\label{eq:finalstep21}
N_n(S^{(i)};Y,(u,t),\phi(n;))=N_n^\sharp(S;Y,(t),\phi)+O(N_n^{{\rm err}}(S;Y,(t),\phi)).
\end{equation}
We use twisted Poisson summation to write
\begin{equation*}
N_n^\sharp(S;Y,(t),\phi)=
Y^5\sum_{w\in {V^\ast(\Z)}}\widehat{\Psi^\sharp}
\Bigl(\frac{(t)\cdot Yw}{n^2}\Bigr)
\widehat{\phi(n;\cdot)}(w).
\end{equation*}
From Property (d) of the function $\Psi^\sharp$, it follows that, up to negligible error, we can restrict the above sum to $w=(a,b,c,d,e)\in V^\ast(\Z)$ satisfying 
\begin{equation*}
|a|\ll_\epsilon \frac{t^4n^2}{\eta Y^{1-\epsilon}};
\quad
|b|\ll_\epsilon \frac{t^2n^2}{\eta Y^{1-\epsilon}};
\quad
|c|\ll_\epsilon \frac{n^2}{\eta Y^{1-\epsilon}};
\quad
|d|\ll_\epsilon \frac{n^2}{\eta t^2 Y^{1-\epsilon}};
\quad
|e|\ll_\epsilon \frac{n^2}{\eta t^{4} Y^{1-\epsilon}}.
\quad
\end{equation*}
Indeed, for the rest of the $w$'s, we obtain the necessary saving from the superpolynomial decay of $\widehat{\Psi^\sharp}$. Hence, upto a negligible error, we have
\begin{equation}\label{eq:finalstep22}
\begin{array}{rcl}
|N_n^\sharp(S;Y,(t),\phi)-\nu(\phi(n;))\widehat{\Psi^\sharp}(0)Y^5|&\ll_\epsilon& \displaystyle \eta^{-5}t^6X^{5/6+10\delta+\epsilon}\cdot\frac{\nu(\phi(n;))}{n_1},
\\[.2in]
|N_n^{{\rm err}}(S;Y,(t),\phi)-\nu(|\phi(n;)|)(\widehat{\Psi^\flat}(0)-\widehat{\Psi^\sharp}(0))Y^5| &\ll_\epsilon& \displaystyle \eta^{-5}t^6X^{5/6+10\delta+\epsilon}\cdot\frac{\nu(|\phi(n;)|)}{n_1},
\end{array}
\end{equation}
where we use Corollary \ref{prop:FT} to estimate $\widehat{\phi(n;)}$ and $\widehat{|\phi|(n;)}$. Let $\chi_{n^2}$ denote the characteristic function of the set of elements $f$ such that $n^2\mid\Delta(f)$,
and note that some bounded constant multiple of  $\chi_{n^2/(c,n^2)}$
dominates $|\phi(n;)|$, for some positive integer $c$ determined by $\phi$.
Combining \eqref{eq:finalstep21} and \eqref{eq:finalstep22}, we obtain
\begin{equation*}
\begin{array}{rcl}
N_n(S^{(i)};Y,(t),\phi)&=&\nu(\phi(n;))\Vol(S^{(i)})X^{5/6}
\\[.1in]
&&\displaystyle+O\bigl(\nu(\chi_{n^2})\eta X^{5/6}\bigr)+O_\epsilon\Bigl(t^6\eta^{-5}\frac{\nu(\chi_{n^2})}{n_1}X^{5/6+10\delta+\epsilon}\Bigr).
\end{array}
\end{equation*}
Finally, integrating the left hand side of the above equation over $u$ and $t$, we obtain the following estimate for $\cI^{(1)}(\phi(n;),Y;\kappa)$:
\begin{equation}\label{eq:finalpart2Nnest}
\begin{array}{rcl}
\displaystyle\cI^{(1)}(\phi(n;\cdot),Y;\kappa)&=&\displaystyle\nu(\phi(n;\cdot))\Vol(S^{(i)})X^{5/6}\int_{t\geq \frac{\sqrt[4]{3}}{\sqrt{2}}}\int_{u\in N(t)} \psi_1\Bigl(\frac{t}{Y^{\kappa}}\Bigr)du\frac{d^\times t}{t^{2}}
\\[.2in]&&\displaystyle
+O\bigl(\nu(\chi_{n^2})\eta X^{5/6}\bigr)+O_\epsilon\Bigl(\eta^{-5}\frac{\nu(\chi_{n^2})}{n_1}X^{5/6+10\delta+2\kappa/3+\epsilon}\Bigr).
\end{array}
\end{equation}

\medskip
\medskip

\noindent {\bf Step 3: Estimates for the main ball. Part 2: $n$ has small squarefree part.}
When $n$ has small squarefree part, we simply use \eqref{eq:smallt} to write
\begin{equation}\label{eq:finalpart3Nnest}
\begin{array}{rcl}
\displaystyle\cI^{(1)}(\phi(n;\cdot),Y;\kappa)&=&\displaystyle\nu(\phi(n;\cdot))\Vol(S^{(i)})X^{5/6}\int_{t\geq \frac{\sqrt[4]{3}}{\sqrt{2}}}\int_{u\in N(t)} \psi_1\Bigl(\frac{t}{Y^{\kappa}}\Bigr) du\frac{d^\times t}{t^{2}}
\\[.2in]&&\displaystyle
+\supp(\phi(n;))O\Bigl(\frac{X^{2/3+\kappa/3}}{n^8}+\frac{X^{1/2+2\kappa/3}}{n^6}+X^{2\kappa/3}\Bigr).
\\[.2in]
&=&\displaystyle\nu(\phi(n;\cdot))\Vol(S^{(i)})X^{5/6}\int_{t\geq \frac{\sqrt[4]{3}}{\sqrt{2}}}\int_{u\in N(t)} \psi_1\Bigl(\frac{t}{Y^{\kappa}}\Bigr) du\frac{d^\times t}{t^{2}}
\\[.2in]&&\displaystyle
+O\Bigl( n^2\nu(\chi_{n^2})X^{2/3+\kappa/3}
+n^4\nu(\chi_{n^2})X^{1/2+2\kappa/3}
+n^{10}\nu(\chi_{n^2})X^{2\kappa/3}\Bigr).
\end{array}
\end{equation}

\medskip
\medskip

\noindent {\bf Step 4: Estimates for the main ball. Part 3: Summing over $n$.} Write $n=n_1n_2^2m_3$, where $n_1$, $n_2$, and $m_3$ are pairwise coprime, $n_1$ and $n_2$ are squarefree, and $m_3$ is cubefull. Fix $\theta>0$ to be optimized later. We sum $\cI^{(1)}(\phi(n;\cdot),Y;\kappa)$ over $1\leq n<X^{1/12+\delta}$, using \eqref{eq:finalpart2Nnest} when $n_1> X^{1/12-\theta}$ and using \eqref{eq:finalpart3Nnest} when $n_1\leq X^{1/12-\theta}$. This yields
\begin{equation}\label{eq:phiFinalMainBall1}
\sum_{1\leq n<X^{1/12+\delta}}{\cI^{(1)}}(\phi(n;\cdot),Y;\kappa)=\nu(\phi)\Vol(S^{(i)})X^{5/6}\int_{t\geq \frac{\sqrt[4]{3}}{\sqrt{2}}}\int_{u\in N(t)} \psi_1\Bigl(\frac{t}{Y^{\kappa}}\Bigr) du\frac{d^\times t}{t^{2}}+O_\epsilon({\rm error}),
\end{equation}
where the error term is given by
\begin{equation*}
\begin{array}{rcl}
{\rm error}=&&
\displaystyle\sum_{\substack{n< X^{1/12+\delta}\\n_1>X^{1/12-\theta}}}\Bigl(\nu(\chi_{n^2})\eta X^{5/6}+\eta^{-5}\frac{\nu(\chi_{n^2})}{n_1}X^{5/6+10\delta+2\kappa/3+\epsilon}\Bigr)
\\[.25in]&+&\displaystyle
\sum_{\substack{n< X^{1/12+\delta}\\n_1\leq X^{1/12-\theta}}}
\Bigl( n^2\nu(\chi_{n^2})X^{2/3+2\kappa/3}
+n^4\nu(\chi_{n^2})X^{1/2+2\kappa/3}+n^{10}\nu(\chi_{n^2})X^{2\kappa/3}\Bigr)
\\[.25in]&+&\displaystyle
X^\epsilon\sum_{n\geq X^{1/12+\delta}}\nu(\chi_{n^2})X^{5/6}.
\end{array}
\end{equation*}

From Proposition \ref{prop:ppden}, we have the bound $\nu(\chi_{n^2})\ll n_1^{-2}n_2^{-4}m_3^{-3/2}\log n$.
Using this, we bound the first summand of the first line in the error as being $\ll_\epsilon$
\begin{equation}\label{eq:tempnsum1}
\eta X^{5/6+\epsilon}
\sum_{n_1>X^{1/12-\theta}}n_1^{-2}
\sum_{n_2^2\ll \frac{X^{1/12+\delta}}{n_1}}n_2^{-4}
\sum_{m_3\ll\frac{X^{1/12+\delta}}{n_1n_2^2}}m_3^{-3/2}
\ll \eta X^{3/4+\theta+\epsilon}.
\end{equation}
Similarly, the second summand in the first line is $\ll_\epsilon$
\begin{equation*}
\begin{array}{rcl}
&&\displaystyle\eta^{-5} X^{5/6+10\delta+2\kappa/3+\epsilon}
\sum_{n_1>X^{1/12-\theta}}n_1^{-3}
\sum_{n_2^2\ll \frac{X^{1/12+\delta}}{n_1}}n_2^{-4}
\sum_{m_3\ll\frac{X^{1/12+\delta}}{n_1n_2^2}}m_3^{-3/2}
\ll
\eta^{-5} X^{2/3+2\theta+10\delta+2\kappa/3+\epsilon}.
\end{array}
\end{equation*}
To bound the three summands in the second line,
we note that $\sum_{m<M}m^\alpha=O(M^{\alpha+1/3})$ for $\alpha>0$,
where in the sum $m$ runs through all cubefull numbers less than $M$.
Thus, the three summands are respectively
$\ll_\epsilon$
\begin{equation}\label{eq:tempnsum2}
\begin{array}{rcl}
\displaystyle X^{2/3+\kappa/3+\epsilon}\sum_{n_1\leq X^{1/12-\theta}}
\sum_{n_2^2\ll \frac{X^{1/12+\delta}}{n_1}}
\sum_{m_3\ll\frac{X^{1/12+\delta}}{n_1n_2^2}}m_3^{1/2}
&\ll_\epsilon&
\displaystyle
X^{3/4-\theta/6+5\delta/6+\kappa/3+\epsilon},
\\[.5in]
\displaystyle X^{1/2+2\kappa/3+\epsilon}\sum_{n_1\leq X^{1/12-\theta}}n_1^2
\sum_{n_2^2\ll \frac{X^{1/12+\delta}}{n_1}}n_2^4
\sum_{m_3\ll\frac{X^{1/12+\delta}}{n_1n_2^2}}m_3^{5/2}
&\ll_\epsilon&
\displaystyle
X^{3/4-\theta/6+17\delta/6+2\kappa/3+\epsilon},
\\[.5in]
\displaystyle
X^{2\kappa/3}\sum_{n_1\leq X^{1/12-\theta}}n_1^8
\sum_{n_2^2\ll \frac{X^{1/12+\delta}}{n_1}}n_2^{16}
\sum_{m_3\ll\frac{X^{1/12+\delta}}{n_1n_2^2}}m_3^{17/2}
&\ll_\epsilon&
\displaystyle
X^{3/4-\theta/6+53\delta/6+2\kappa/3+\epsilon}.
\end{array}
\end{equation}
Finally, the last line in the error term is $\ll X^{3/4-\delta+\epsilon}$.

Optimizing, we pick $\eta=X^{-1/72+\theta/6+5\delta/3+\kappa/9}$ and $\theta=1/96+43\delta/8+5\kappa/12$ to obtain
\begin{equation}\label{eq:phiFinalMainBall}
\begin{array}{rcl}
\cI^{(1)}(\phi,Y;\kappa)&=&\displaystyle\nu(\phi)\Vol(S^{(i)})X^{5/6}\int_{t\geq \frac{\sqrt[4]{3}}{\sqrt{2}}}\int_{u\in N(t)}\psi_1\Bigl(\frac{t}{Y^{\kappa}}\Bigr) du\frac{d^\times t}{t^{2}}
\\[.25in]&&\displaystyle
+O_\epsilon\Bigl(X^{3/4-1/576+83\delta/24+5\kappa/72+\epsilon}
+X^{3/4-\delta+\epsilon}
\Bigr).
\end{array}
\end{equation}

\medskip
\medskip
\noindent {\bf Step 5: Estimates for the cusp. Part 1: $n$ is large.}
We begin with the following result:
\begin{lemma}\label{lemma_pre_prop}
Let $1\leq F_0,F_1,F_2, F_3, F_4$ be real numbers, and let $B_{\vec{F}}$ denote the set of elements $f(x,y)=f_0x^4+f_1x^3y+f_2x^2y^2+f_3xy^3+f_4y^4\in V(\R)$ such that $-F_i\leq f_i\leq F_i$ for each $i\in\{0,1,2,3,4\}$. For a positive integer $n$, we have
\begin{equation}\label{eq:unifskew}
\#\bigl\{f\in
V(\Z)^\gen\cap B_{\vec{F}}:a(f)b(f)\neq 0,\,n^2\mid\Delta(f)\bigr\}\ll_\epsilon 
n^\epsilon\nu(\chi_{n^2})M(\vec{F},n),
\end{equation}
where
$M(\vec{F},n):=F_0F_1(\max(F_2,n^2)\max(F_3,n^2)\max(F_4,n^2))$
\end{lemma}
\begin{proof}
For a ring $R$, and fixed elements $a\in R$ and $b\in R$, let $V(R)_{a,b}$ denote the set of elements $f\in V(R)$ with $a(f)=a$ and $b(f)=b$. For $m\geq 1$ with prime factorization $m=p_1^{k_1}\cdots p_\ell^{k_\ell}$, let $\nu_{a,b}(\chi_{m^2})$ denote the density in $V(\widehat{\Z})_{a,b}$ of the set of elements $f\in V(\widehat{\Z})_{a,b}$ with $m^2\mid\Delta(f)$. It is clear that we have 
\begin{equation*}
\nu_{a,b}(\chi_{m^2})=\prod_{i=1}^\ell\nu_{a,b}(\chi_{p_i^{2k_i}}).
\end{equation*}
Let $p$ be a prime. If $\lambda$ and $r$ are elements of $\Z_p^\times$, and $f\in V(\Z_p)$, then $\Delta(f(x,y))=\Delta(\lambda f(x,ry))$. As a consequence, it follows that for any $k\geq 1$, we have $\nu_{a,b}(\chi_{p^k})=\nu_{a',b'}(\chi_{p^k})$ as long as $v_p(a)=v_p(a')$ and $v_p(b)=v_p(b')$. 
Therefore, for nonzero elements $a$ and $b$ of $\Z_p$, we have the bound
\begin{equation*}
\nu_{a,b}(\chi_{p^k})\ll \nu(\chi_{p^k})v_p(a)v_p(b).
\end{equation*}

Returning to the sets $B_{\vec{F}}$, we simply fiber over $a$ and $b$ with $ab\neq 0$, and note that the LHS of \eqref{eq:unifskew} is equal to
\begin{equation*}
\begin{array}{rcl}
&&\displaystyle\sum_{1\leq |a|\leq F_0}
\sum_{1\leq |b|\leq F_1}
\#\bigl\{f\in V(\Z)_{a,b}^\gen\cap B_{\vec{F}}:n^2\mid\Delta(f)\bigr\}
\\[.2in]&\ll&\displaystyle
\sum_{1\leq |a|\leq F_0}\sum_{1\leq |b|\leq F_1}
\nu_{a,b}(\chi_{n^2})\max(F_2,n^2)\max(F_3,n^2)\max(F_4,n^2)
\\[.2in]&\ll&\displaystyle
\sum_{\substack{{\rm rad}(d_a)\mid n\\{\rm rad}(d_b)\mid n}}
\sum_{\substack{|a|\leq F_0,\,|b|\leq F_1\\d_a\mid a,\,d_b\mid b}}
d_ad_b\nu(\chi_{n^2})\max(F_2,n^2)\max(F_3,n^2)\max(F_4,n^2)
\\[.2in]&\ll&\displaystyle
\sum_{\substack{d_a\leq F_0, d_b\leq F_1\\{\rm rad}(d_a)\mid n,\,{\rm rad}(d_b)\mid n}}
\nu(\chi_{n^2})M(\vec{F},n)
\\[.2in]&\ll_\epsilon&\displaystyle
n^\epsilon\nu(\chi_{n^2})M(\vec{F},n),
\end{array}
\end{equation*}
as necessary.
\end{proof}

Pick $\theta>0$ to be optimized later (this $\theta$ is of course independent of the $\theta$ which was optimized in the previous step). Let $n$ be an integer such that $X^{1/12-\theta}\leq n< X^{1/12+\delta}$.
We use Proposition~\ref{lemma_pre_prop} (and Lemma \ref{lem:b0bound} for the elements $f$ with $b(f)=0$) to write 
\begin{equation}\label{eq:tempstep5}
\cI^{(2)}(\phi(n;\cdot),Y;\kappa)\ll_\epsilon\displaystyle
X^{2/3}+\int_{Y^\kappa\ll t\ll Y^{1/4} }n^\epsilon\nu(\chi_{n^2})
\frac{Y^2}{t^6}\max(Y,n^2)\max(t^2Y,n^2)\max(t^4Y,n^2)\frac{d^\times t}{t^2}.
\end{equation}
Suppose first that $n< X^{1/12}$. Then we have $Y\geq n^2$, and so the integral in the equation above is $\ll X^{5/6-\kappa/3+\epsilon}\nu(\chi_{n^2})$. Next suppose that $n\geq X^{1/12}$. We will finally choose $\kappa\geq 6\delta$, and assume this now. In that case, we have $t^2Y\gg n^2$. Thus the integral is $\ll X^{3/4-\kappa/3+\epsilon}n^2\nu(\chi_{n^2})$.
Summing this over $X^{1/12-\theta}\leq n< X^{1/12+\delta}$ therefore yields
\begin{equation}\label{eq:Cuspbign}
\begin{array}{rcl}
\displaystyle\sum_{n=X^{1/12-\theta}}^{X^{1/12+\delta}}\cI^{(2)}(\phi(n;\cdot),Y;\kappa)
&\ll_\epsilon & X^{2/3+\epsilon}+
\displaystyle X^{5/6-\kappa/3+\epsilon}\sum_{n=X^{1/12-\theta}}^{X^{1/12}}\nu(\chi_{n^2})
\\[.15in]
&&\displaystyle +X^{3/4-\kappa/3+\epsilon}\sum_{n=X^{1/12}}^{X^{1/12+\delta}}n^2\nu(\chi_{n^2})
\\[.2in]
&\ll_\epsilon& X^{2/3+\epsilon}+X^{3/4+\theta-\kappa/3+\epsilon}+X^{3/4+\delta-\kappa/3+\epsilon}.
\end{array}
\end{equation}
Above, we estimate the sums over $n$ of $\nu(\chi_{n^2})$ and $n^2\nu(\chi_{n^2})$ just as in \eqref{eq:tempnsum1} and \eqref{eq:tempnsum2}, respectively. Moreover, the contribution of $X^{2/3}$ in \eqref{eq:tempstep5} comes from Lemma \ref{lem:b0bound}, which bounds the number of relevant forms $f(x,y)$ with $b(f)=0$.
When summed over $n$, each such form is counted at most $O_\epsilon(X^\epsilon)$ times, with a weight bounded by $1$. Thus the sum over of the error term from Lemma \ref{lem:b0bound} can be bounded by $O_\epsilon(X^{2/3+\epsilon})$ as stated.

\medskip
\medskip

\noindent {\bf Step 6: Estimates for the cusp. Part 2: $n$ is small.} Let $n\leq X^{1/12-\theta}$ be a  fixed positive integer.
We use \eqref{eq:larget} to write $\cI^{(2)}(\phi(n;\cdot),Y;\kappa)$ as a sum of two main terms along with an error that is $\ll_\epsilon$
\begin{equation}\label{eq:cuspesttemp2}
X^{2/3+\epsilon}+\frac{X^{2/3}}{n^6}\sum_{0\neq|a|\ll Y^{1-4\kappa}}\frac{\supp(\phi(n;\cdot),a)}{|a|}\ll_\epsilon X^{2/3+\epsilon}+\frac{X^{2/3}}{n^6}\sum_{0\neq|a|\ll Y^{1-4\kappa}}\frac{\supp(\chi_{n^2},a)}{|a|}.
\end{equation}
Let $q=p^k$ be a prime power. We know that $\supp(\chi_{q^2},a)=\supp(\chi_{q^2},b)$ if the valuations of $a$ and $b$ at $p$ are the same. It follows that if $p^\ell\parallel a$ for $\ell\leq 2k$, then we have
\begin{equation*}
\supp(\chi_{q^2},a)\ll \frac{\supp(\chi_{q^2})}{q^{2k-\ell}}=\frac{\supp(\chi_{q^2})}{q^{2k}}\cdot (a,q^2),
\end{equation*}
which, by the Chinese remainder theorem, implies that we have
\begin{equation*}
\supp(\chi_{n^2},a)\ll_\epsilon \frac{\supp(\chi_{n^2})}{n^{2-\epsilon}}(a,n^2).
\end{equation*}
Hence, we have
\begin{equation*}
\begin{array}{rcl}
\displaystyle\sum_{0\neq |a|\ll Z}\frac{\supp(\chi_{n^2},a)}{|a|}
&\ll&\displaystyle
\sum_{d\mid n^2}\sum_{\substack{0\neq |a|\ll Z\\(a,n^2)=d}}\frac{\supp(\chi_{n^2},a)}{|a|}
\\[.25in]&\ll_\epsilon&\displaystyle
\frac{\supp(\chi_{n^2})}{n^{2-\epsilon}}\sum_{d\mid n^2}\sum_{\substack{0\neq |a|\ll Z\\d\mid a}}\frac{d}{|a|}
\\[.25in]&\ll_\epsilon&\displaystyle
\frac{\supp(\chi_{n^2})}{n^{2-\epsilon}}Z^\epsilon.
\end{array}
\end{equation*}
Therefore, the error term in \eqref{eq:cuspesttemp2}, summed up over $n< X^{1/12-\theta}$, is $\ll_\epsilon$
\begin{equation*}
X^{3/4-\theta+\epsilon}+X^{2/3+\epsilon}\sum_{n< X^{1/12-\theta}}m_3^{1/2}\ll X^{3/4-\theta+\epsilon}.
\end{equation*}
Optimizing, we pick $\theta=\kappa/6$. Combining this with \eqref{eq:Cuspbign} yields
\begin{equation}\label{eq:CuspFinalEst}
\begin{array}{rcl}
\cI^{(2)}(\phi,Y;\kappa)&=&
\displaystyle\nu(\phi)\Vol(S^{(i)})X^{5/6}\int_{t>0}\psi_2\Bigl(\frac{t}{Y^\kappa}\Bigr)t^{-2}d^\times t
\\[.2in]
&&\displaystyle +
\frac{1}{4}\bigl(D^+(\phi,1/2)\V^+(S^{(i)})+D^-(\phi,1/2)\V^-(S^{(i)})\bigr)
\\[.2in]
&&\displaystyle +
O_\epsilon(X^{3/4-\kappa/6+\epsilon}+X^{3/4+\delta-\kappa/3+\epsilon}).
\end{array}
\end{equation}

\medskip
\medskip 

\noindent {\bf Step 7: Putting it all together.}
Combining \eqref{eq:Mtailcut1}, \eqref{eq:phiFinalMainBall}, and \eqref{eq:CuspFinalEst}, we obtain
\begin{equation*}
\begin{array}{rcl}
N^{(i)}(\phi,X)&=&
\displaystyle\frac{1}{\sigma_i\Vol(G_0)}
\Bigl(\nu(\phi)\Vol(\FF)\Vol(S^{(i))})X^{5/6}+\frac{1}{4}\sum_{\circ=\pm}D^\circ(\phi,1/2)\V^\circ(S^{(i)})X^{3/4}\Bigr)
\\[.2in]
&&\displaystyle +
O_\epsilon(X^{3/4-\delta+\epsilon}+X^{3/4-1/576+83\delta/16+5\kappa/72+\epsilon}+X^{3/4-\kappa/6+\epsilon}+X^{3/4+\delta-\kappa/3+\epsilon}).
\end{array}
\end{equation*}
Optimizing, we pick $\kappa=6\delta$ and $\delta=1/3804$ to bound the error term by $O_\epsilon(X^{3/4-1/3804+\epsilon})$, thereby obtaining the result.
\end{proof}

\subsection{Computing the primary and secondary terms}

We begin with the following Jacobian change of variables formula proved in \cite[Proposition 3.11]{BS2Sel}.
Let $\omega$ be a generator of the rank-$1$ module of top degree differentials of $\PGL_2$ over $\Z$.
Let $dv$ be Euclidean measure on $V$. 
Then we have 
\begin{proposition}[\cite{BS2Sel}]\label{prop:Jac}
Let $F$ be $\R$, $\C$, or $\Q_p$ for some prime $p$. Let $R$ be an open subset of $F\times F$, and let $s:R\to V(F)$ be a continuous function such that the invariants of $s_{I,J}:=s(I,J)$ are $I$ and $J$. Then for any measurable function $\phi:V(F)\to\R$, we have
\begin{equation*}
\int_{v\in \PGL_2(F)\cdot s(R)}\phi(v)dv
=\Bigl|\frac{1}{27}\Bigr|\int_R\int_{\PGL_2(F)}\phi(g\cdot s_{I,J})\omega(g)dIdJ,
\end{equation*}
where we regard $\PGL_2(F)\cdot s(R)$ as a multiset and $|\cdot|$ denotes the absolute value of elements in $F$.
\end{proposition}

Next, we evaluate the quantities $\V^\pm(S^{(i)})$ defined in \S4 by
\begin{equation*}
\mathcal{V}^\pm(S^{(i)}):=
\int_{\pm s>0}
\frac{\Vol(S^{(i)}|_s)}{|s|^{1/2}}ds,
\end{equation*}
where $S^{(i)}=S_1^{(i)}=G_0R_1^{(i)}$.
To this end, we have the following lemma.
\begin{lemma}\label{lem:hyperellipinfint}
Let $F$ be $\R$ or $\Q_p$ for some prime $p$. Let $f\in V(K)\backslash\{\Delta=0\}$ be an $F$-soluble binary quartic form with invariants $I(f)=I$ and $J(f)=J$, and denote $E^{IJ}$ by $E$. Then we have
\begin{equation*}
\int_{(x,z)\in C_f(F)}\frac{dx}{|z|}=\int_{(x,y)\in E(F)}\frac{dx}{|y|},
\end{equation*}
where $C_f:z^2=f(x,1)$ is the genus-$1$ curve corresponding to $f$.
\end{lemma}
\begin{proof}
By assumption, $f$ is soluble over $F$, which implies $C_f(F)$ is nonempty and hence contains some point $Q$. Since $C_f$ is a trivial principal homogeneous space for $E$ over $F$, it follows that $E(F)$ has a simple transitive action on $C_f(F)$. The map $\phi_Q:E(F)\to C_f(F)$, sending $P$ to $P(Q)$ is a bijection. For the purpose of proving the lemma, it is enough to show that the Jacobian of $\phi_Q$ (with respect to the measures $dx/z$ and $dx/y$ on $C_f(F)$ and $E(F)$, respectively), is~equal to $1$.

Let $\overline{\phi}_Q:E(\overline{F})\to C_f(\overline{F})$ be the map sending $P$ to $P(Q)$, where $\overline{F}$ is the algebraic closure of $F$. The Jacobians of $\phi_Q$ and $\overline{\phi}_Q$ are clearly equal. We now prove that the Jacobian of $\overline{\phi}_Q$ is $1$, as follows: first, by replacing $f$ by a $\PGL_2(\overline{F})$-translate, if necessary, we may assume that $f(x,y)=x^3y-I/3xy^3-J/27y^4$. In this case, the curve $C_f(\overline{F})$ can be naturally identified with $E(\overline{F})$ by sending $(x,z)$ to $(x,y)$. Under this identification, we simply have $\overline{\phi}_Q(P)=P+Q$, which has Jacobian $1$ as necessary. The lemma follows.
\end{proof}

Next, we prove the following result.
\begin{proposition}
For $i\in\{0,1,2+,2-\}$, we have
\begin{equation}\label{eq:scminfvolcomp}
\frac{\V^+(S^{(i)})+\V^-(S^{(i)})}{\Vol(G_0)}
=\frac{C_{3/4}^{\circ}}{27\pi},
\end{equation}
where we take $\circ$ to be $\Delta>0$ when $i\in\{0,2+,2-\}$ and $\circ$ to be $\Delta<0$ when $i=1$.
\end{proposition}
\begin{proof}
We start with writing
\begin{equation}\label{eq:secinfvoltemp}
\begin{array}{rcl}
\displaystyle \frac{\V^+(S^{(i)})+\V^-(S^{(i)})}{\Vol(G_0)}
&=&\displaystyle
\frac{1}{\Vol(G_0)}\int_{G_0\cdot R^{(i)}_1}\frac{df}{\sqrt{|a(f)|}}
\\[.2in]&=&\displaystyle
\frac1{27}\int_{f_{IJ}\in R^{(i)}_1}\int_{\gamma\in G_0}
\frac{1}{\sqrt{|a(\gamma\cdot f_{IJ})|}}d\gamma dIdJ
\\[.2in]&=&\displaystyle
\frac1{27}\int_{f_{IJ}\in R^{(i)}_1}\int_{\theta\in K}
\frac{1}{\sqrt{|a(\theta\cdot f_{IJ})|}}d\theta dIdJ,
\end{array}
\end{equation}
where the first equality follows from the definition of $S^{(i)}=G_0\cdot R_1^{(i)}$, the second equality from a direct application of Proposition \ref{prop:Jac}, and the third from the fact that the left hand side of \eqref{eq:scminfvolcomp} is independent of the right $K$-invariant set $G_0$; this last fact is deduced from the fact that the leading constants of the two terms in the right hand side of \eqref{eq:Shintanifinal} must be independent of $G_0$. Above, as in \S4, the measure $d\theta$ is Haar-measure on $K=\SO_2(\R)$ normalized to have volume $1$. For $f\in V(\R)\backslash\{\Delta=0\}$ with invariants $I$ and $J$, we compute the innermost integral above to be
\begin{equation*}
\begin{array}{rcl}
\displaystyle\int_{\theta\in K}\frac{1}{\sqrt{|a(\theta\cdot f)|}}d\theta&=&\displaystyle\frac{1}{2\pi}\int_{\theta=0}^{2\pi}\frac{1}{\sqrt{|f(\cos(\theta),\sin(\theta))|}}d\theta
\\[.2in]&=&\displaystyle
\frac{1}{2\pi}\int_{\theta=0}^{2\pi}\frac{1}{\sin^2(\theta)\sqrt{|f({\rm cotan}(\theta),1)|}}d\theta
\\[.2in]&=&\displaystyle
\frac{1}{\pi}\int_{x=-\infty}^{\infty}\frac{1}{\sqrt{|f(x,1)|}}dx,
\end{array}
\end{equation*}
using the change of variables $x={\rm cotan}(\theta)$. 
Now the set $\{(x,\sqrt{|f(x,1)|}):x\in\R\}$ parametrizes precisely half of the real points on the two hypperelliptic curves $C_f$ and $C_{-f}$ (the other half being parametrized by the set $\{(x,-\sqrt{|f(x,1)|}):x\in\R\}$). Thus, from Lemma \ref{lem:hyperellipinfint} and the definition of $\widetilde{\Omega}(I,J)$ in~\eqref{eq:omegatildeded}, we have
\begin{equation*}
\begin{array}{rcl}
\displaystyle\int_{\theta\in K}\frac{1}{\sqrt{|a(\theta\cdot f)|}}d\theta
&=&\displaystyle
\frac{1}{\pi}\int_{\substack{(x,z)\in C_f\\z>0}}\frac{dx}{z}+\frac{1}{\pi}\int_{\substack{(x,z)\in C_{(-f)}\\z>0}}\frac{dx}{z}
\\[.2in]&=&\displaystyle
\frac{1}{\pi}\int_{\substack{(x,y)\in E^{I,J}\\y>0}}\frac{dx}{y}+
\frac{1}{\pi}\int_{\substack{(x,y)\in E^{I,-J}\\y>0}}\frac{dx}{y}
\\[.2in]&=&\displaystyle
\frac{1}{\pi}\widetilde\Omega(I,J).
\end{array}
\end{equation*}
The result now follows immediately from \eqref{eq:secinfvoltemp} and the definition of $C_{3/4}^{\circ}$ in \eqref{eq:mainconstdeg}.
\end{proof}

It is now possible to compute the values of both $\V^+(S^{(i)})$ and $\V^-(S^{(i)})$:
\begin{cor}\label{cor:vvsi}
We have $\V^\pm(S^{(i)})=c^{\pm,i}\Vol(G_0)C^\circ_{3/4}/(27\pi)$, where $c^{\pm,0}=c^{\pm,1}=1/2$, $c^{\pm,2\pm}=1$, $c^{\pm,2\mp}=0$, and where we take $\circ$ to be $\Delta>0$ when $i\in\{0,2+,2-\}$ and $\circ$ to be $\Delta<0$ when $i=1$.
\end{cor}
\begin{proof}
This result is an immediate consequence of Propoposition along with the following identities:
\begin{equation*}
\V^\pm(S^{(2\mp)})=0;\quad
\V^+(S^{(0)})=\V^-(S^{(0)});\quad
\V^+(S^{(1)})=\V^-(S^{(1)}).
\end{equation*}
The first identity is immediate since every element in $R_1^{(2+)}$ (resp.\ $R_1^{(2-)}$), and hence every element in $S^{(2+)}=G_0\cdot R_1^{(2+)}$ (resp.\ $S^{(2-)}=G_0\cdot R_1^{(2-)}$), is positive (resp.\ negative) definite. The second identity can be deduced as follows: for $i\in\{0,1\}$, the set $\{-f:f\in R^{(i)}\}$ is also a fundamental set for the action of $\PGL_2(\R)$ on $V(\R)^{(i)}$. Hence, as in \eqref{eq:secinfvoltemp}, we have
\begin{equation*}
\V^+(S^{(i)})+\V^-(S^{(i)})=\int_{G_0\cdot R_1^{(i)}}\frac{df}{\sqrt{|a(f)|}}=\int_{G_0\cdot (-R_1^{(i)})}\frac{df}{\sqrt{|a(f)|}},
\end{equation*}
with the contribution from either integral to $\V^\pm(S^{(i)})$ coming from forms $f$ with $\pm a(f)>0$. Since $a(-f)=-a(f)$, the result follows.
\end{proof}

We may now compute the values of $M_{5/6}^{(i)}(\phi)$ and $M_{3/4}^{(i)}(\phi)$ for large and locally well approximated functions $\phi$.
\begin{proposition}\label{prop:Mscomp}
Let $i\in\{0,1,2+,2-\}$, let $j\in\{0,1\}$ and let $\phi$ be a large and locally well approximated function. Then we have
\begin{equation*}
\begin{array}{rcl}
M_{5/6}^{(i)}(\phi)&=&\displaystyle
\frac{2\nu(\phi)\zeta(2)}{27\sigma_i}C^\circ_{5/6};
\\[.2in]
M_{3/4}^{(j)}(\phi)&=&\displaystyle
\frac{D^+(\phi,1/2)+D^-(\phi,1/2)}{216\sigma_j\pi}C_{3/4}^{\circ};
\\[.2in]
M_{3/4}^{(2\pm)}(\phi)&=&\displaystyle
\frac{D^{\pm}(\phi,1/2)}{108\sigma_2\pi}C_{3/4}^{\Delta>0}.
\end{array}
\end{equation*}
As before, we take $\circ$ to be $\Delta>0$ when $i\in\{0,2+,2-\}$ or $j=0$ and to be $\Delta<0$ when $i=1$ or $j=1$.
\end{proposition}
\begin{proof}
The first equality above follows from Proposition \ref{prop:Jac}, while the remaining equalities follow from Corollary \ref{cor:vvsi}.
\end{proof}

Theorem \ref{th:mainCongCountFull} follows from Theorem \ref{thm:NphiXM} and Proposition \ref{prop:Mscomp}. Proposition \ref{prop:Mscomp} also completes the proof of Theorem \ref{Th:Shintani4} in \S4. For Theorem \ref{th:mainCongCountFin}, the only remaining pieces are the values of $\kappa_{5/6}(\sigma_p)$ and $\kappa_{5/6}(\sigma_p)$ for splitting types $\sigma_p$; these are computed in the appendix.
Finally, we prove the main elliptic curve results. 

\medskip

\noindent{\bf Proof of Theorem \ref{Theorem1:cong}:} 
Let $\phi:\Z^2\to\R$ be large and locally approximated. From Theorem \ref{th:ellipselpar} and Lemmas \ref{lem:red1} and \ref{lem:red2}, it follows that we have
\begin{equation*}
\begin{array}{rcl}
\displaystyle\sum_{E_{AB}\in\EE(X)^+}(|\Sel_2(E_{AB})|-1)\phi(A,B)&=& N^{(0)}(\psi,2^6X)+N^{(2+)}(\psi,2^6X)+O_\epsilon(X^{2/3+\epsilon}),
\\[.1in]
\displaystyle\sum_{E_{AB}\in\EE(X)^-}(|\Sel_2(E_{AB})|-1)\phi(A,B)&=& N^{(1)}(\psi,2^6X)+O_\epsilon(X^{2/3+\epsilon}),
\end{array}
\end{equation*}
where $\psi\colon V(\Z)\rightarrow \R$ is defined by $\psi(f)=\phi(A(f),B(f))\ell(f)/m(f)$. By Lemma \ref{lem:lawainv}, it follows that $\psi$ is large and well approximated, and so we may apply Theorem \ref{thm:NphiXM} to write the left hand sides of the above equation as a sum of two main terms along up to a sufficiently small error. Finally, the fact that the main terms arising from Theorem \ref{thm:NphiXM} align with the main term claimed in Theorem \ref{Theorem1:cong} follows from \cite[Theorem 3.1]{BS2Sel} by simply approximating each $\phi_p$ by a finite sum of characteristic functions. The result follows. $\Box$

\medskip

Theorem \ref{Theorem1} is an immediate consequence of Theorem \ref{Theorem1:cong}. 

\appendix

\section{Computations of primary and secondary local densities}

Let $\sigma_p$ be a splitting type modulo $p$. We begin by computing the primary density $\kappa_{5/6}(\sigma_p)$ for primes $p\geq 3$.
\begin{lemma}
Let $p\geq 3$ be a prime. The values of $\kappa_{5/6}(\sigma_p)$ are as given in Table \ref{tab:primaryden}.
\end{lemma}

\begin{table}[ht]
\centering
\begin{tabular}{|c | l|}
\hline
$\sigma_p$ & 
$\kappa_{5/6}(\sigma_p)$ 
\\[5pt]
\hline\hline
$(1111)$ & 
$(p+1)(p-1)^2(p-2)/(24p^4)$ 
\\[10pt]
$(112)$ & 
$(p+1)(p-1)^2/(4p^3)$ 
\\[10pt]
$(13)$ & 
$(p+1)^2(p-1)^2/(3p^4)$ 
\\[10pt]
$(22)$ &
$(p-1)^2(p+1)(p-2)/(8p^4)$ 
\\[10pt]
$(4)$ & 
$(p+1)(p-1)^2/(4p^3)$ 
\\[10pt]
\hline
$(1^211)$ & 
$(p+1)(p-1)^2/(2p^4)$ 
\\[10pt]
$(1^22)$ & 
$(p+1)(p-1)^2/(2p^4)$ 
\\[10pt]
$(1^31)$ & 
$(p+1)(p-1)/p^4$ 
\\[10pt]
\hline
$(1^21^2)$ & 
$(p+1)(p-1)/(2p^4)$
\\[10pt]
$(2^2)$ & 
$(p-1^2)/(2p^4)$
\\[10pt]
$(1^4)$ & 
$(p+1)(p-1)/p^5$
\\[10pt]
\hline
\end{tabular}
\caption{Splitting type densities in $V(\Z_p)$ for $p\geq 3$}\label{tab:primaryden}
\end{table}
\begin{proof}
To compute the density of integral forms with some fixed splitting type $\sigma_p$, note that this is the same as the density of elements in $V(\F_p)$ with splitting type $\sigma_p$, and that this latter density can be computed via a simple counting argument. For example, an element in $V(\F_p)$ with splitting type $(1111)$ is determined by four distinct points in $\P^1_{\F_p}$, up to multiplication by an element in $\F_p^\times$. There are ${p+1}\choose{4}$ such sets of four points, and multiplying by $(p-1)/p^5$ yields the density. The cases of the other splitting types are similar.
\end{proof}

Next, let $\phi:V(\Z)\to\R$ be a large and locally well approximated function via the functions $\phi_p:V(\Z_p)\to\R$. We proved in Corollary \ref{cor:anacont} that $D^\pm(\phi,s)$ has an analytic continuation to the right of $\Re(s)=1/3$. Assume that $\phi_p$ is invariant under multiplication by units in $\Z_p$, i.e., $\phi_p(f)=\phi_p(uf)$ for $u\in\Z_p^\times$. Then $D^\pm(\phi,s)$ has an Euler product expansion:
\begin{equation*}
\begin{array}{rcl}
&&\displaystyle D^\pm(\phi,s)=
\displaystyle\sum_{a>0}\frac{\nu_{\pm a}(\phi)}{a^s}
=\sum_{a>0}\prod_{p^k\parallel a}\frac{\nu_{p^k}(\phi_p)}{p^{ks}}=\prod_p D_p(\phi_p,s),
\\[.2in]\mbox{where}&&
\displaystyle D_p(\phi_p,s):=
\nu_1(\phi_p)+\frac{\nu_p(\phi_p)}{p^s}
+\frac{\nu_{p^2}(\phi_p)}{p^{2s}}+\cdots
=\sum_{k\geq 0}\frac{\nu_{p^k}(\phi_p)}{p^{ks}}.
\end{array}
\end{equation*}
We now compute these densities $\nu_{p^k}(\phi_p)$ for certain functions $\phi_p$.

\begin{proposition}
For a splitting type $\sigma$, let $\chi_\sigma:V(\Z_p)\to\{0,1\}$ denote the characteristic function of the set of elements in $V(\Z_p)$ having splitting type $\sigma$. The values of $\nu_{p^k}(\chi_\sigma)$ are as given in Table~\ref{tablenu}.
\end{proposition}

\begin{table}[ht]
\centering
\begin{tabular}{|c | c| c| c|}
  \hline
  &&&\\[-8pt]
  $\phi$&$\nu_1(\phi)$&$\nu_p(\phi)$&$\nu_{p^k}(\phi)$, $k\geq 2$\\[10pt]
  \hline\hline
  &&&\\[-8pt]
  $\chi_{(1111)}$ & 
  $(p-1)(p-2)(p-3)/(24p^3)$ & 
  $(p-1)^2(p-2)/(6p^3)$ & 
  $(p-1)^2(p-2)/(6p^3)$
  \\[10pt]
  
  $\chi_{(112)}$ &
  $(p-1)^2/(4p^2)$ &
  $(p-1)^2/(2p^2)$ &
  $(p-1)^2/(2p^2)$
  \\[10pt]
  
  $\chi_{(13)}$ &
  $(p+1)(p-1)/(3p^2)$ &
  $(p+1)(p-1)^2/(3p^3)$ &
  $(p+1)(p-1)^2/(3p^3)$ 
  \\[10pt]
  
  $\chi_{(22)}$ &
  $(p-1)(p^2-p-2)/(8p^3)$&
  $0$ &
  $0$ 
  \\[10pt]
  
  $\chi_{(4)}$ &
  $(p+1)(p-1)/(4p^2)$&
  $0$ & 
  $0$ 
  \\[10pt]
  
  \hline
  $\chi_{(1^211)}$ &
  $(p-1)(p-2)/(2p^3)$ &
  $3(p-1)^2/(2p^3)$ &
  $3(p-1)^2/(2p^3)$ 
  \\[10pt]
  
$\chi_{(1^22)}$ &
$(p-1)/(2p^2)$&
$(p-1)^2/(2p^3)$&
$(p-1)^2/(2p^3)$
\\[10pt]

$\chi_{(1^31)}$ &
$(p-1)/p^3$&
$2(p-1)/p^3$&
$2(p-1)/p^3$\\[10pt]

\hline
  $\chi_{(1^211)}^\max$ &
  $(p-1)^2(p-2)/(2p^4)$ &
  $(3p-2)(p-1)^2/(2p^4)$ &
  $(p-1)^3/p^4$\\[10pt]

  $\chi_{(1^22)}^\max$ &
$(p-1)^3/(2p^4)$&
$(p-1)^2/(2p^3)$&
0\\[10pt]

$\chi_{(1^31)}^\max$ &
$(p-1)^2/p^4$&
$(2p-1)(p-1)/p^4$&
$(p-1)^2/p^4$\\[10pt]

\hline
$(1^21^2)$ &
$(p-1)/(2p^3)$&
$(p-1)/p^3$ &
$(p-1)/p^3$\\[10pt]
$(2^2)$ &
$(p-1)/(2p^3)$ &
0&
0\\[10pt]
$(1^4)$ &
$1/p^3$&
$(p-1)/p^4$&
$(p-1)/p^4$\\
\hline
\end{tabular}
\caption{Splitting type densities in $V(\Z_p)_a$}\label{tablenu}
\end{table}

\begin{proof}
Since the set of elements in $V(\Z_p)$ with a given splitting type is defined modulo $p$, the densities associated to $\chi_\sigma$ can be computed by counting the relevant elements in $V(\F_p)$. So for example, the number of elements $f(x,y)\in V(\F_p)$ with leading coefficient $1$ (resp.\ leading coefficient $0$) and factoring into four distinct linear factors is equal to $p(p-1)(p-2)(p-3)/24$ (resp.\ $p(p-1)^2(p-2)/6$, leading to the listed values of $\nu_a(\chi_{(1111)})$. The computations for the other unramified splitting types are similar. 

Next consider $\sigma=(1^211)$. An element $f(x,y)\in V(\Z_p)$ has splitting type $\sigma$ if and only if the reduction of $f$ modulo $p$ has three distinct roots in $\P^1_{\F_p}$, one of which (say $r$) is a double root. Moreover it is easy to see that the resolvent of $f$ is maximal if and only if $p^2\nmid f(\tilde{r})$ for a lift $\tilde{r}$ of $r$. As before, the value of $\nu_1(\chi_{(1^211)})$ can be computed via counting roots in $\F_p$ to be $p(p-1)(p-2)/(2p^4)$, as listed in the table. Since each such element in $V(\Z_p)$ is maximal with probability $(p-1)/p$, the listed value of $\nu_1(\chi_{(1^211)}^\max)$ is correct. To compute the values of $\nu_{p^k}(\chi_\sigma)$, we separate into cases depending on whether the root at infinity is the double root or a single root. In the former case, we must count elements in $V(\F_p)$ of the form $\alpha y^2(x-r_1y)(x-r_2y)$, where $\alpha\in\F_p^\times$ and $r_1,r_2\in\F_p$ are distinct. There are clearly $p(p-1)^2/2$ such elements, and such a form lifts to a maximal quartic form in $V(\Z_p)$ if and only if $k=1$. In the latter case, we must count elements in $V(\F_p)$ of the form $\alpha y(x-r_1)(x-r_2)^2$; there are $p(p-1)^2$ such elements, and they lift to maximal elements with probability $1-1/p$. Adding up the contributions from these cases gives the desired values. The computations in the other cases are similar, where we also note that a binary quartic form with splitting type $(1^21^2)$, $(2^2)$, or $(1^4)$ is never strongly maximal.
\end{proof}

\bibliographystyle{abbrv}
\bibliography{references}

\end{document}